\documentclass[12pt]{amsart} 
\usepackage[utf8]{inputenc}
\usepackage{amsmath, amsthm, amscd, amsfonts,amssymb,graphicx,enumerate,tikz-cd}
\numberwithin{equation}{section}
\makeatletter
\@namedef{subjclassname@2010}{%
  \textup{2010} Mathematics Subject Classification}
\makeatother

\newcommand{\teq}{\arabic{section}.\arabic{equation}}

\newcommand{\teql}{\Alph{section}.\arabic{equation}}
\newcommand{\sqr}[2]{{\vcenter{\vbox{\hrule height.#2pt\hbox{\vrule width.#2pt
height#1pt \kern#1pt\vrule width.#2pt}\hrule height.#2pt}}}}

\newcommand{\ssquare}{{\qquad\hfill$\square$}}

\newcounter{eqcount}
\newcounter{ttopic}

\newenvironment{edesc}{\refstepcounter{equation}\begin{enumerate}}%
{\end{enumerate}}
\newcommand{\ring}[1]{{\mathbb #1}}
\newcommand\bZ{{\ring{Z}}}
\newcommand\bC{{\ring{C}}} 
\newcommand\bF{{\ring{F}}} \newcommand\bQ{{\ring{Q}}}

\newcommand{\csp}[1]{{\mathbb #1}}
\newcommand{\tsp}[1]{{\mathcal #1}}

\newcommand{\prP}{\csp{P}}

\newcommand{\sC}{{\tsp{C}}} 
\newcommand{\sO}{{\tsp{O}}} 

\newcommand{\sP}{{\tsp {P}}} 
 
\newcommand{\sT}{{\tsp {T}}} \newcommand{\sH}{{\tsp {H}}}
 
\newcommand{\sM}{{\tsp {M}}} 
\newcommand{\sG}{{\tsp {G}}} 
\newcommand{\sR}{{\tsp {R}}} 

\newcommand{\bT}{{\csp {T}}} 

\newcommand{\eql}[2]{{\rm (\ref{#1}\ref{#2})}} %ref. to long equation

\newcommand{\vect}[1]{{\pmb #1}} 
 \newcommand{\bg}{\vect{g}}
 
\newcommand{\bp}{{\vect{p}}} 
\newcommand{\bv}{{\vect{v}}} 
 \newcommand{\bz}{{\vect{z}}}
\newcommand{\bh}{{\vect{h}}}  
\newcommand{\row}[2]{{#1_1,\ldots,#1_{#2}}}

\newcommand{\smatrix}[4]{{\big(\begin{array}{cc}
\!\lower2pt\hbox{$\scriptstyle#1$} &\lower2pt\hbox{$\scriptstyle#2$}\!
\\\! \raise2pt\hbox{$\scriptstyle#3$} &\raise2pt\hbox{$\scriptstyle#4$}
\!\end{array}\big)}}
\newcommand{\col}[2]{{\big(\begin{array}{c}
\!\lower2pt\hbox{$\scriptstyle#1$}  \!
\\\! \raise2pt\hbox{$\scriptstyle#2$}
\!\end{array}\big)}}

\newcommand{\texto}[1]{{\textr{#1}}}
\newcommand{\GL}{\texto{GL}} 
 \newcommand{\ind}{\texto{ind}}
\newcommand{\PSL}{\texto{PSL}} \newcommand{\PGL}{\texto{PGL}}
 
 \renewcommand{\ni}{\texto{Ni}}
 
\newcommand{\textr}[1]{{\text{\rm #1}}}
\newcommand{\tr}{\textr{tr}} \newcommand{\ord}{\textr{ord}}
\newcommand{\abs}{\textr{abs}}  
 
 \newcommand{\inn}{\textr{in}}
 
\newcommand{\pr}{\textr{pr}}
\newcommand{\RET}{{\text{\rm RET}}}

\newcommand{\textb}[1]{{\text{\bf #1}}}
\newcommand{\bfC}{{\textb{C}}}
\newcommand{\longmapright}[2]{\smash{\mathop{\hbox to
#2pt{\rightarrowfill}}\limits^{#1}}}
\newcommand{\Longmapright}[2]{\smash{\mathop{\hbox to
#2pt{\Rightarrowfill}}\limits^{#1}}}
\newcommand{\longmapleft}[2]{\smash{\mathop{\hbox to
#2pt{\leftarrowfill}}\limits^{#1}}}
\newcommand{\mapdown}[1]{\Big\downarrow\rlap{$\vcenter{\hbox{$\scriptstyle#1$}}
$}}

\newcommand{\np}{{+}}   \newcommand{\nm}{{-}}

\newcommand{\lrang}[1]{{\langle #1\rangle}}

\newcommand{\eqdef}{\stackrel{\text{\rm def}}{=}}

\newfont{\sevenrm}{cmr7}
\newfont{\bsevenrm}{cmbx7}
\newfont{\mathseven}{cmsy7}
\newfont{\bigmath}{cmsy10 scaled 1200}
\newfont{\fiverm}{cmr5}
\newfont{\bfiverm}{cmbx5}
\newfont{\hel}{cmbx10 scaled 1400}
\newfont{\eu}{eufb10}
\newfont{\sseu}{eufm5}
\newfont{\seu}{eufm7}
\newfont{\Cal}{cmmib10}
\newfont{\sCal}{cmmib7}
\newfont{\zch}{eusb10}

\theoremstyle{plain}
\newtheorem{thm}{Theorem}[section] %numbering for subunit in last position
\newtheorem{lem}[thm]{Lemma}
\newtheorem{princ}[thm]{Principle}

\newtheorem{prop}[thm]{Proposition}
\newtheorem{cor}[thm]{Corollary}

\theoremstyle{definition}
\newtheorem{defn}[thm]{Definition}
\newtheorem{exmp}[thm]{Example}

\theoremstyle{remark}
\newtheorem{rem}[thm]{Remark}
\newcommand{\xs}{\times^s\!}
\def\pic #1 by #2 (#3){\vbox to #2{\hrule width 
#1 height 0pt depth 0pt\vfill\special{picture #3}}}
\def\scaledpicture#1
by #2 (#3 scaled #4){{\dimen0=#1 \dimen1=#2\divide\dimen0 by
 1000 \multiply\dimen0 by #4\divide\dimen1 by 1000 \multiply\dimen1 
by #4\pic \dimen0 by \dimen1 (#3 scaled #4)}}
\newenvironment{exmpl}{\begin{exmp}}{\hfill $\triangle$ \end{exmp}}
\newcommand{\mx}{{\scriptstyle  mx}}
 
\newcommand{\comm}[1]{{}}
\renewcommand{\RET}{{\rm RET}}
\newcommand{\CLC}{{\rm CLC}}

\renewcommand{\phi}{\varphi}
\newcommand{\Fr}{\text{Fr}}

\newcommand{\sW}{\tsp W}

\newcommand{\C}{{\text{\rm C}}}

\newcommand{\RH}{\text{\rm RH}}

\newcommand{\psigma}{{\pmb \sigma}}
\newcommand{\ptau}{{\pmb \tau}}

\newcommand{\geng}{{{\text{\bf g}}}}
\newcommand{\sh}{{{\text{\bf sh}}}}

\newcommand{\ochar}{{{\text{\rm {\bf o}-char}}}}

\begin{document}
\baselineskip=17pt

\title[Irreducibility Hypothesis]{\small{Taming Genus 0 (or 1) components on \\ variables-separated equations}}

\author[M.~D.~Fried]{Michael
D.~Fried}
\address{Emeritus, UC Irvine \\ 3547 Prestwick Rd, Billings MT 59101}
\email{mfried@math.uci.edu}
%\author[I.~Gusi\'c]{Ivica~Gusi\'c}
%\address{faculty -- FKIT, Univ. of Zagreb  \\ Marulicev trg 19, Zagreb, Croatia}
%\email{igusic@fkit.hr}

\date{}

\begin{abstract} To describe curves of form  $\sC_{f,g}\eqdef \{(x,y)| f(x) - g(y)=0\}$ and their number theory properties, you must address $\sC_{f,g}$ whose projective normalization has a genus 0 (or 1) {\sl component}. For $f$ and $g$ {\sl polynomials\/} and $f$ indecomposable, \cite{Fr73a} distinguished $\sC_{f,g}$ with $u=1$ versus $u>1$ components (Schinzel's problem). For $u=1$,  \cite[(1.6) of Prop.~1]{Fr73b} gave a direct genus formula.  To complete $u>1$ required an adhoc genus computation.  

\cite{Fr12} revisited later work. Pakovich \cite{Pak18b}, an example, dropped the indecomposable and polynomial restrictions, but added $\sC_{f,g}$ is irreducible ($u=1$). He showed -- for fixed $f$ -- unless the Galois closure of the cover for $f$ has genus 0 (or 1), the genus grows linearly in  $\deg(g)$. Cor.~\ref{methodI} and Cor.~\ref{methodII} extend \cite[Prop.~1]{Fr73b} and  use {\sl Nielsen classes\/} to generalize Pakovich's formulation for $u>1$.   

Using the solution to the Davenport and Schinzel problems, {\sl Hurwitz families\/} track the significance of these components, an approach motivated by Riemann's relating $\theta$ functions and half-canonical classes. 
\end{abstract}

% The topics I tend to use from Math Reviews
% 11-XX Number theory
% 12-XX Field theory and polynomials
% 14-XX Algebraic geometry 
% 20-XX Group theory and generalizations 
% 30-XX Functions of a complex variable

% Primary: Fine and Course Moduli Spaces, Primitive Groups, Compact Riemann surfaces and unformization, Galois Theory
% Secondary: Polynomials in real and complex fields: factorization, Field Arithmetic, Extensions Wreath Products and other compositions of groups


\thanks{}
\maketitle
\setcounter{tocdepth}{3}
\tableofcontents

\section{Fiber products of rational function covers} 

We use notation from \cite{Fr12}, quoting from it to relax the hypotheses  \eql{Pak}{Paka}  and \eql{Pak}{Pakb} below. We start with the same language as in \cite{Pak18b} as if referencing one rational function $f$ at a time.  
 Yet, almost always, results depend only on the {\sl Nielsen class\/} (\S\ref{usenc};  \S\ref{cbnc} gives a quick survey) of the cover given by $f$. That allows focusing on discrete data packets that help find such low genus components.  

For $z$ a complex variable, denote the projective line uniformized by the $z$ -- including $z=\infty$ -- by $\prP^1_z$. Consider the set of variables separated equations
 $$\sC_{f,g}\eqdef \{(x,y)| f(x)-g(y)=0\} \text{  with }f,g\in \bC(x)\text{, and $f$ fixed}.$$ Regard $\sC_{f,g}$ as an affine piece of the fiber product $\prP^1_x\times_{\prP^1_z}\prP^1_y$ of the rational function pair $(f,g)$ of respective degrees $m$ and $n$. Denote the {\sl genus\/} of a compact connected Riemann Surface $W$ by $\geng_W$. 
 
The projective normalization, $\tilde{\sC}_{f,g}$,  of $\sC_{f,g}$ is a disjoint union of (say, $u$) compact connected Riemann surfaces -- {\sl components}.\footnote{Most statements on compact Riemann surfaces apply to projective nonsingular curves in any characteristic, except where {\sl branch cycles\/} enter. Despite Grothendieck's theorem on characteristic $p$ fundamental groups, tame ramification is more complicated.} Each has a genus, given by notation like ${}_i\geng_{f,g}$, $i=1,\dots,u$. Each surface inherits canonical maps $\pr_x$ and $\pr_y$, resp., to $\prP^1_x$  and $\prP^1_y$. Excluding a finite set of points, $\tilde{\sC}_{f,g}$ is the fiber product without normalization. Prop.~\ref{compimage} reminds of the universal property of fiber products of covers and a useful way to relate different fiber products. 

For a (ramified) cover, $\phi: X\to Z$, there is a {\sl Galois closure\/} cover, $\hat \phi: \hat X\to Z$. Refer to the Galois group of that cover as the {\sl monodromy group\/} of $\phi$.  The Galois closure of an irreducible cover $\phi$ interprets as the (projective) normalization of a component of a fiber product construction, compatible with our consideration of components of $\tilde \sC_{f,g}$ (\S \ref{gcc}). With $f\in \bC(x)$ fixed, $$\text{consider  }\sR_f=\{g\in \bC(x)| f(x)-g(y) \text{ is  reducible}\}.$$ 
 
\begin{thm}[Pakovich \cite{Pak18b}] \label{PakThm} Assume also: 
\begin{edesc} \label{Pak} \item \label{Paka} the Galois closure of the cover $\prP^1_x \to \prP^1_z$ has genus $> 1$; and  
\item  \label{Pakb} we run over all $g$ not in $\sR_f$. \end{edesc} Then, the genus of $\tilde \sC_{f,g}$ goes to infinity with the degree of $g$. \end{thm} 

Much mathematics features equations of form $\tilde \sC_{f,g}$. For that,  significant cases occur when  $f$ and $g$ have an   {\sl entanglement\/} (as discussed in \S\ref{noclassif}), exposed through using their Galois closures as covers of $\prP^1_z$. Foremost among those is the (nonobvious) reducibility of $\tilde \sC_{f,g}$, as occurred in considering Davenport's and Schinzel's problems and now Pakovich's result. 

\S\ref{irreducibilitya} starts by giving a precise characterization of the reducibility of these fiber products, a characterization advantageous  over dealing directly with equations. At the end of the section, it separates this from finding components of genus 0. It illustrates this with a running example that sets up this reducibility property using these tools: {\sl Nielsen classes\/}, the {\sl  genus 0 problem} and a formula to compute the genus of components of $\tilde \sC_{f,g}$. \S\ref{thelit} then gives the results of the paper, which are roughly these:

\begin{edesc} \label{goals} \item \label{goalsa} Illustrating the parameters for Nielsen classes for finding examples and producing results generalizing those of Pakovich  (\S\ref{usebcs}).\footnote{Example: we would never have discovered the complete description of Davenport's problem if we had relied on the ''look'' of equations.}

\item \label{goalsb} Using a formula -- beyond \cite{Fr73b} -- for detecting and computing genuses of components of reducible $\tilde{\sC}_{f,g}$ in the \eql{goals}{goalsa} families.  
\item \label{goalsc} Finding/describing such components starting from \eql{goals}{goalsb} based on coalescing (\S\ref{coelnicl}) illustrated on running examples.  
\item Generalizing Thm.~\ref{PakThm} using the number of irreducible components as a parameter. Then, characterizing when it changes in following a path of varying $g\,$s.  
\end{edesc} 

The most important difference between \cite{Pak18b} and this paper is that we start with data about the fiber products $\tilde \sC_{f,g}$. With this, we can be specific about components, $W$,  and then about the nature of the projection $\pr_y: W\to \prP^1_y$ appropriate for considering extending Thm.~\ref{PakThm}. It was Pakovich's idea to separate out covers satisfying \eql{Pak}{Paka}, which we also adapted for $\pr_y$. 

Cases with $\tilde \sC_{f,g}$ having a genus 1 component are of interest; \S\ref{bcdeg7}  gives examples. There is one situation in which genus one components must be considered separately, covered in Lem.~\ref{genus1}, which characterizes what must happen with $\pr_y$. 

\subsection{Characterizing what happens when $g\in \sR_f$} \label{irreducibilitya}   In particular, we look at the  $f\,$s failing the hypothesis \eql{Pak}{Paka}. \S\ref{ncnontrans} does an exposition on the fiber product in terms of branch cycles. 

Prop.~\ref{redRed} interprets reducibility of $\tilde \sC_{f,g}$. The  point of \cite{Fr73a} was to profitably relate pairs of covers $\phi_X:X\to \prP^1_z$ and $\phi_Y:Y\to \prP^1_z$ -- for which even if given by rational functions, by using relations between their monodromy (Galois closure) groups, $G_{\phi_X}$ and $ G_{\phi_Y}$. With that, we formulate dropping the irreducibility assumption \eql{Pak}{Paka}. Using Prop.~\ref{redRed}, we can start by fixing a pair $(f^*,g^*)$ having the same Galois closures. Then, address the following questions running over $g^*\circ g_1\in \bC(x)$ with $\deg(g_1)\ge 1$. 
\begin{edesc} \label{Pakvar} \item \label{Pakvara}  When do component genuses on $\tilde \sC_{f^*,g^*\circ g_1}$ go to $\infty$ with  $\deg(g_1)$?
\item \label{Pakvarc} How to avoid   $g_1\,$s for which   $\tilde \sC_{f^*,g^*\circ g_1}$ has genus 0 (or 1) components for $\deg(g_1)$ arbitrarily large, contrary to \eql{Pakvar}{Pakvara}? \end{edesc} 

We are heavy on examples illustrating using the tools.   Thm.~\ref{genPakThmGen} uses this to extend Thm.~\ref{PakThm}.  Thm.~\ref{polyPakThm} is a  model for the best possible explicitness, relating \S\ref{irreducibilityb} to the {\sl genus 0 problem}. Three subsections pave the way to showing how we can be precise about when \eql{Pakvar}{Pakvara} fails. 

\begin{defn} \label{indecomp} A cover $\phi: X \to Z$ is {\sl indecomposable\/} if it does not decompose into covers, $\phi_1:X \to W$ and $\phi_2:W \to Z$, with $\deg(\phi_i)>1$, $i=1,2$. Equivalently,  $\hat \phi$  in its natural $\deg(\phi)$ permutation representation is {\sl primitive\/}\footnote{Assume $\phi$ is irreducible. Then the permutation representation $T_\phi$ is primitive if there is no group properly between $G_\phi(T_\phi,1)$ and $G_\phi$.} \cite[Lem.~2]{Fr70}. \cite[Thm.~4.5]{Fr12} reviews  characterizing this  
\begin{equation}  \label{redconds} \text{when }f,g\in \bC[x] \text { and } f \text{  is indecomposable}.\footnote{An $f\in \bQ(x)$ can be indecomposable over $\bQ$, but decompose over $\bar \bQ$. It was a fundamental to \cite{Fr70}, for $f\in \bQ[x]$, this could not happen. Here we are always over $\bC$.}  \end{equation}   \end{defn} 

\S\ref{compfried-pak} compares Pakovich's and our approaches/goals. Each considers a starting place referencing pairs of rational functions $(f,g)$. Each considers statements as $g$ changes. Pakovich considers $\sC_{f,g}$ by starting from particular pairs $(f,g)$ described by their coefficients, rather than recognizing parameters for putting such pairs in natural families. Our examples illustrate the following Principle. 

\begin{princ}  \label{ncdiscreteness} Results only depend on the {\sl  Nielsen class\/}, \S\ref{usenc}, in which $(f,g)$ falls. Then,  different Nielsen classes reveal distinct phenomena.\footnote{A general phenomenon: Different Nielsen classes define different moduli problems.} \end{princ}

\S\ref{charred} characterizes irreducible components of $\tilde \sC_{f,g}$, a characterization depending only on the Nielsen class of $(f,g)$.   
\S\ref{startdeg7} uses \S\ref{charred} to create, by example, notation for using branch cycles (and Nielsen classes). This allows generalizing Thm.~\ref{PakThm} with examples, without  coefficients for rational functions, consistent with his case. This includes using systems of imprimitivity in decomposing rational functions.  

Following these definitions and motivations, \S\ref{Pakfails} lists general results, illustrating using  Ex.~\ref{deg7covergenus} as the start of a running example. 

\subsubsection{Our approaches vs \cite{Pak18b} and \cite{Pak22}} \label{compfried-pak} \S\ref{cbnc} gives an overview on using branch cycles (for covers of $\phi_X: X\to \prP^1_z$). This effective computational tool immediately gives the genus, $\geng_X$, of $X$  from Riemann-Hurwitz (\RH, \eqref{rh}), when $X$ is {\sl irreducible}. We also discuss the {\sl orbifold character\/}, $\ochar$, used by Pakovich. 

\S\ref{usenc} reminds of the (Hurwitz) space of covers in a given  Nielsen class (Princ.~\ref{nielsenClass}) generalizes the moduli space of curves of genus $\textbf{\bf g}$.  There are several types. The most commonly used are {\sl absolute\/} spaces. These come with a group $G$, conjugacy classes $\bfC$  satisfying the properties in \eqref{RET}, and a permutation representation, $T: G\to S_n$, attached to the degree $n$ covers,  $\phi:X\to \prP^1_z$, parametrized in the family. The notation is $\ni(G,\bfC,T)$. 

Assuming $\phi$ is irreducible ($T$ is transitive), you immediately compute the genus of the cover from a representing element $\psigma\in \ni(G,\bfC,T)$ according to {\sl Riemann-Hurwitz\/} (\RH) as in Ex.~\ref{deg7covergenus}.\footnote{Another notation: $\psigma\in \bfC$ -- in some order the entries of $\psigma$ fall in the classes $\bfC$.} The {\sl cycle type\/} comes by regarding the elements of $\bfC$ as in $S_{\deg(\phi)}$.

Define $\sG_{f,g}$ to be $(f^*,g^*)$ for which:  
\begin{edesc}\label{factor}  \item \label{factora}  $f$ (resp.~$g$) factors through $f^*: \prP^1_{x^*}\to \prP^1_z$ (resp.~$\prP^1_{y^*}\to \prP^1_z$) ; and 
\item  \label{factorb} $\hat f^*: \hat \prP^1_{x^*}\to \prP^1_z$ and $\hat g^*:\hat\prP^1_{y^*}\to \prP^1_z$,  are  equivalent Galois covers. \end{edesc} Denote the common Galois group in \eql{factor}{factorb} by $G_{f^*,g^*}$. 

For Pakovich's problem, we need to extend Nielsen classes to take account of the involvement of two covers $(f^*,g^*)\in \sG_{f,g}$ as entwined by Prop.~\ref{redRed}. The notation we use here is $\ni(G,\bfC,\bT)$, with $\bT=(T_1,T_2)$ two representations of $G$ as in \eqref{bTnielsen}. That uses {\sl inner\/} Hurwitz spaces and the monodromy group of the Galois closure of a cover.

For some groups, several conjugacy classes have the same cycle type. 
\begin{exmpl}[$A_n$ example]  \label{Anex} For $n\ge 3$ odd, $n$-cycles in $S_n$ are in $A_n$. There are two conjugacy classes, $\C_1,\C_2$,  of such. They are conjugate by an outer automorphism from $S_n$. 
That is, by the normalizer, $N_{S_n}(G)$, of $G$ in $S_n$. 

Use Def.~\ref{normbfC}: If $\bfC$ consists of $\C_i$ with multiplicity $m_i$, $i=1,2$, then $N_{S_n}(G,\bfC)=S_n$ if and only if $m_1=m_2$. \end{exmpl}

\begin{exmpl}[Projective linear groups] \label{PGLex} Take $\bF_q$ the finite field of order a prime-power $q=p^t$. The groups $G=\PGL_d(\bF_{q})$, used in our examples, have permutation representations of degree $n$ with several classes of $n$-cycles. Here, though, what happens in Ex.~\ref{Anex} doesn't hold. 

For example, with $d=2$, take the permutation representation $T_1$ to be the action on $n=\frac{p^{2\np1}\nm1}{p\nm1}$ points and $T_2$  the action on lines of projective $2$-space, $\prP^2(\bF_p)$.  There are pairs of $n$-cycles not conjugate by $N_{S_n}(G)$ \cite[Lem.~5]{Fr73a} and a more general proof \cite[\S4.2]{Fr12}. The cycle types for the permutation representations $T_1$ and $T_2$ are the same,  though the permutation representations are different.\end{exmpl}  

\begin{exmpl} [Genus of a cover from branch cycles]  \label{deg7covergenus} 
\RH\ computes the genus of covers given by the branch cycles called ${}_1\psigma$ for the degree 7 example $f$ in \S\ref{bcdeg7} as discussed in \S\ref{charred}.   
$$\begin{array}{c}{}_1\sigma_1=(x_1\,x_3)(x_4\,x_5), {}_1\sigma_2=(x_1\,x_4\,x_6\,x_7)(x_2\,x_3), {}_1\sigma_3=(x_1\,x_2\,x_3\,x_4\,x_5\,x_6\,x_7)^{-1} \\
2(7 \np \geng_f \nm 1)= \ind({}_1\sigma_1)\np  \ind({}_1\sigma_2)\np \ind({}_1\sigma_3)=2\np 4\np 6 \implies \geng_f = 0.\end{array}$$
What you would expect from a cover given by a rational function. Notice: ${}_1\sigma_1\cdot {}_1\sigma_2\cdot {}_1 \sigma_3=1$, the product-one condition of \eqref{RET}. 
\end{exmpl} 

A Nielsen class, with its accompanying Hurwitz monodromy action, gathers a family of covers into an algebraic variety (Hurwitz spaces). The structure gives tools and coordinates for collections of pairs $(f,g)$ for which the properties -- say, reducibility and number of components, genuses -- are constant on the families. So, specific solutions to the problems correspond to points in these spaces. 

If we stick to the $(f,g)$ formulation, then our central problem is to consider pairs of $(f,g)$ for which $\tilde \sC_{f,g}$ has a low genus component. \S\ref{charred} starts with this pair notation but ends by noting all properties are constant on covers in a given Nielsen class. 

We give representative Nielsen class data illustrating how much things vary when $(G,\bfC,\bT)$ changes. Our main examples use when $f$ is indecomposable. Since we must consider covers that are reducible, \S\ref{ncnontrans}  extends {\sl Nielsen classes\/} to drop the usual transitivity assumption, the main point in generalizing the genus formula (for components) when $\tilde \sC_{f,g}$ is reducible. 

\subsubsection{Characterizing reducibility of $\tilde \sC_{f,g}$} \label{charred}  We start with data for constructing $(f,g)$ examples for which $\tilde \sC_{f,g}$ has more than one component.\footnote{Some examples, mostly excluded by \eql{Pak}{Paka}, get repeatedly rediscovered. \S\ref{orbzero}  revisits these  briefly to show why they are misleading about the nature of the problem.}  Examples toward extending Thm.~\ref{PakThm} will begin as a Nielsen class --- lists of branch cycles -- defining pairs of covers, $(f^*,g^*)$, satisfying Prop.~\ref{redRed} properties.  

The monodromy (Galois closure) group, $G_f$, of a cover with its (faithful) representation $T_f: G_f \to S_{\deg(f)}$ acts on cosets of the subgroup, $G_f(T_f,i)$, stabilizing an integer, $1\le i\le \deg(f)=m$. Since $f$ is irreducible, the $G_f(T_f,i)\,$s are conjugate, giving {\sl equivalent\/} permutation representations. Then, the quotient on the Galois closure cover $\hat f: \hat \prP^1_x \to \prP^1_z$ returns $$f: \hat \prP^1_x/G_f(T_f,i)\to \prP^1_z: \text{covers for different $i\,$s are equivalent, Def.~\ref{covequiv}}.$$ 

Prop.~\ref{redRed} is about a group $G$ having two different permutation representations for which their {\sl entanglement\/} represents having a reducible fiber product, much of which generally applies to the fiber produce $\tilde \sC_{\phi_X,\phi_Y}$, for irreducible covers $\phi_X$ and $\phi_Y$. With our intense preoccupation with genus zero covers and our computations with branch cycles and Nielsen classes, it is more efficient to use rational functions for most results. For the letters of those representations --when we must distinguish them -- we use  $\row x m$ (resp.~$\row y n$) representing the zeros of $f(x)\nm z$ (resp.~$g(y)\nm z$) on which $T_1$ (resp.~$T_2$) act transitively. 

Prop.~\ref{redRed}, \cite[Prop.~2]{Fr73a} or \cite[Lem.~4.2]{Fr12}, transfers considering reducibility of $\tilde \sC_{f,g}$ to the case where  $G_{f}=G_{g}$, with two (faithful) permutation representations, $T_1=T_{f},T_2=T_{g}$.

\begin{defn} \label{extrep} Given a cover of groups $\psi: G^\dagger \to G$, with respective transitive permutation representations $T^\dagger$ and $T$  defined by the cosets of $H^\dagger$ and $H$, we say $T^\dagger$ extends $T$, if $\psi(H^\dagger)$ is a conjugate of $H$. \end{defn} 

Or, in Def.~\ref{extrep}, we could say $T$ is a quotient of $T^\dagger$, with the understanding there could be several $T^\dagger\,$s giving a specific $T$. We follow Prop.~\ref{redRed} with clarifying comments on rational functions.  

\begin{prop} \label{redRed} There exists $(f^*,g^*)\in \sG_{f,g}$  satisfying the following.  

\begin{edesc} \label{redGal} \item \label{redGala} $\hat f^*: \prP^1_{x^*}\to \prP^1_z$ and $\hat g^*: \hat \prP^1_{y^*} \to \prP^1_z$,  are  equivalent as Galois covers. 
\item \label{redGalb}  Components of $\tilde \sC_{f,g}$ and $\tilde \sC_{f^*,g^*}$ correspond  one-one. 
\end{edesc} For $(f^*,g^*)\in \sG_{f,g}$: 

\begin{edesc} \label{concredGal} 
\item \label{redGalc}  $f^*$ and $g^*$ have exactly the same branch points, with respective branch cycles of $f^*$ and $g^*$ of the same order. 
\item  \label{redGald}  Components on $\prP^1_{x^*}\times_{\prP^1_z} \prP^1_{y^*}\leftrightarrow$ orbits of $G_{f^*}(T_{f^*},1)$ under $T_{g^*}$ \\ \hbox{\hskip .2 in} $\leftrightarrow$  orbits of $G_{g^*}(T_{g^*},1)$ under $T_{f^*}$ (switch $T_{f^*}$ and $T_{g^*}$). 
\item \label{redGald'} A component $W$ $\leftrightarrow$ $G_{f^*}(T_{f^*},1)$ orbit $I$ (resp.~$G_{g^*}(T_{g^*},1)$ orbit $J$) in \eql{concredGal}{redGald}, has degree $|I|$ (resp.~$|J|$) over  $\prP^1_{x^*}$ (resp.~$\prP^1_{y^*}$).\footnote{Given an irreducible factor $h(x_1,y)$ of $f(x_1)\nm g(y)$ over $\bC(x_1)$, if $h(x_1,y_1)=0$, the degree of $h$ in $y$ is the number conjugates of $y_1$  under $G_{f^*}(T_{f^*},1)$.}  \end{edesc} 
\end{prop} 

\begin{equation} \label{proofuse} \text{\sl Comments: Proof and use of Prop.~\ref{redRed}:}\end{equation}  In \eql{redGal}{redGald} using $\row x {m^*}$ and $\row y {n^*}$, restate \lq\lq orbits of $G_{f^*}(T_{f^*},1)$ under $T_{g^*}$\rq\rq\ as \lq\lq orbits of the subgroup of $G_{f^*}$ fixing $x_1$ acting on $\row y {n^*}$.\rq\rq\   

If $I$ is such an orbit on the $y\,$s, then the degree of the component over $\prP^1_{x^*}$ is $k_I=|I|$. The notation switching $x^*$ and $y^*$ would have the degree of the component over $\prP^1_{y^*}$ corresponding to an  orbit $J$ on the $x\,$s as $\ell_J=|J|$. 

\cite[Prop.~2]{Fr73a}  uses the Theorem of {\sl natural irrationalities\/} in a see-saw argument. Here is the original paper: $$\text{http://www.math.uci.edu/\hbox{\~{}}mfried/paplist-ff/dav-red.pdf}.$$   

There is a natural partial ordering on $\sG_{f,g}$ using Def.~\ref{extrep}. There is also a maximal simultaneous Galois cover of $\prP^1_z$ through which both  $\hat f$ and $\hat g$ factor (with group $G^\mx$ in Rem.~\ref{minGalclos}) which may not be represented in $\sG_{f,g}$.  If it is, this would be a unique maximal element for the partial ordering.

\begin{equation} \text{{\sl Comments: Our interest is in more than one component of $\tilde \sC_{f^*,g^*}$:} }\end{equation} 
 \begin{edesc} \label{f*g*components} \item Equivalence of $\hat f_*$ and $\hat g_*$  in \eql{redGal}{redGala} implies $\tilde\sC_{f^*,g^*}$ components are quotients of $\hat \prP^1_{x^*}$ as given in  \eqref{correspondences}. 
\item If $\tilde \sC_{f,g}$ is irreducible, then $(f^*(x)=g^*(y)=z)\in \sG_{f,g}$. Our interest is \eql{redGal}{redGalb} beyond Thm.~\ref{PakThm}: more than 1 component. 
\item $(f^*,g^*)\in \sG_{f,g}$ can be a significant entanglement (the same meaning as in \S\ref{noclassif}), even if $\tilde \sC_{f^*,g^*}$ is irreducible. 
\end{edesc} 

\begin{equation} \label{quotreps} \text{\sl Comments: Quotient representations with $u>1$ in \eqref{redGal}:} \end{equation} 
Prop.~\ref{redRed} can give many several examples of $\tilde\sC_{f^*,g^*}$ having the same number of components as does $(f,g)$. \begin{edesc} \label{repqts} \item From \eqref{factor}, there is has a natural partial ordering on $\sG_{f,g}$: For $(f^*,g^*), (f^\star, g^\star)\in \sG_{f,g}$ then $(f^*,g^*)\ge  (f^\star, g^\star)$ if the cover pair $(f^*,g^*)$ factors through $(f^\star, g^\star)$. 
\item The count of components on $\tilde \sC_{f^*,g^*}$ is nondecreasing (take $Z=Z'$ in Prop.~\ref{compimage}). 
\item \label{repqtsc}  In our examples $(f^\star, g^\star)$ is minimal among quotients of $(f^*,g^*)\in \sG_{f,g}$ with $\tilde \sC_{f^\star,g^\star}$ having multiple components, relying on naming such examples from the genus 0 problem. 
\end{edesc} 
Still, in \eql{repqts}{repqtsc} you could start anywhere in a chain with a term with a fiber having multiple components to  consider a Pakovich-type result. 

\begin{equation} \text{\sl Comments: Distinguished properties of rational function covers:}\end{equation} It is easy to write genus 0 covers of $\prP^1_z$ -- just list coefficients of a rational function -- the properties \eqref{goodrat} may explain why so many study them. 
\begin{edesc} \label{goodrat} \item \label{goodrata}  A general cover $\phi: X \to Z$ can be reducible, but a curve cover from a rational function $f$ is always irreducible.  
\item \label{goodratb} Composites of rational functions are always defined, but such a composition may not be unique.  \end{edesc}  

Reminder of the proof of \eql{goodrat}{goodrata}: Assume $f=f_1/f_2$, with relatively prime $f_1,f_2\in \bC[x]$.  Then $f_1(x)-f_2(x)z$ is irreducible, since any factors would have degree at least 1 in $z$. \S\ref{compmotivation} reminds of motivations from \eql{goodrat}{goodratb}.

\begin{defn}[Cover equivalence] \label{covequiv} Two covers $\phi_i: X_i\to Z$, $i=1,2$, are (absolute) equivalent if there is $\psi: X_1\to X_2$ such that $\phi_2\circ \psi=\phi_1$. When the $X_i\,$s are both $\prP^1_x$ (and $Z=\prP^1_z$), then $\phi$ identifies with an element of $\PSL_2(\bC)$ (M\"obius transformation). Further, call the $f_i\,$s, $i=1,2$,  {\sl reduced equivalent\/} if there are $\psi,\psi'\in \PSL_2(\bC)$ for which $\psi'\circ f_2\circ \psi=f_1$. \end{defn} 

See Rem.~\ref{remredequiv} for a warning on using reduced equivalence. Ex.~\ref{diagcomp1} comments on the most obvious cases when $\tilde \sC_{f,g}$ is reducible. 

\begin{exmpl} \label{diagcomp1} Assume $f=f_1\circ f_2$, $\deg(f_1)>1$, $g=g_1\circ g_2$, with  $f_1$ and  $g_1$ equivalent covers of $\prP^1_z$.  
\begin{edesc} \label{trivsit} \item Then $\tilde \sC_{f_1,g_1}$ is reducible with  a component  isomorphic to the diagonal in the fiber product of $f_1$ with itself. 
\item \label{trivsitb}  Similarly, $\tilde \sC_{f,g}$ is reducible with  a component isomorphic to the (projective normalization of) $\{(x,y)\mid f_2(x)\nm g_2(y)=0\}$. 
\end{edesc} 
Taking $f_2(y)=y$ in \eql{trivsit}{trivsitb} gives a genus 0 component of $\tilde \sC_{f,g}$. The components remaining {\sl complementary to\/} the diagonal could be significant. List the orbits in \eql{redGal}{redGald} as $\row I u$. The case \eql{trivsit}{trivsitb} that corresponds to $u=2$ is with $T_{f^*}$ {\sl doubly transitive}. As a case of the genus formula, Cor.~\ref{2-foldfiber} gives the genus of the non-diagonal component. 
\end{exmpl}

The following lemma is automatic from the natural projection maps to $\prP^1_x$ and $\prP^1_y$ attached to a component of a fiber product.  Parallel to Prop.~\ref{redRed} for reducibility; it characterizes when a component of the fiber product has genus 0. Neither  Prop.~\ref{redRed}  nor does it show how to find such a component or compute its genus. Those are the main topics of the paper. 

\begin{lem}[Non-uniserial] \label{twodecomp} Suppose $\tilde \sC_{f^*,g^*}$ has a genus 0 component, $W$.  That gives rational functions $h_x\in \bC(w)$ and $h_y\in \bC(w)$ that fit in a commutative diagram: 

%\begin{equation} \begin{tikzcd}
%%[column sep=small]
%\sT^\gamma \arrow[d, "\alpha(\sT^\gamma)"'] \arrow[rr, "\Psi^\gamma"]  &  & U_r\times \prP^1 \\
%\sT  \arrow[urr, "\Psi"]  &   & 
%\\ \end{tikzcd}  \end{equation}

\begin{equation} \label{genus0diag} \begin{tikzcd}
%[column sep=small]
&&  \arrow["h_x"', lld]  W \arrow[dd, "h^*"'] \arrow[drr, "h_y"]  &\!\!\!\!\!\!\!\!\!\!\!\!\!\!\!\subset \tilde \sC_{f^*,g^*} &\\ 
\prP^1_x \arrow[drr, "f^*"] && &   &   \arrow["g^*"', lld] \prP^1_y 
\\ &&\prP^1_z&&\end{tikzcd}  \end{equation}

\begin{edesc} \label{ratfunctdec} \item \label{ratfunctdeca} Thus,  \eqref{genus0diag} gives an $h^*\in \bC(w)$ with two decompositions:   $$h^*=f^*\circ h_x=g^*\circ h_y.$$ 
\item \label{ratfunctdecb}  Conversely, the rational function decomposition of \eql{ratfunctdec}{ratfunctdeca} produces a genus 0 component of $\tilde \sC_{f^*,g^*}$ as in \eqref{genus0diag}. 
\item \label{ratfunctdecc} For each $g_2$, $\tilde \sC_{f^*,g^*\circ h_y\circ g_2}$ has a genus zero component. \end{edesc} 
Replacing $g^*$ by an equivalent cover $g^*\circ\alpha$, $\alpha \in \PSL_2(\bC)$, requires replacing $h_y$ by $\alpha^{-1}\circ h_y$ to keep the decomposition of $h^*$.\end{lem} 

\begin{proof} The only item requiring further proof is \eql{ratfunctdec}{ratfunctdecc}. Apply Ex.~\ref{diagcomp1} to $\tilde \sC_{f^*,f^*\circ h_x\circ g_2}$ using that $g^*\circ h_y\circ g_2=f^*\circ h_x\circ g_2$. \end{proof}

\cite[Thm.~2]{Fr73b} viewed searching for genus 0 components of $\tilde \sC_{f^*,g^*}$ with $f^*,g^*\in \bC[x]$ (polynomials)  as generalizing Ritt's Theorem. That  we regarded as isting all cases of $h^*\in \bC[x]$ (polynomial) in Lem.~\ref{twodecomp}, as in  \begin{equation} \label{cfr} \text{{\sl Cyclic factor reduction Diagram} \cite[p.~46]{Fr73b}}.\end{equation}  

Prop.~\ref{compinc} has two distinct types of potential components that figure in our generalization of Thm.~\ref{PakThm}. One from our main technique, the other -- less likely, called a {\sl decomposition variant\/} -- given by  \eqref{genus0diag}.  
 
\begin{rem} \label{minGalclos} Given covers $\hat \phi_X: \hat X\to Z$ and $\hat \phi_Y: \hat Y\to Z$, there is a minimal simultaneous Galois cover, $\hat \phi_{X,Y}: \hat V_{X,Y} \to Z$ of them both. It has group $G_{\phi_X,\phi_Y}$ with fiber product description $G_{\phi_X}\times_{G^{\mx}} G_{\phi_Y}$ where $G^{\mx}$ the maximal group through which both $\hat \phi_X$ and $\hat \phi_Y$ factor. \end{rem} 

\begin{rem} \label{extredRed} We didn't need an extra assumption on the covers from $f$ and $g$ being irreducible in Prop.~\ref{redRed}, because -- \eql{goodrat}{goodrata}  -- rational function covers are always irreducible. Making that extra assumption, though, then extends Prop.~\ref{redRed} to any pair of covers.\end{rem} 

\begin{rem}[Subdegrees] \label{subdegrees} Orbit lengths, even without assuming $\phi_X$ and $\phi_Y$ are genus zero cover, orbit lengths of, say, of $G(T_1,x_1)$ on $\{y_1,\dots, y_n\}$ are so-called {\sl subdegrees}. For $f$ equivalent to $x^m$, or resp.~to Ex.~\ref{ochar2^2infty}; (called Chebychev), then sub degrees are 1, resp. 2. Other than these two examples,  if $f\in \bC[x]$  (a polynomial) and $T_1$ is primitive, then the only sub degrees  of $\tilde\sC_{f,f}$ are 1 and $m\nm 1$ \cite[Thm.~1]{Fr70}: $T_1$ is {\sl doubly transitive}.\end{rem}   

\begin{rem} \label{remredequiv} I use reduced equivalence (Def.~\ref{covequiv}) often in my papers. It suits moduli problems and produces reduced Hurwitz spaces whose distinct points correspond to covers that truly seem distinct rather than related by an algebraic trick. Here, though, it requires careful use; you must, when involving a pair $(f^*,g^*)$, apply composite on the left by $\alpha\in \PSL_2(\bC)$ simultaneously to both rational functions.\end{rem} 

\subsubsection{Relevant branch cycles} \label{startdeg7}  We create the notation for a Nielsen class example. \S\ref{bcdeg7} uses it to illustrate Prop.~\ref{redRed}. This is for a natural family of pairs of polynomial covers $(f^*,g^*)$ of degree 7, whose members show how to drop the irreducibility hypothesis of Thm.~\ref{PakThm}. 

Back to Prop.~\ref{redRed}: Start, not with rational function pairs,  but with a group having two permutation representations (in this illustrating case of degree 7) $G = G_{f^*} = G_{g^*}$, with two 
(faithful) permutation representations, $T_1 = T_{f^*} $, $T_2 = T_{g^*}$; here $G=\GL_3(\bZ/2)=\PGL_2(\bZ/2)$.\footnote{\S\ref{thelit} tells some of the story behind this example, and the {\sl Genus 0 Problem\/} that generalizes it. \cite{Fr05b} gives more detail  up to the early 2000s.}  

The actions of $T_1$ and $T_2$ are, respectively, on points and lines of projective 2-space over $\bZ/2$. Use $\row x 7$ and $\row y 7$ as the respective letters of the representations in \eqref{bcfg}. 

Now we want cover pairs $(f^*,g^*)$ satisfying Pakovich considerations. 
\begin{edesc} \label{nonpak} \item \label{nonpaka} Ensure $\tilde \sC_{f^*,g^*}$ is reducible; 
\item \label{nonpakb}  Check that $(f^*,g^*)$ corresponds to a pair of genus 0 covers of $\prP^1_z$  (so given by rational functions).  
\item \label{nonpakc}  Compute if a component of $\tilde \sC_{f^*,g^*}$  has genus 0 (or 1)? 
\item \label{nonpakd} If the desire is for $f^*$ to be a polynomial cover, guarantee a totally ramified branch point. 
\end{edesc} 

This is how you get such. Choose a collection of $r$ conjugacy classes $\bfC$ to form a {\sl nonempty\/} Nielsen class with properties listed in \eqref{RET}: 
\begin{equation} \ni(G,\bfC)=\{\pmb\mu\mid \lrang{\mu_1,\mu_2,\cdots,\mu_r}=G, \mu_1\cdot \mu_2\cdots\mu_r=1 \text{ and }\pmb \mu \in \bfC\}.\end{equation}  
That gets us to satisfying \eql{nonpak}{nonpaka}. This follows because for any $\sigma\in G$, from equality of the traces of $T_1(\sigma)$ and $T_2(\sigma)$, a special case of \cite[Lem.~2]{Fr73a}\footnote{Not something that I would expect readers interested in this problem to know.} which also guessed -- and later proved -- the generalization of where such triples $(G,T_1,T_2)$ came from. This, and Prop.~\ref{gendeg7} -- which we use to explain the goals of the genus 0 problem in \S\ref{classif} -- are documented in \cite{Fr99}. 

Now, choose $\bfC$ judiciously to assure \eql{nonpak}{nonpakb} holds. Using \RH, we find there is a  {\sl most general\/} -- Def.~\ref{defcoalescing} on {\sl coalescing\/} explains that phrase -- set of conjugacy classes $\bfC$ in $G$, so that all degree 7 rational function pairs $(f^*,g^*)$ satisfying \eql{nonpak}{nonpakb} come from  $\pmb \mu\in \ni(G,\bfC)$. 

That most general $\bfC\eqdef \bfC_{2^6}$ consists of  six repetitions of the involution class of $G$. We explain the last sentence of Prop.~\ref{gendeg7} by example in our continuation below to discussing \eql{nonpak}{nonpakc} and \eql{nonpak}{nonpakd}. 

\begin{prop} \label{gendeg7} All degree 7 rational function pairs with six branch points and $\tilde \sC_{f^*,g^*}$ reducible, come from coalescing in  this diagram: 
\begin{equation} \label{gendeg7rat} \begin{array}{c}  \pmb \mu\in  \ni(\PGL_2(\bZ/2),\bfC_{2^6})\mapsto (T_1(\pmb \mu),T_2(\pmb \mu)) \\  \in \ni(\PGL_2(\bZ/2),\bfC_{2^6},T_1)\times \ni(\PGL_2(\bZ/2),\bfC_{2^6},T_2).\end{array} \end{equation} Further, all other such Nielsen classes giving diagrams like that of \eqref{gendeg7rat} come from coalescing elements in $\ni(\PGL_2(\bZ/2), \bfC_{2^6})$.\end{prop} 

The notation $\bT$ on representations alluded to in \S\ref{compfried-pak} corresponds to the diagram \eqref{gendeg7rat}. Statement \eql{deg7prop}{deg7propa} is from the general theory of Hurwitz spaces as in \cite{Fr77}. Statement \eql{deg7prop}{deg7propb}, though, is a consequence of $H_6$ transitivity on $\ni(\PGL_3,\bfC_{2^6})$ as in \eqref{braidaction}. Denote by $U_r$, the collection of $r$ distinct points on $\prP^1_z$. 

\begin{edesc} \label{deg7prop} \item \label{deg7propa}  Any $\bz\in U_6$ gives $|\ni(\PGL_3(\bZ/3),\bfC_{2^6})|$ degree 7 pairs $(f^*,g^*)$ with branch points at $\bz$ satisfying \eql{nonpak}{nonpaka} and \eql{nonpak}{nonpakb}.  
 \item \label{deg7propb}  Running over $\bz\in U_6$, there  is one connected Hurwitz space component of the pairs in \eql{deg7prop}{deg7propa}. \end{edesc} 
 
In \S\ref{7branchcyclesa} we find that the components on $\tilde \sC_{f^*,g^*}$ given by \eqref{deg7prop} don't have genus 0, and they certainly aren't given by polynomials. That is they satisfy neither of  \eql{nonpak}{nonpakc} nor \eql{nonpak}{nonpakd}.\footnote{The polynomials cases were for Davenport's and Schinzel's problems  in classifying when $\tilde \sC_{f^*,g^*}$ was reducible.} 

To get the general cases satisfying polynomial condition \eql{nonpak}{nonpakd}, take the number of elements, $r$, in $\bfC$ to be 4; three are of the involution class, $\C_2$ and the 4th, called $\C_\infty$, is a 7-cycle under $T_1$ or $T_2$. See \S~\ref{pakgoal}. There are  {\sl  two\/} 7-cycle conjugacy classes  in $\PGL_2(\bZ/2)$.  Denote the conjugacy collection by $\bfC_{\infty\cdot 2^3}$. These appear again in \S\ref{pakgoal}. 

Cor.~\ref{satellitecor} has the two conjugacy classes sets for polynomials, $\bfC_{2\cot3\cdot7}$ and $\bfC_{2\cot4\cdot7}$, that give Nielsen classes whose fiber product representatives have (two each) components of genus 0 or 1. We are getting close to Ex.~\ref{deg7covergenus}, but those are 3-tuples, not 4-tuples.  Further, again for $\bfC_{\infty\cdot 2^3}$, the resulting $\tilde \sC_{f^*,g^*}$ -- formed analogous to \eqref{deg7prop}; to get polynomials take the branch point associated to $\C_\infty$ to be $\infty$  -- don't have a genus 0 component. \S\ref{deg7} uses these to illustrate the general genus computation for components. 

\subsection{Results of the paper and role of the Classification} \label{thelit}  In lieu of Discreteness Princ.~\eqref{ncdiscreteness}, the search is for Nielsen classes $\ni(G,\bfC,\bT)$ defined by some pair $(f^*,g^*)$ in the conclusion of Prop.~\ref{redRed}.  Especially they have the same Galois closures with corresponding pairs of permutation representations for which $\tilde \sC_{f^*,g^*}$ is reducible. 

\S\ref{Pakfails} describes the subsections using their emphases on Nielsen classes. A Nielsen class gives an algebraic structure to the family of covers they describe. Using Prop.~\ref{redRed} shows how crucial it is to understand that Galois closure. It is a tool for describing  Nielsen classes of covers $g$ for which $\tilde\sC_{f,g}$ is reducible. Also, it allows us to compute the genuses of those components.  The goal is to regulate existence of a rational function series  $\{{}_ig_1\}_{i=1}^\infty$ with:  
\begin{edesc} \item $\lim_{t\mapsto \infty} \deg({}_ig_1)=\infty$; and 
\item $\tilde \sC_{f,g\circ   {}_i g_1}$ has, say, a genus 0 (or 1) component.\end{edesc}  Thereby, $\{{}_ig_1\}_{i=1}^\infty$ gives a failure for $(f,g)$ to condition \eql{Pak}{Pakb}.

\S\ref{classif} goes from our degree 7 examples into the central role of the classification relevant to this paper: Resolution of the genus 0 problem. We formulate that as the description of Nielsen classes, $\ni(G,\bfC,T)$,  of families of genus 0 indecomposable covers: $T$ is a primitive representation of $G$. 

It concludes with notation for going beyond that indecomposable case.  Still, Rem.~\ref{indecomps} explains why we decided -- since so few know the tools we discuss here -- to avoid, for now, creating detailed induction notation that would start at decomposable rational functions. Pakovich's topic includes dealing with decomposable rational functions. We have generalized it with explicit examples, and their tools, starting with indecomposable functions.

\S\ref{noclassif} discusses work on separated variable equation properties that benefit from using groups but is entirely classification free. In particular, it mentions examples where solutions started without reference to the classification. Yet the solution was achieved as a collaboration combining the work of adherents of arithmetic/geometry and group theory. 

 \S\ref{compmotivation}  touches on Pakovich's motivation, places where separated variables equations are significant but hidden as a topic of concentration. It ties to the ubiquitous use of rational function fiber products in moduli problems, specifically back to Riemann proving properties of $\theta$ functions.

 \subsubsection{Results of the paper} \label{Pakfails}  
\S\ref{cbnc} and \S\ref{usenc} do a 2-step introduction to Nielsen classes.  They also introduce Pakovich's $\ochar$ (orbifold characteristic): apply {\sl Riemann-Hurwitz\/} (\RH) to the genus $\geng_{\hat f}$ of the Galois closure, then divide by $|G_f|$. This gives a simplifying expression for quantities of interest to us. We regard it as known to Riemann  (despite its tie to {\sl rea\/}l 3-manifolds). 

\S\ref{ncnontrans} shows how to treat Nielsen classes of these pairs $(f,g)$. With a slight abuse we also use $U_f$ (or $U_{\ni(G,\bfC)^\abs}$)  to refer to those Nielsen classes. 

\S\ref{genusformCfg} gives the formula for the $\tilde \sC_{f^*,g^*}$ component genuses -- Cors.~\ref{methodI}  and \ref{methodII} (Methods I and II)  -- when it is reducible, extending the original formula of  \cite[(1.6) of Prop.~1]{Fr73b}. As a corollary, it gives the genus of the nondiagonal component of $\tilde \sC_{f^*,f^*}$ when $T_{f^*}$ is doubly transitive rather than just primitive.  \footnote{Cor.~\ref{methodI} applies when components are on a fiber product of nonsingular covers.}  \cite[Prop. 1]{Fr73b} was the case the fiber product is irreducible. So, this  precisely extends adhoc discussions in \cite{Fr73b} and \cite{Fr74}.\footnote{Pakovich suggested a precise genus calculation is computationally formidable. We think our argument is a counter, especially including our Nielsen class formulation.}

\S \ref{charred} has already taken this approach: 1st find $(f^*,g^*)$ (or rather Nielsen classes; Prop.~\ref{redRed}) that have at their core $\tilde \sC_{f^*,g^*}$ reducible. Then focus on characterizing when such components might have genus 0 (Lem~\ref{twodecomp}) and how that defies a simplistic version of Pakovich's Theorem \ref{PakThm}. 

\S\ref{irreducibilityb} ramps up our running example to illustrate treating Nielsen classes of $\tilde \sC_{f^*,g^*}\,$s that have genus 0 components. 

\S\ref{contextdeg7} then applies coalescing to go from a large Nielsen class of reducible to $\tilde \sC_{f^*,g^*}\,$s to those with genus 0 components. We mostly refer to other literature for what has happened in those directions that stem from the solution of Davenport's and the genus 0 problem. My concern here is how these characterizations can tackle generalizing Pakovich's Theorem. 

\S\ref{overview-PartI} emphasizes the permutation pairs $(T_1,T_2)$ of a group $G$ that appear in Prop.~\ref{redRed} and how using Nielsen classes allows direct access to the results of the genus 0 problem. \S\ref{ocharg0} completes our paper-long -example  where $f^*$ appears in Nielsen classes with $T_1$ of degree 7. In particular, we use it to illustrate how precise we can be in generalizing Thm.~\ref{PakThm}. 

R.~Guralnick produced the final precise solution on what are the primitive permutation representations $(G,T)$ represented by rational functions (genus 0 covers).\footnote{He also directed the results for  variants on covers of a fixed higher genus.}  To keep the statement uncomplicated, we restrict to just those primitive groups with their {\sl cores\/} related to simple non-cyclic groups. Then, there is a vast but finite list of such $(G, T)$ giving genus 0 covers. 

Outside group theorists, I've found little knowledge of how simple and primitive groups relate. Here are some comments.  \cite{AOS85} constructs a template of five patterns of primitive groups. Into four of those, you insert almost simple groups. Affine groups comprise the fifth. So, the solution of the Genus 0 Problem ran through two filters: \cite{AOS85}; and the distinct series of finite simple groups, together with affine. This {\sl lexigraphic\/}  procedure accounts for the number and length of contributions to the genus 0  resolution (for covers over $\bC$).

Here are keywords that connect simple groups to the primitive group classification. According to \cite{GLS}, a {\sl quasisimple\/} group $G$ is a perfect central cover $G \to S$ of a simple group $S$. Here: {\sl cover\/} means onto homomorphism; {\sl perfect\/} means the commutators $g_1g_2g_1^{-1}g_2^{-1}$ in $G$ generate $G$; and {\sl central\/} means the kernel is in the center of $G$. Such a cover is a special case -- because we don't assume $S$ is simple -- of a {\sl Frattini\/} central cover: where the map, if restricted to a proper subgroup of $G$, won't be a cover.  Then, if $S$ is perfect, so is $G$.

A component, $H\le G$, of $G$, is a quasisimple subgroup with a composition series, between $H$ and $G$, a sequence of groups each normal in the next.  The group generated by components and the maximal normal nilpotent subgroup of $G$ is called the {\sl generalized Fitting subgroup}, $F^*(G)$, of $G$.  \cite{GLS} calls a group $G$ almost simple if $F^*(G)$ is quasisimple.

The one significant exception to finiteness is where the core is an alternating group (includes several different groups and representations, including $G=S_n$). We give two examples that arose in actual applications to show that, among series of such genus 0 pairs, the non-obvious conclusion -- when restricted to consider such pairs giving reducible fiber products -- is there can be both infinite series (\S\ref{doubledegree}) of such or just finitely many (\S\ref{appSpaces}). 

\S\ref{genPak} gives a Pakovich generalization using the characterization of what $g_1\,$s to avoid, so that genus 0 components don't appear on the fiber products $\tilde \sC_{f^*,g^*\circ g_1}\,$s in the corresponding Nielsen class. Generalizing Thm. 1.1, the result characterizes paths of $g_1\,$s for which the number of irreducible components doesn't change from that of $\tilde \sC_{f^*,g^*}$. Then, for each such path, it concludes a result akin to Thm.~\ref{PakThm}.

\subsubsection{Decomposable rational functions} \label{classif}

To create covers $f: \prP^1_x\to \prP^1_z$, with specific decompositions use groups constructed to have various systems of imprimitivity, the topic of this subsection. This suits a Nielsen class  given by $G$, a permutation representation $T$, and conjugacy classes $\bfC$, with the stipulation that \RH\ gives genus 0 for a cover associated with $\psigma\in \bfC$. 

Rem.~\ref{decchain+} notes that the proof of Lem.~\ref{decchain} has taken advantage, notationally, of the covers being given by rational functions.

\begin{lem}[The decomposition chain] \label{decchain}  Maximal chains of decompositions of $f\in \bC(x)$ ($\deg(f)=m$) correspond to the following equivalent sets.  
\begin{edesc} \label{maximprim} \item \label{maximprima}  Maximal chains of subsets $$ \begin{array}{rl}&\{1\}<I_2 <\dots I_{u\nm1}< \{1,\dots,m\},\text{ with $1\in I_j$; and for $\sigma \in G_f$},\\ 
& \text{{\hskip .5in} and $1\le j\le u$, if  $1\in (I_j)T_f(\sigma)$, then  $(I_j)T_f(\sigma)= I_j$.}\end{array}$$ 
\item \label{maximprimb} Maximal chains of \eql{maximprim}{maximprima} correspond to maximal chains (by inclusion) of groups between $G_f$ and $G_f(T_f,1)$: overgroups of $G_f(T_f,1)$.  \end{edesc} 
\end{lem} 
 
\begin{proof}  Take $H$ a group between $G(T_f,1)$ and $G_f$. The Galois correspondence produces a chain of fields $\bC(x_1)\ge \bC(x_H) \ge \bC(z)$ with $f_H, f_H^*\in \bC(x)$, $f_H^*(f_H(x))=z$ and $x_H=f_H(x_1)$. Then, $x_1$ is a zero of $f_H(x)=x_H$; the other zeros are $\deg(f_H)$ conjugates of $x_1$, all zeros of $f(x)=z$. 

Given a fixed labeling of the zeros of $f(x)=z$, $\row x m$, the result is a set $I_H$ appearing in \eql{maximprim}{maximprima}.  This is a system of {\sl imprimitivity}; $I_H$ and its conjugate sets by the action of $G_f$ partition $\{1,\dots,m\}$ into disjoint sets. 
\end{proof} 

Based on Prop.~\ref{redRed}, given $f$, to find cases where Thm.~\ref{PakThm} does not hold, consider these steps.   

\begin{edesc} \label{stepcomp} \item \label{stepcompa}  List nontrivial composition factors $f^*$ of $f$. 
\item \label{stepcompb} For each $f^*$ in \eql{stepcomp}{stepcompa} list  (up to equivalence of covers) $g^*\,$s (including $f^*$) where  $\tilde \sC_{f^*,g^*}$ has more than one component as in \eqref{redGal}.  
\item \label{stepcompc} Compute the genuses of the components of $\tilde \sC_{f^*,g^*}$. \end{edesc}

 Locating proper subgroups  $G_f\ge H\ge  G_f(T_f,1)$, with the Nielsen class of $T_H$ of genus 0, is the 2nd key to finding genus 0 components of   reducible $\tilde \sC_{f,g}$ as $g$ varies (Lem.~\ref{compcor}).  Take $$\phi_{X,X'}: X'\to X, \phi_{Y,Y'}: Y'\to Y,\phi_{Z,Z'}: Z'\to Z$$ to be three (nonconstant) covers of irreducible compact Riemann surfaces.

\begin{prop}[Component image] \label{compimage}  Suppose $W$ is an irreducible normal projective curve and $\phi: W \to Z$ is a morphism that factors through $\phi_{Z,X}$ and $\phi_{Z,Y}$ so that $\phi_{Z,X}\circ \phi=\phi_{Z,Y}\circ \phi$. Then, $\phi$ factors through a (unique) component of $\tilde \sC_{\phi_{Z,X},\phi_{Z,Y}}$. 

Assume given covers $$\phi_{Z,X}: X\to Z\text{ and }\phi_{Z,Y}: Y\to Z,$$ and similar for $X',Y',Z'$ compatible with the maps above. These induce  $$\phi_{X,Y,Z; X',Y',Z'}: \tilde \sC_{\phi_{Z',X'},\phi_{Z',Y'}}\to \tilde \sC_{\phi_{Z,X},\phi_{Z,Y}},$$  an onto  map on components that is genus nonincreasing. 

If a component $W'\le \tilde \sC_{\phi_{Z',X'},\phi_{Z',Y'}}$ has $\geng_{W'}>1$, and the restriction map $W'\to W$ has degree exceeding 1, then $\geng_{W'}> \geng_W$. Also, $\geng_{W'}=1=\geng_W$ if and only if the restriction map is unramified. 
\end{prop}

\begin{proof} The first paragraph is a restatement of the universality of fiber products phrased to emphasize components. 

The map $\phi_{X,Y,Z; X',Y',Z'}$ has the effect: $(x',y')\mapsto (\phi_{X,X'}(x'),\phi_{Y,Y'}(y'))$ for $\phi_{Z',X'}(x')=\phi_{Z',X'}(y')$. Compatibility  on maps means both coordinates lie above the same point of $Z$. Extended  to fiber product normalizations, $$\text{this follows from } \phi_{Z,X}\circ 
\phi_{X,X'}=  \phi_{Z,y}\circ  \phi_{Y,Y'}.$$ 

The genus nonincreasing statement follows from considering the map on any component $W'\subset \tilde \sC_{\phi_{Z',X'},\phi_{Z',Y'}}\to W\subset \tilde \sC_{\phi_{Z,X},\phi_{Z,Y}}$. The genus of the image is the number of linearly independent holomorphic differentials on $W$. The pullback of these differentials on $W'$ remain linearly independent. They give a subspace of the holomorphic differentials on $W'$. Thus, the genus of $W'$ is at least that of the genus of $W$.  

Finally, the last paragraph is a well-known consequence of \RH. 
\end{proof}
  
Most of Cor.~\ref{indredRed} is immediate from Prop.~\ref{redRed}.\footnote{It simplifies if $f^*$ is indecomposable.}  According to \eqref{redGal}, such $g^*$ give covers $Y_{g^*}\to \prP^1_z$ with the same Galois closure as has $f^*$. By the Galois correspondence, up to equivalence, covers correspond to subgroups of $G_{f^*}$. There are only finitely many.  In applying Prop.~\ref{redRed}, as in \eqref{f*g*components}  our interest is in considering nontrivial $\tilde \sC_{f^*,g^*}$ with $u$ components, $u>2$. 

\begin{cor} \label{indredRed} From  \eql{redGal}{redGala}, given $f^*$, the complete collection of $g\in \sR_{f^*}$ have the form $g^\star\circ g_1$ where $g^\star$,  up to equivalence of covers, runs over a finite set, $\sR_{f^*}^\star$, of $\prP^1_z$ covers with the following properties. 
\begin{edesc} \label{decvarconds} \item   The Galois closure $\hat f^*: \hat X_{f^*}\to \prP^1_z$ has  $\hat g^\star: \hat Y_{g^\star}\to \prP^1_z$ as a quotient. 
\item There is $f^\star\in \bC(x)$ for which $f^*=f^\star\circ f_1$ with $\hat f^\star$ and $\hat g^\star$ equivalent as Galois covers.  
\item With $(f^\star,g^\star)$ replacing $(f^*,g^*)$ and the permutation representations $(T_{f^\star},T_{g^\star})$ replacing $(T_{f^*},T_{g^*})$, properties \eqref{concredGal} hold. 
\end{edesc}    
\end{cor} 

\begin{rem}[Convenience of Luroth's Theorem] \label{decchain+} The proof of Lem.~\ref{decchain} took advantage of Luroth's Thm.~that the field corresponding to $H$ between $\bC(x_1)$ and $\bC(z)$ has a rational function generator $x_H$. That was a notational convenience, but -- thanks to the generality of Luroth -- a version of \eqref{maximprim} holds for any cover.  
\end{rem}

\begin{rem}{Using indecomposable $f^*\,$s} \label{indecomps} The genus 0 problem, our genus formula, Nielsen classes, and coalescing give a wealth of examples. I decided to leave proceeding with the induction until someone writes a GAP or Maple or Mathematica program for the genus computation using our examples, including the comments we have labeled \eqref{fail1aex}, \eqref{fail1bex} and \eqref{fail1cex}. These will automatically raise new starting points with decomposable rational functions. With more empirical data, good notation for interesting cases, especially decomposition variants, will fall in place. \end{rem} 

\subsubsection{Groups versus equations} \label{noclassif} \cite[p.~72]{O15} has a cocktail party description of the simple group classification. Mathematicians could use a guide to using pieces of the classification and its proof. A guide aimed at non-group theorists. \cite[p.~70]{O15} motivates the classification by reference to the {\sl Higgs boson}. This has problems. 

\begin{edesc} \item {\sl Awe\/} for the Higgs boson is more common than {\sl knowledge\/} of how to define a {\sl boson\/} (hint: foremost it includes photons and gluons). 
\item Practical molecular chemistry, unmentioned, as does much about elementary particles seriously uses (simple et al) groups. 
\item  The genus 0 problem (\S\ref{genus0prob}) via \cite{AOS85}  -- now commonly  applied to monodromy groups of those 9th grade  rational functions -- is the expertise of the four mathematicians central to \cite{O15}. 
\end{edesc} 

Here's a way to look at Davenport's problem and its solution.\footnote{Although not stated as I have here, special forays into versions of Chebotarev Density and Hilbert's Irreducibility have taken up much literature, as noted in \cite{FrJ86}.} 
Assume two covers, $\phi_X: X\to \prP^1_z$ and $\phi_Y: Y\to \prP^1_z$ as above, is defined over a number field $K$, with ring of integers $\sO_K$. Consider a relation, $R$, between the traces $\tr(T_X)$ and $\tr(T_Y)$ is interpreted  as follows. Running over the Frobenius $\Fr_{z',\phi_X}$ attached to $z'\in \prP^1_z(\sO/\bp)$ in the cover $\phi_X$, assume \begin{equation} \label{R-entanglement} \begin{array}{c} (\tr(T_X)(\Fr_{z',\phi_X}),\tr(T_Y)(\Fr_{z',\phi_Y}) \\\text{ satisfies $R$ for almost all $\bp$ and all $z'\in \prP^1_z(\sO/\bp)$}\end{array}.\footnote{It suffices to replace $z'\in \prP^1_z(\sO/\bp)$ by excluding a bounded -- independent of $|\sO/\bp|$ -- set of $z'$. The final result generalizing MacCluer's Theorem was that, for a given $\bp$, there were no exceptional $z'$: {\sl Monodromy precision \cite[\S3.2]{Fr05}}.}\end{equation} 
We say $(\phi_X,\phi_Y)$ are {\sl $R$-entangled}. Here are two examples. 
\begin{equation} \label{davrelation} \begin{array}{c} \text{Davenport-entanglement}: R_D \implies \tr(T_X) > 0 \equiv \tr(T_Y) > 0. \\
\text{Galois-entanglement}: R_G \implies  \tr(T_X)=\tr(T_Y). \end{array} \end{equation} 
I list steps to solving Davenport's problem in this case: $\phi_X$ and $\phi_Y$ are polynomials, $\phi_X=f$ and  $\phi_Y=g$,  $f$ is indecomposable,  $f$ and $g$ inequivalent. 
\begin{edesc} \label{D-entanglement}  \item \label{D-entanglementa}  $R_D$-entanglement ($(f,g)$ a Davenport pair) $\implies$ $R_G$-entanglement $\implies \tilde \sC_{f,g}$ is reducible. 
\item \label{D-entanglementb} There are only finitely many Nielsen classes of Davenport pairs. 
\item \label{D-entanglementc}  $K$ any number field, there is an explicit description of the Nielsen classes of \eql{D-entanglement}{D-entanglementb} over $K$, and of Davenport pairs over $K$. 
\item \label{D-entanglementd}  For $K=\bQ$ there are none. \end{edesc} 

There was no classification of simple groups in 1969 when I submitted \cite[Footnote p.~1]{Fr73a}  to {\sl The Journal of Algebra}. It   conjectured that primitive permutation representations that satisfied Prop.~\ref{redRed} were on the groups $\PSL_m(\bZ_q)\,$s -- with one degree 11 exception.\footnote{This was correct, so \cite[\S9]{Fr99}  gives tests of  Cor.~\ref{methodII} on remaining examples.} A precise statement of \eql{D-entanglement}{D-entanglementb} -- again, for polynomial covers, based on genus 0 monodromy -- has all degrees for Davenport pairs in the list \eqref{genus0reduct}. As \cite[\S 8 and \S 9]{Fr99} documents, Hurwitz space properties gave the main arithmetic results about indecomposable Davenport polynomial pairs. 

Also, starting from \cite{Fr73b}, there were many applications to problems involving covers over finite fields with wild ramification (otherwise from Grothendieck's theorem, the result would be the same). To contrast with characteristic 0 phenomena, \cite[Thm.~5.7 and \S 7]{Fr99} shows the genus 0 conclusion limiting the monodromy groups won't hold.  

Though not an expert on the classification, I got the group theory to work.\footnote{I found that few group theorists are expert on the {\sl whole\/} classification.} It hinged on technical discoveries \cite[Lems.~1-5]{Fr73a}.    
\begin{edesc} \label{techgp} \item \label{techgpa} Properties in Prop.~\ref{redRed}:  $G_f=G$, supports a pair of permutation representations $(T_1,T_2)$ attached to a doubly transitive design. 
\item \label{techgpb}  The $n$ cycles in these groups  consist of more than one conjugacy class, even modulo the normalizer in $S_n$ of $(G,T_i)$, $i=1,2$.  
\item \label{techgpc} -1 is not a multiplier in a doubly transitive design \cite[Lem.~5]{Fr73a}. 
\item \label{techgpd} The Davenport covers have no more than three finite (not including $\infty$) branch points \cite[Thm.~1]{Fr73a}. 
\end{edesc} 

We call Cor.~\ref{findeg7polya} and Cor.~\ref{satellitecor} the {\sl Degree 7 Corollaries\/}, is the core of the case for (genus 0 cover) Nielsen classes  of degree $n=7$ whose conjugacy class set contains an $n=7$-cycle. The cases -- as listed in Thm.~\ref{polyPakThm} --  for general $n$, with an $n$-cycle where $T_f$ is a primitive representation of the monodromy group -- can be carried out as in our running example. We expect the details will be illuminating. 
 
\cite[Prop.~4.4]{Fr12} gives a modern proof of \eql{techgp}{techgpc}; \cite[Lem.~5]{Fr73a} used a classical idea hinted to the author by Tom Storer. The upshot was the same: Davenport pairs weren't defined over $\bQ$. I eschewed writing explicit equations. For the list generating \eql{techgp}{techgpb} \cite{CoCa99} used Pari (and \cite{Fr73a}) to write such equations. \cite{DLSc61}, \cite{Le64} and \cite{Sc82} motivated considering these problems.

The distinction between  $K=\bQ$ in \eql{D-entanglement}{D-entanglementd}  and general $K$ \eql{D-entanglement}{D-entanglementc}  alludes to the two halves of the title of \cite{Fr12}. That, with  \cite{FrGS93} and \cite{GMS03}, shows how pieces of the classification -- motivating the Genus 0 problem -- appeared in the service of number theory.   These papers worked by detaching group theory's role. This precluded having to manipulate algebraic equations. Still, work goes on that centers on equation manipulation. 

\cite{AZ01}, \cite{AZ03}, \cite{B99}, \cite{BT00}, \cite{Haj97}, \cite{Haj98} and \cite{BeShTi99} 
continued in that vein, over $\bQ$, but they added finding genus 1 curves with  infinitely many $\bQ$ points.  The last three papers took on specific equations with complicated coefficients, \cite[\S11]{Fr99}. Our general results, however, reduced their problems to showing such particular expressions as  
$$ x(x + a) . . . (x + (k \nm 1)a) = y(y + b) . . . (y + (m \nm1)b).$$ 
were not compositions of polynomials linearly equivalent to Chebyshev polynomials. They used Mazur's famous result on modular curves -- delineating precisely  the small torsion groups possible for elliptic curves over $\bQ$ -- by showing their equations had too many solutions to be given by  torsion. Again, Mazur's theorem was only applicable if equations were over $\bQ$. 

\begin{rem} \cite[\S7.2.3]{Fr12} revisits \cite{AZ03} and \cite{BeShTi99}. Both run into  showing irreducibility of specific fiber products. They didn't use much of \cite{Fr73a} which easily handles their examples.  \cite{AZ03} mistakenly thought \cite{Fr73a} used the simple group classification (see above;  because they based their connection to the problem through \cite{Fr73b}?).\end{rem} 
 
\begin{rem} \cite[Prop.~2.6]{AZ03} has six separated variable polynomial equations with  $
\infty$-ly many $\bQ$ solutions with $f \in \bQ[x]$ and $g(x) = cf(x)$, $c\not = 0,\pm 1$.  How do these examples fit into Prop.~\ref{redRed}? Hint: Use \eql{redGal}{redGalc} to see that $f$ and $cf$ don't have equivalent Galois closure covers. \end{rem} 

\subsubsection{Why consider rational function fiber products?} \label{compmotivation}  My early papers gave motivation for tools like Prop.~\ref{redRed}  phrased as {\sl Davenport\/} and {\sl Hilbert-Siegel Problems\/}. Their point was to understand how different algebraic curve covers relate by separating their variables using Galois Theory. It took advantage of the method by which Siegel proved his most famous theorem: finiteness of quasi-integral points on {\sl any\/} closed curve of genus $>0$. He created separated variables equations giving an equivalent result,  to which he could apply the Thue-Siegel-Roth Theorem. 

Riemann's generalization of Abel's Theorem -- for which he created the distinction between even and odd $\Theta$ functions -- motivated Siegel. He used odd functions for the generalization though he recognized that producing all objects -- including meromorphic differentials of various kinds -- on algebraic curves of genus $\geng$ could use both. This he rephrased using the technical device behind forming canonical $\Theta$ divisors: {\sl half-canonical classes}. 

Denote the divisor class of one of these by $[D]$, then its degree is $\geng\nm 1$, and $[2D]$ is the canonical class.\footnote{Half-canonical classes form a homogeneous space for 2-division points on Jacobians of curves of genus $\geng$. But they are not the same, except when $\geng=1$.}  They divide into two types, even and odd, according to the dimension of their linear systems being even or odd. He needed that {\sl every curve\/} had some of these that were non-degenerate: dimension 0 or 1, respectively, for even or odd. 

He knew those parities were deformation invariants. He found these desired classes by going to the boundary of the space of genus $\geng$ curves, the locus of hyperelliptic curves. This locus has a description where the actual linear system dimensions are constant on components, though only the parity is constant on the moduli of genus $\geng\ > 2$ curves.   

Only from these separated variable equations was he able to stratify their loci to precisely describe the canonical classes with a given dimension to their linear system.  Based on details in \cite{Fa73}, \cite[\S 6.2 and \S B]{Fr10}   produced a version of automorphic functions from even $\theta$-nulls on Hurwitz spaces of odd-branching type, with one proviso. That the Hurwitz space covers contained general genus $\geng$ curves. A generalization would be to decide when nondegenerate, even canonical classes could be found near the boundary of any odd-branching\footnote{That means branch cycles have odd-order.} type Hurwitz space.  An approach would be to find the nature of hyperelliptics accessible from curves on those  Hurwitz spaces.  

Prop.~\ref{fibercomponents} uses the Hurwitz spaces $\sH(G,\bfC,\bT)$  to display components of $\tilde \sC_{f^*,g^*}$ -- when it is reducible -- in families. This puts structure on those components -- equations with interesting properties -- that would otherwise be hidden from view.\footnote{There is a general problem in identifying when a curve of genus $\geng$ identifies as a component of a separated variables equation.}

\section{Branch cycles and Nielsen classes} \label{usebcs} There are many expositions on this 1st part of  Riemann's Existence Theorem (\RET) including \cite[\S 1-4]{Fr12}, \cite[Chap.~5]{Vo96} and \cite[Chap.~4]{FrRET}.

\subsection{Covers/branch cycles in a Nielsen class} \label{cbnc} Consider a degree $n$ cover $\phi: X\to \prP^1_z$, ramified over $\bz=\{\row z r\}$.  Associate to it {\sl classical generators\/}, $\row \sP r$, of the r-punctured sphere $$U_\bz\eqdef \prP^1_z\setminus \{\row z r\},\text{ based at a point $z_0$}.$$ 

We have chosen an ordering of $\row z r$ given by their subscripts. We need a little extra notation on the branch points to avoid ambiguity, so we use $B_f$ for $\{\bz\}$.  These generators are disjoint (piecewise smooth) paths except for the base point, and they issue from that base point in clockwise order according to the ordering on $\bz$. 

To the cover associate {\sl branch cycles\/} $\psigma= (\row \sigma r)$ relative to  $\row \sP r$: $\sigma_i\in S_n$ corresponds to a closed path from $z_0$ going (clockwise) around $z_i$.  \S\ref{usenc} reminds us of the definition of the Nielsen class associated to $(G,\bfC,T)$, the data for a family of covers. 

In dealing with Pakovich's problem, \S\ref{compfamilies} extends the Nielsen class notation to consider the pairs of $\prP^1_z$ covers -- having a nontrivial relation-- giving fiber products with more than one component that arise from Prop.~\ref{redRed}.  Then, assuming the self-normalizing condition on the pairs of representations associated to the covers, an apt name for Prop.~\ref{fibercomponents} would be picturing $\tilde \sC_{f^*,g^*}$ components  in families. 

Finally, \S\ref{coelnicl} gives the coelescing lemma by which we go from the reducibility hypothesis on fiber products to finding families (on the boundary) where the components have genus 0 or 1. 
 
 \subsubsection{Nielsen classes generalize curves of genus $\textbf{g}$} \label{usenc}  
 Basic properties of $\psigma$ (giving {\sl branch cycles\/} for $\phi$);   
 \begin{edesc} \label{RET} \item Generation: Its entries generate $G_\phi\le S_n$ (embedding by $T_\phi$);  
 \item Product-one: $\sigma_1\cdots \sigma_r=1$; and  
 \item Conjugacy classes: Independent of the classical generators, $\psigma$ defines $r$ conjugacy classes (some possibly repeated), $\bfC$, in $G_\phi$.
 \end{edesc} 
 
 \begin{defn} \label{normbfC} For $T: G \to S_n$, and $\bfC$ a conjugacy class collection, denote the the normalizer of $G$ in $S_n$ by  $N_{S_n}(G)$, and the subgroup of $N_{S_n}(G)$ that conjugates $\bfC$ into itself by $N_{S_n}(G,\bfC)$. \end{defn}

Further, $\phi$ (for any ramified cover) produces a canonical permutation representation, $T_\phi: G_\phi\to S_n$, up to conjugation by $G$.  It does this in  2 steps. 
\begin{edesc} \label{prodTphi} \item Produce the Galois closure $\hat \phi:\hat X_\phi \to \prP^1_z$ by the fiber product construction (\S\ref{gcc}). 
\item Take $G(T_\phi,1)$ to be the group giving $X\to \prP^1_z$ as the quotient of $\hat X_\phi/G(T_\phi,1) \to \prP^1_z$.
\end{edesc} 

{\sl Riemann's Existence Theorem\/} (\RET) uses that, given classical generators, $\row \sP r$, elements $\psigma \in G^r$ (replacing $G_\phi$ by $G$), satisfying \eqref{RET}, automatically define a cover $\phi: X\to\prP^1_z$ with $\psigma$ as branch cycles, unique up to equivalence of covers (Def.~\ref{covequiv}). We write $\psigma \in G^r\cap \bfC$ to indicate the conjugacy classes they define. 

\begin{defn}[Nielsen classes] \label{defnc} We say either a cover $\phi$ or any branch cycle description of it is in the Nielsen class $\ni(G,\bfC)$ defined by  \eqref{RET} if $T_\phi$ is equivalent $T$. We call $\ni(G,\bfC)/N_{S_n}(G,\bfC)\eqdef \ni(G,\bfC,T)$ {\sl absolute classes}. {\sl Inner Nielsen\/} classes account for using the Galois closure. Those with branch points $\bz$ correspond one-one to $\ni(G,\bfC)/G$.\footnote{Notice, a single element of a Nielsen class defines the Nielsen class.} \end{defn} 

For a given set of branch points, $\bz$, elements of $\ni(G,\bfC)/N_{S_n}(G)$ correspond one-one  with covers up to (absolute) equivalence.\footnote{Without a choice of classical generators based at the branch points, this correspondence is not canonical.} For any $\sigma\in S_n$, $\ord(\sigma)$ is the least common multiple of its disjoint cycle orders. 
Fix the number $r$ of branch points for the covers we take for our genus computations. 

The genus stays the same if we move the branch points (keeping them separate) to any location we desire. Therefore, we can take all branch cycles relative to a {\sl fixed\/} set of classical generators for any particular situation. Covers of $\prP^1_z$ ramified over $\bz$ correspond to branch cycles computed from those classical generators. 

\begin{princ} \label{nielsenClass} Deforming branch points -- keeping them distinct -- canonically pulls an initial cover uniquely along a trail of covers over the branch point path. So, it  defines (moduli) spaces of cover equivalence classes  in a Nielsen class, $\ni(G(\psigma), \bfC)$: $\bfC$ classes in $G(\psigma)$ defined by the initial cover.  \end{princ} 

Different equivalences  between covers (of $\prP^1_z$) in Princ.~\ref{nielsenClass} define different moduli spaces. Here it is usually {\sl absolute equivalence\/} Def.~\ref{covequiv}: covers are isomorphic by a {\sl continuous\/}\footnote{Thus, automatically analytic by Riemann's removable singularities theorem.}  map commuting with the maps to $\prP^1_z$. 

The index, $\ind(\sigma)$, of a cycle $\sigma\in S_n$ is $\ord(\sigma)\nm 1$. Extend $\ind$ to a product of disjoint cycles additively. Then, the genus,  $\text{\bf g}_f$, of an {\sl  irreducible cover\/} $f: X \to \prP^1_z$ is immediate from branch cycles defining the cover, as is the genus, $\text{\bf g}_{\hat f}$, of the Galois closure $\hat f:\hat X \to \prP^1_z$. See Ex.~\ref{deg7covergenus}. 
\begin{equation} \begin{array}{rl}  \label{rh} 2(\deg(f) + \text{\bf g}_f \nm1) &= \sum_{i=1}^r \ind(\sigma_i) \\
2(|G_f| + \text{\bf g}_{\hat f} \nm 1) &=\sum_{i=1}^r |G_f|/\ord(\sigma_i)(\ord(\sigma_i)-1).\end{array} \end{equation} 
The latter formula uses that $\sigma_i$ in the regular representation of $G_f$ is the product of $|G_f|/\ord(\sigma_i)$ disjoint cycles of length $\ord(\sigma_i)$. 

 Pakovich (Rem.~\ref{orbhist}) uses an {\sl orbifold\/} characteristic for a cover. From branch cycles $\psigma$,  its expression  from \RH\ for $ \geng_{\hat f}$, appears essentially by dividing by $|G_f|$: 
\begin{equation} \label{orb}  \ochar_f=\frac{2(1\nm \text{\bf g}_{\hat f})}{|G_f|}=2 + \sum_{i=1}^r (1/\ord(\sigma_i) - 1). \end{equation} 
Riemann knew this. Still, Rem.~\ref{orbhist} recounts its generalization and prestige tied to (real) 3-manifolds.  

Given branch cycles $\bg$ for one (irreducible) cover $\phi: X\to \prP^1_z$ with representation $T_\phi$ as in \eqref{prodTphi}, interpret elements $\sigma_i$ in  \eqref{rh} as $T_\phi(\sigma_i)$. Consider  any other (faithful, transitive) permutation representation $T_{\phi'}$ of $G_\phi$. 

Canonically produce a new cover $\phi': X'\to \prP^1_z$, quotient of  the Galois closure, $\hat \phi: \hat X\to \prP^1_z$ with branch cycles  $T_{\phi'}(\psigma)$. The genus of this cover comes from \eqref{rh} by replacing $n_\phi$ by $n_{\phi'}$ and each $T_\phi(\sigma_i)$ by $T_{\phi'}(\sigma_i)$. 

\subsubsection{Components in families} \label{compfamilies} Ex.~\ref{pakNiext} extends Nielsen class notation for Pakovich's problem. 
Ex.~\ref{pakNiext} uses $\row x m$ (resp.~$\row y n$) for letters on which $T_1$ (resp.~$T_2$) act transitively, as suggested by Rem.~\ref{extredRed}. 

\begin{defn}[Self-normalizing] \label{selfnorm} Refer to (transitive) $T_1: G\to S_n$ as {\sl self-normalizing\/} if the normalizer  of $G(T_1,x_i)$ in $G$ is itself. Transitivity implies this does not depend on the choice of $i$. \end{defn}

We are after handling pairs of rational functions $(f,g)$ for which $\tilde \sC_{f,g}$ is reducible, indeed, for finding Nielsen classes of such pairs by for which components have genus 0 (or 1), as in our examples. 

Therefore, use Prop.~\ref{redRed}, and extend Nielsen classes to the notation $\ni(G,\bfC,\bT)$ with $\bT=(T_1,T_2)$  interpreted on the common Galois closures of two covers with the same branch points. In earlier  language, $(f^*,g^*)\in \sG_{f,g}$. Below $G$ is the common Galois closure group of the covers for $f^*$ and $g^*$, with $\bfC$ their common conjugacy classes. Then,  $T_{1,i}$ and $T_{2,j}$ are the coset representations defined from those of $f$ and $g$ by the images in $G=G_f=G_g$ of the respective $G(T_f,x_i)$ and $G_g(T_g,y_j)$ given by Def.~\ref{extrep}.\footnote{Rem.~\ref{minGalclos} says this more generally using a simultaneous covering of the Galois closure of $f$ and $g$, and the common quotient of these, $G^\mx$. I see no gain in this.}
 
Map $\pmb \mu\in \ni(G,\bfC)$ into the diagonal of $\ni(G,\bfC)$: $$\delta(\pmb \mu)\eqdef (\row \mu r;\row \mu r)\in \Delta(\ni(G,\bfC)\times \ni(G,\bfC))/G.$$  There are still $r$ conjugacy classes in the diagonal of $G\times G$.  For  $\bz\in U_r$ and $\delta(\pmb \mu)$ as above denote by $N_{i,j}$ the collection \begin{equation} \label{bTnielsen} \{ (\psigma=T_{1,i}(\pmb \mu); \ptau=T_{2,j}(\pmb \mu)).  \end{equation}  

 \begin{lem} \label{covequiv}
 
The Nielsen class of $\psigma\in \ni(G,\bfC,T_1)$  and of $\ptau\in  \ni(G,\bfC,T_2)$) gives a pair  $(f^*,g^*)$ ramified over $\bz$ satisfying the conditions of Prop.~\ref{redRed}. Now we form $\phi_{\psigma,\ptau}: \tilde \sC_{f^*,g^*} \to \prP^1_z$. Then, $(\psigma, \ptau)$ gives a cover $\phi_{i,j}: W_{i,j}\to \prP^1_z$ corresponding to a component of $ \tilde \sC_{f^*,g^*}$. 

The cover equivalences on $\phi_{i,j}$ commuting with the equivalences on the covers given by $f^*$ and $g^*$ are elements of $N_{S_m}(G,\bfC)\cap N_{S_n}(G,\bfC)$ corresponding to the representations $T_1$ and $T_2$.  
\end{lem} 

\begin{proof} As noted in \S\ref{nielsenClass}, fixing the cover corresponding to $\pmb \mu$ requires classical generators of the fundamental group of $\prP^1_z\setminus \{\bz\}$. Now we figure what are the equivalences of the cover given by $\pmb \mu$ that give a cover  of $\prP^1_z$ that is also a Galois closure of both $f^*$ and $g^*$. 

Such an equivalence would induce an automorphism on $G$, that would stabilize both sets of groups \begin{equation} \label{inntoabs} \{G(T_f,x_i)\mid 1\le i\le m\}\text{ and }\{G(T_g,y_j)\mid 1\le j\le n\}.\end{equation}  It would also conjugate the conjugacy classes of $\bfC$. 
So, what we are allowed is $\mu \mapsto h\mu h^{-1}$ with $h\in N_{S_m}(G,\bfC)\cap N_{S_n}(G,\bfC)$. 
\end{proof}

Now consider the implications of Lem.~\ref{covequiv} at the level of Hurwitz spaces. 
Assume  $\sH(G,\bfC,T)=\sH^\abs$ is an absolute Hurwitz space {\sl with fine moduli\/}:  self-normalizing holds for $G(T,1)$ in $G$. Then, there is a unique total family, or {\sl fine moduli\/} structure, for both the absolute and inner spaces, with the inner space constructed naturally by a fiber product construction from the absolute space (\S\ref{gcc}). 

\begin{equation}\label{innabsdiagram} \begin{array}{ccccc} 
\sT^\inn &\longmapright{\Psi_\inn} {50}  &   \sH(G,\bfC)^\inn\times \prP^1_z  &
 \longmapright{} {50} &  U_r\times\prP^1_z\\
\mapdown{\Psi_{\abs,\inn}}&&  \mapdown{\Phi_{\abs,\inn}\times \text{\bf Id}_z}&& \mapdown{\text{\bf Id}_r\times \text{\bf Id}_z} \\
\sT^\abs & \longmapright {\Psi_\abs} {50}& \sH(G,\bfC,T)\times \prP^1_z& \longmapright{} {50}  &U_r\times\prP^1_z.
\end{array} \end{equation}

\begin{edesc} \item The fiber of $\Psi_\abs$ over $\bp\times \prP^1_z$, $\bp' \in \sH(G,\bfC,T)$ represents the class of the covers associated to $\bp'\in \sH(G,\bfC,T)$.
\item Similarly, the fiber of  $\Psi_\inn$ defined by $\hat \bp\in \sH(G,\bfC)^\inn$ lying over $\bp$ gives a Galois cover that appears in Ex.~\ref{pakNiext}.  
\end{edesc}

\begin{exmpl}[Components] \label{pakNiext}   Take  $\psigma=T_{1,i}(\pmb \mu), \ptau = T_{2,j}(\pmb \mu)$ as in \eqref{bTnielsen}. At the level of Hurwitz spaces, suppose $\hat \bp\eqdef \hat \bp(\pmb \mu)\in \sH(G,\bfC)^\inn$ corresponds to the Galois cover, $\hat \phi: \hat X\to \prP^1_z$ associated to $\pmb \mu$, with group identified with $G$. Consider $W_{i,j}\eqdef W_{i,j}(\pmb \mu)$,  \begin{equation}  \label{correspondences} \begin{array}{c} \text{the image of }(\hat \phi_1(i),\hat \phi_2(j)): \hat X\to (\hat X/G(T_1,x_i),\hat X/G(T_2,y_j)).\\  \text{It factors through } \\ \hat X/G(T_1,x_i)=X_i\to \prP^1_z\text{ and }\hat X/G(T_2,y_j)=Y_j\to \prP^1_z.\end{array} \end{equation} That is,  $W_{i,j}$ represents an irreducible component on $\tilde \sC_{\phi_1(i),\hat \phi_2(j)}$.  \end{exmpl} 

\begin{defn} \label{niGCbT} With notation of Ex.~\ref{pakNiext}, denote the set of  \\ $\{(\psigma=T_{1,i}(\pmb \mu), \ptau = T_{2,j}(\pmb \mu))\}/N_{S_m}(G,\bfC)\cap N_{S_n}(G,\bfC)\text{ by }\ni(G,\bfC,\bT).$ \end{defn} 

Prop.~\ref{fibercomponents} assumes self-normalizing holds for both $T_1$ and $T_2$.  Reminder: $r$ is the number of classes in $\bfC$. Rem.~\ref{Hrbasics} reminds of $H_r$ basics. Denote a collection of $u$ representatives of $G(T_1,x_1)$ orbits on $\row y n$ by $J$.

\begin{prop} \label{fibercomponents}  For $j_0\in J$ there is a space $\sW_{1,j_0}$ that fits in a diagram 

\begin{small}\begin{small}
\begin{equation} \label{HSfiberprod} \begin{tikzcd}
%[column sep=small]
&&  \arrow["\pr_x"', lld]  \sW_{1,j_0} \subset \sH(G,\bfC,\bT) \arrow[ddd] \arrow[drr, "\pr_y"]  &&\\ 
\sT(G,\bfC,T_1)\arrow[d, "\Psi_{T_1}"]&&&&\sT(G,\bfC,T_2)\arrow[d, "\Psi_{T_2}"] \\
\sH(G,\bfC,T_1)\times \prP^1_z\arrow[drr] && &   &   \arrow[ lld] \sH(G,\bfC,T_2)\times\prP^1_z 
\\ &&U_r\times \prP^1_z&&\end{tikzcd}  \end{equation} \end{small}\end{small}

with the following properties. 
\begin{edesc} \label{propw1j0} \item \label{propw1j0a} For $\bz\in U_r$, the pullback of $\bz\times \prP^1_z$ on $\sW_{1,j_0}$ is a cover given by branch cycles $(\pmb \sigma;\pmb \tau)$ as in Ex.~\ref{pakNiext}. 

\item \label{propw1j0b}  \eql{propw1j0}{propw1j0a} can be written as a quotient of the Galois cover with branch cycles given by $\pmb \mu$  as in \eqref{correspondences}. 
\item \label{propw1j0c} Branch cycles $\psigma$ (resp.~$\ptau$) give the genus zero cover $f^*$ (resp.~$g^*$) by which we recognize the cover of \eql{propw1j0}{propw1j0b} as a component of $\tilde \sC_{f^*,g^*}$. 
\end{edesc} 

Running over $j_0\in J$, denote the disjoint union of the spaces $\sW_{1,j_0}$  modulo the action of $N_{S_m}(G,\bfC)\cap N_{S_n}(G,\bfC)$ by $\sH(G,\bfC,\bT)$.\footnote{Thus, these representations in the fibers of Nielsen class family don't exceed that given by in Prop.~\ref{redRed}.} 

\begin{edesc} \label{HrgivesniGCbT} \item \label{HrgivesniGCbTa} $H_r$ acts on $\ni(G,\bfC,\bT)$ compatible with its action on $\ni(G,\bfC,T_i)$, $i=1,2$, thus producing $\sH(G,\bfC,\bT)$ as a cover of $U_r\times \prP^1_z$. 
\item  \label{HrgivesniGCbTb}  Also, $\sH(G,\bfC,\bT)$ fits in the diagram \eqref{HSfiberprod} in place of $\sW_{1,j_0}$.\footnote{Similarly, if you only mod out by $G$ (rather than $N_{S_m}(G,\bfC)\cap N_{S_n}(G,\bfC)$) on the $\sW_{1,j_0}\,$s there is a diagram with the union of the results and everything else the same.} 
\end{edesc} 
\end{prop}  

\begin{proof} Much of the proof has already been established, so these are comments relating the formulation of $\ni(G,\bfC,\bT)$ and its corresponding Hurwitz space, $\sH(G,\bfC,\bT)$ as a cover of $U_r$, given by the action of $H_r$. 

The construction of the total space $\sT^\inn \longmapright{\Psi_\inn} {40}     \sH(G,\bfC)^\inn\times \prP^1_z $ from the absolute space (as in \S\ref{gcc}) using the self-normalizing condition has an inverse described by \eqref{inntoabs}. The meaning of \eql{HrgivesniGCbT}{HrgivesniGCbTa} is that the formation of $\sH(G,\bfC,\bT)$ is a compatible analog. 

Equivalences on Nielsen classes always include the action of $G$ (Rem.~\ref{Hrbasics}). Running over the $(i,j)$, there is at most one representative for each component because you can always take $x_i=x_1$ by conjugating by an element of $G$. Since the stabilizer of that first position is then $G(T_1,x_1)$, you have the leeway to conjugate by that group to set $y_j$ to be any particular element in the orbit of $y_j$. 
\end{proof}

\begin{rem}[$H_r$ basics] \label{Hrbasics} Braid orbits on the Nielsen classes correspond to components.  In the action of $H_r$ on any Nielsen class, braids always generate conjugation by $G$. That is, for $g\in G$ there exists $q_g$ whose action on $\psigma$ gives conjugation by $g$ \cite[Lem.~3.8]{BFr82}: we say $g$ is {\sl braided}. As for action of $N_{S_m}(G,\bfC)\cap N_{S_n}(G,\bfC)=N_{m,n}$ on $\ni(G,\bfC,\bT)$, if  $N_{m,n}/G$ is not trivial, these elements may not be braided. 

That means they may correspond to distinct {\sl inner\/} Hurwitz space (Def.~\ref{defnc}) components if you don't mod out by this group. \end{rem}

\begin{exmpl} \label{N1capN2} Examples -- a la Rem.~\ref{Hrbasics} -- of nontrivial $N_{m,n}/G$.  
\begin{edesc} \item \label{N1capN2a} When $G$ is $A_n$, the outer automorphism from $S_n$ usually permutes the stabilising subgroups. 
\item \label{N1capN2b} For $G=\PSL_{n\np1}(\bF_q)$, with representations  on points and hyperplanes and $q=p^e$, $e>1$, the Frobenius acts nontrivially in both. \end{edesc}  
For example in \eql{N1capN2}{N1capN2a}, although \S\ref{doubledegree} is with $G=S_n$, there are similar examples with $G=A_n$. 
\end{exmpl} 

\begin{rem} \label{davcomps} Davenport {\sl polynomial\/} pair examples always had Hurwitz space components with moduli definition field a proper extension of $\bQ$, as revealed by the {\sl Branch Cycle Lemma\/} (exposition in the  opening sections of \cite{Fr10}). In the analog of nonpolynomial cases (don't include $\C_\infty$) as in   $\ni(\PSL_3(\bF_2), \bfC_{2^6},T_i)$, $i=1,2$, this field can be $\bQ$ (Prop.~\ref{gendeg7}).  
\end{rem}

\begin{rem} \label{nonratfunctions} The formulation of the components $\sW_{1,j_0}$ in Prop.~\ref{fibercomponents} doesn't depend on having genus 0 covers in $\ni(G,\bfC,T_i)$, $i=1,2$. For example, Rem.~\ref{davcomps} applies to them, too, though it might seem more striking that those fibers stratify inputs from rational functions. \end{rem} 

\subsubsection{Coelescing Nielsen classes}  \label{coelnicl} Consider the {\sl coalescing operation}. 

\begin{defn}[Restricted Coalescing] \label{defcoalescing} Given an $r$-tuple $\pmb \sigma\in \ni(G,\bfC, T)$, with $\bfC$ consisting of $r$ classes, refer to $\pmb \sigma'=(\sigma_1,\dots, \sigma_{r\nm2}, \sigma_{r\nm1}\cdot \sigma_r)$ as a {\sl simple coalescing\/}.\footnote{It has an algebraic interpretation regarded as going to the boundary of a Hurwitz space, but we don't need that here.} Call it {\sl restricted\/} if  $\lrang{\pmb \sigma'}=G$. 
\end{defn} 

\begin{defn}[Nielsen class coalescing] \label{inddefcoalescing}  Call $\ni(G,\bfC',T)$ a Nielsen class coalescing of $\pmb \sigma\in \ni(G,\bfC, T)$ if $\bfC'$ is the end of a chain, $$\ni(G,\bfC_j,T), j=0,\dots,v, \text{ with }\bfC_0=\bfC\text{ and }\bfC'=\bfC_v,$$ for which there exists a representative $\pmb \sigma_j\in \ni(G,\bfC_j,T)$ that is a simple (restricted) coalescing of a representative of $\pmb \sigma_{j\nm1}\in \ni(G,\bfC_{j\nm 1},T)$. \end{defn} 

\begin{lem} \label{lemcoalescing}  The genus of an irreducible cover represented by branch cycles $\psigma'$ given by a Nielsen class coalescing starting from $\psigma$ satisfies $\geng_{\psigma'}\le \geng_{\psigma}$. 
\end{lem} 

\begin{proof} Use the induction set up by Def.~\ref{inddefcoalescing}. With no loss, show $$\ind(\sigma_{r\nm1}\sigma_r)\le \ind(\sigma_{r\nm1})\np \ind(\sigma_r).$$ Then, from the definition of $\ind$, consider the restriction of $\lrang{\sigma_{r\nm1},\sigma_r}$ to any orbit of it in the representation $T$. That is, regarding the group as a transitive subgroup of $S_n$ ($n$ the length of the orbit). 

Now, apply \RH\ to the triple $(\sigma_{r\nm1},\sigma_r, (\sigma_{r\nm1}\sigma_r)^{\nm 1})$: 
$$2(n\np \geng \nm1) = \ind(\sigma_{r\nm1})\np \ind(\sigma_r) \np  \ind(\sigma_{r\nm1}\sigma_r)^{\nm 1}).$$ The lemma is proved unless $\ind(\sigma_{r\nm1}\sigma_r) > \ind(\sigma_{r\nm1})\np \ind(\sigma_r)$. 

Since $\geng\ge 0$, this implies $2(n\nm1) < 2\cdot \ind(\sigma_{r\nm1}\cdot \sigma_r)$.  As the minimal index of an element is $n\nm1$, this is a contradiction when it is an $n$-cycle.
\end{proof} 

Use the notation of components on $\ni(G,\bfC,\bT)$ of Ex.~\ref{pakNiext}. 
\begin{cor} \label{compcoel} Coalescing $\pmb \mu\in \ni(G,\bfC,\bT)\leftrightarrow\tilde \sC_{T_1(\pmb \mu),T_2(\pmb \mu)}$  to \\ $\pmb \mu'\in \ni(G,\bfC',\bT)\leftrightarrow\tilde \sC_{T_1(\pmb \mu'),T_2(\pmb \mu')}$  is non-increasing on component genuses.\end{cor}
 
Examples of why we need to consider Pakovich's problem include the following general idea.  Items in \eqref{useNCs} reference where illustrating them appear in our examples. 

Start with $\ni(G,\bfC,\bT)$ for which any representative gives a pair of covers $(\phi_X: X\to \prP^1_z, \phi_Y: Y\to \prP^1_z)$ with $\tilde \sC_{\phi_X,\phi_Y}$ satisfying the equal Galois closure and reducibility conditions of \eqref{redGal}. 

Then, use coalescing to move from a general Nielsen class $\ni(G,\bfC,\bT)$ to Nielsen classes that have the same $(G,\bT),$ for which $\bfC$ has changed to give Nielsen class representatives with the following properties.  
\begin{edesc} \label{useNCs} \item \label{useNCsa}  Nielsen classes where $\phi_X$ and $\phi_Y$ are genus zero covers represented by $(f^*,g^*)$ (list of groups and degrees \eqref{genus0reduct}; degree 7 \S\ref{bcdeg7}). 
\item \label{useNCsb} From \eql{useNCs}{useNCsa} coalesce to where representative fiber products, $\tilde \sC_{f^*,g^*}$, have a genus 0 or 1 component (the degree 7 case \S\ref{deg7}).  
\item \label{useNCsc}  Compare \eql{useNCs}{useNCsb} results with Pakovich's use of $\ochar$ (Ex.~\ref{extpakdeg7}).  
\end{edesc}

\subsection{Intransitivity and genus formulas} \label{ncnontrans}  \S\ref{contextdeg7} seriously uses Nielsen classes, but here we use it lightly. It has been standard to assume the natural permutation representation on the group $G$ is transitive. In this paper, we care about when it is not; the representation breaks into a direct sum of transitive representations on the orbits. 

\S\ref{brcycles-fibprod} introduces the notation to calculate with the orbits of $G(T_2,y_1)$ on $\row x m$ that appear in Prop.~\ref{redRed}. \S\ref{genusformCfg} gives the genus formulas that generalize \cite[(1.6) of Prop.~1]{Fr73b}. 

\subsubsection{Using conjugacy classes on a fiber product} \label{brcycles-fibprod}  
With notation of \S\ref{coelnicl} suppose $\phi_X$ (resp.~$\phi_Y$) -- irreducible covers; $T_1$ and $T_2$ are transitive -- has branch cycles  $\row \sigma r$ (resp.~$\row \tau r$).  Denote the subgroup of $G_{\phi_X}\times G_{\phi_Y}$ generated by 
\begin{equation} \label{sigmadottau} \psigma\cdot \ptau \eqdef (\sigma_1,\tau_1),\dots, (\sigma_r,\tau_r) \text{ by } G(\psigma\cdot\ptau). \end{equation} 
Then, denote the subgroup of $G(\psigma\cdot\ptau)$ that stabilizes $y_i$ by $G(\psigma\cdot\ptau, y_i)$.

\begin{edesc} \label{braidfacts} \item  \label{braidfactsa}  Points on $X$ over $z_i$ correspond one-one with disjoint cycles of $\sigma_i$ whose lengths  are the ramification indices of those points. 
\item  \label{braidfactsb}  $G(\psigma\cdot\ptau)$ acts naturally on the tensor product of the permutation representations $T_{\phi_X}$ and $T_{\phi_Y}$. 
\item   \label{braidfactsc}  Components of $\tilde \sC_{\phi_X,\phi_Y}$ identify with orbits of $G(\psigma\cdot\ptau)$ on the symbols $x_i\otimes y_j$, $i=1,\dots,m; j=1,\dots n$.  
\item  \label{braidfactsd} For $O$, a $G(\psigma\cdot\ptau)$ orbit in \eql{braidfacts}{braidfactsc}, branch cycles for the corresponding component $W_O$ are the restriction of $\psigma\cdot\ptau$ to $O$. 
\item \label{braidfactse} $G_{\phi_Y}$ is transitive on $\{\row y n\}$. So orbits of \eql{braidfacts}{braidfactsc} correspond 1-1 to orbits of $G(\psigma\cdot\ptau, y_j)$ on $\{x_1,\dots, x_m\}\otimes y_j$. \end{edesc} 

Our examples usually have $X$ and $Y$ irreducible -- as when those covers are given by a pair $(f,g)$ of rational functions in one variable \eql{goodrat}{goodrata} -- unless otherwise said. Denote points  on $\tilde\sC_{f,g}$ simultaneously over both $x'\in X$ and $y'\in Y$ by $P_{x',y'}$. Let $\sigma_{i,x'}$ (resp.~$\tau_{i,y'}$) be the disjoint cycle in $\sigma_i$ (resp.~$\tau_i$) corresponding to $x'$ (resp.~$y'$). 

Then, $P_{x',y'}$ is nonempty if and only if $f(x')=g(y')$, and it has only one element, unramified over $z'$ unless  one of $s_{x'}\eqdef \ord(\sigma_{i,x'})$ or $t_{y'}\eqdef \ord(\tau_{i,y'})$ exceeds 1.  Consider $\bp$ on  $\tilde \sC_{f,g}$ with image $x'=x'_\bp \in \prP^1_x$ (resp.~$y'=y'_\bp \in\prP^1_y$, $z'=z'_\bp \in\prP^1_z$). 
We compute the precise ramification, $e_{\bp/y'}$, of  $\bp$  over  $y'$. Use similar notation for other ramification indices of a point over its image. 

For  $u,v\in \bZ $,  denote  the greatest common divisor (resp.~least common multiple) of $u$ and $v$ by $(u,v)$ (resp.~$[u,v]$). \cite[Proof of Prop.~2]{Fr74}: Prop.~\ref{fiberram}  results from computing  points and their ramification over $y=0$ on the normalization of $\{(x,y)\mid x^u=y^v\}$.

\begin{prop} \label{fiberram} Notation as above: Consider the points $\bp\in P_{x',y'}$ in $\tilde \sC_{\phi_X,\phi_Y}$ corresponding to the respective disjoint cycles $\sigma_{i,x'}$ and $\tau_{i,y'}$. 

\begin{edesc} \label{rampts} \item \label{ramptsa}  Then, there are $|P_{x',y'}| = (s_{x'}, t_{y'})$ points of $\tilde \sC_{\phi_X,\phi_Y}$ each of ramification index $e_{\bp/y'}=\frac{[s_{x'},t_{y'}]}{t_{y'}}$ over   $y'$.   
\item \label{ramptsb}  $\bp\in P_{x',y'}$ correspond one-one to disjoint cycles in $\sigma_{i,x'}^{t_{y'}}$. 
\end{edesc} 
\end{prop}

\begin{rem} Assume $X$ and $Y$ (as usual) are irreducible, but $\tilde\sC_{f,g}$ has more than one component. Then,  $\bp\in P_{x',y'}$ as in \eql{rampts}{ramptsb} corresponding to the cycles in $\sigma_{i,x'}^{t_{y'}}$ may fall in different components. This is precisely what happens to the two 2-cycles in the 2nd line of \eqref{diffend}. 
\end{rem} 

\subsubsection{Genus Corollaries of Prop.~\ref{fiberram}} \label{genusformCfg} Cor.~\ref{prync} makes more precise statements toward getting branch cycles for $\pr_y: \tilde \sC_{f,g} \to \prP^1_y$, and then for the restriction of $\pr_y$ to each component using the notation of Prop.~\ref{fiberram}. Our applications to components of $\tilde \sC_{f,g}$ will start from Prop.~\ref{redRed}, and then $G(\psigma\cdot\ptau)=G(\psigma)=G(\ptau)=G$. 

We only need \RH\ applied to \eql{rampts2}{rampts2a} and \eql{rampts2}{rampts2b} to get the Genus Corollary statements that follow. Prop.~\ref{ncinterp}  outlines a proof of the more precise statement \eql{rampts2}{rampts2c} on branch cycles for $\pr_y: W \to \prP^1_y$ with $W$ a fiber product component. This is based on \RET\ and the references at the top of \S \ref{usebcs}. We use that to interpret generalizing Thm.~\ref{PakThm} using Nielsen classes. 

With $\psigma\cdot\ptau$ branch cycles for $\tilde \sC_{f,g}$, Prop.~\ref{redRed} corresponds a component $W$ of the fiber product  with an orbit $J$ of $G(\psigma\cdot\ptau,y_1)$   on $\row x m$. Prop.~\ref{fiberram} runs over the disjoint cycles , $\tau_{i,y'}$, of $\tau_i$, and so lists the  points on $Y$ over the $i$th branch point of $\phi_Y$, with their ramification indices in the cover  $\pr_y: \tilde \sC_{\phi_X,\phi_Y}\to \prP^1_y$. 

With $W$ a component of $\tilde \sC_{\phi_X,\phi_Y}$, denote restriction of $\pr_y$ to $W$ by $\pr_W$. 
Cor.~\ref{prync} gives a representative of a branch cycle associated with every point $\tau_{i,y'}$, and then uses these to provide the Nielsen class of $\pr_W$. Denote the Nielsen class of $\pr_y$ by $\ni_{\pr_y}$, and the Nielsen class of $\pr_W$ by $\ni_W$. 
Use the Prop.~\ref{redRed} association $W\leftrightarrow J$, a $G(T_2,y_1)$ orbit on $\{\row x m\}$. 

Cor.~\ref{prync} is directly aimed at Method II, which hits a key point in applying a generalization of Pakovich's Thm.~\ref{PakThm}: Detecting the nature of the covers $\pr_W: W \to \prP^1_y$ with $W$ -- according to the classification of Prop.~\ref{compinc} -- a component of a fiber product $\tilde \sC_{f^*,g^*}$. 

\begin{cor}  \label{prync} Denote $(\sigma_i,\tau_i)\eqdef \gamma_i$. For each $(i,y')$ choose $y_b$ in the support of $\tau_{i,y'}$. Then, choose $h_{y_b,y_1}\in G(\psigma\cdot\ptau)$  that maps $y_b$ to $y_1$. Denote $ h_{y_b,y_1} \gamma_i^{t_{y'}} h_{y_b,y_1}^{\nm1}$ by $g_{i,y_b}$. This fixes $y_1$. Up to left action by $G(T_2,y_1)$, $h_{y_b,y_1}$  doesn't depend on $y_b$. 

Write $g_{i,y_b}$ as $(g_{i,y_b}',g_{i,y_b}'')$ with the first part a conjugate of $\sigma_i^{t_{y'}}$ fixed on $y_1$ and defined up to conjugation by elements of $G(\psigma\cdot \ptau)$ fixed on $y_1$. 
\begin{edesc} \label{rampts2} 
\item \label{rampts2a}  Running over $(i,y',y_b)$ as chosen above, Prop.~\ref{fiberram} gives $g_{i,y',y_b}'$ as a  representive of the branch cycle for $\pr_y$ over the point $y'$.  
\item \label{rampts2b}  Restricting the $g_{i,y',y_b}'\,$s of \eql{rampts2}{rampts2a} to include only disjoint cycles  supported in $J$ gives conjugacy classes of branch cycles for $\pr_W$. 
\item \label{rampts2c}  Including reordering the $g_{i,y',y_b}'\,$s in \eql{rampts2}{rampts2b}, there are choices of $ h_{y_b,y_1}\,$s giving branch cycles for $\pr_W: W\to \prP^1_y$.\footnote{We mean the result is actual branch cycles, not just conjugacy classes, including generation and product-one  of \eqref{RET}.} 
 \end{edesc} 
\end{cor} 

\begin{proof} Denote  $(\sigma_i,\tau_i)$ by $\beta$. Suppose we choose $y_{b'}$ in the support of $\tau_{i,y'}$ instead of $y_{b}$. Then, a power, say $m$, of $\tau_i$ maps $y_{b'}$ to $y_{b}$. Form  $h_{y_b,y_1}\beta^m$. It maps $b'$ to $y_1$, and conjugating $\beta$ by it has exactly the same effect as conjugating by $h_{y_b,y_1}$. 

From Prop.~\ref{fiberram}, the result is an element of $G(\psigma\cdot\ptau)$ that fixes $y_1$, and in its first slot gives a conjugate of $\sigma_i^{t_{y'}}$ that is a branch cycle for $\pr_y$ for the branch point $y'$.  Thus, \eql{rampts2}{rampts2a} follows because we are running over the collection of branch points in the cover $\pr_y$. Then,  \eql{rampts2}{rampts2b} amounts to restricting to disjoint cycles that correspond to points in $W$. 

Finally, \eql{rampts2}{rampts2c} follows from \RET: There exists conjugations by the stabilizer of $y_1$ for each element listed in \eql{rampts2}{rampts2b} so you get an element of the Nielsen class of the cover $\pr_W$. See Rem.~\ref{hrtrans}.\end{proof} 

\begin{cor}[Method I]  \label{methodI}  Take a component $W$ of $\tilde \sC_{f,g}$, but use the permutation representation $T_W$ of $G(\psigma\cdot\ptau)$ for the cover $W\to \prP^1_z$. Compute the permutations $T_W(\sigma_1,\tau_1),\dots,T_W(\sigma_r,\tau_r)$. Then, the genus $\text{\bf g}_W$ of $W$  is computed from \RH:  
$$2(|\deg(W/\prP^1_z)| \np \geng_W \nm 1) = \sum_{i=1}^r \ind(T_W(\sigma_i,\tau_i)). $$
\end{cor} 

The first case of Cor.~\ref{methodII} is  included in the general case $u\ge 1$; both cases count ramification over a particular component of the fiber
product running over each point on $Y\to \prP^1_z$ that contributes nontrivially to the ramification. If $u=1$, running over all disjoint cycles  is equivalent to running over all ramified points in the unique component. 

\begin{cor}[Method II; See Rem.~\ref{congClassdep}] \label{methodII}  If $ \tilde \sC_{f,g}$ is irreducible, then the genus, $\text{\bf g}_{f,g}$, of its unique  component  is given is in  \cite[(1.6) of Prop.~1]{Fr73b}:  
\begin{equation} \label{genx-ytoy} \begin{array}{rl} &2(\deg(f) + \text{\bf g}_{f,g} \nm 1)= \sum_{\bp\in \sC_{f,g}}e_{\bp/y'} \nm 1\\
=&\sum_{\text{branch points z' of }f}  \sum_{(x',y') \mapsto z'} (s_{x'},t_{y'})\Bigl(\frac{[s_{x'},t_{y'}]}{t_{y'}}-1\Bigl).\end{array}  \end{equation} 
When $u\ge 1$, for $1\le u' \le u$, using the correspondence $W_{u'} \leftrightarrow J_{u'}$ to get the disjoint cycles, $\beta^*$, of all the elements of \eql{rampts2}{rampts2b} (on $J_{u'}$). That gives the analog of  \eqref{genx-ytoy} to compute the genus $\geng_{W_{u'}}$ of $W_{u'}$. 
\begin{edesc} \item Left side of \RH: $2(\deg(W_{u'}/\prP^1_y)\np \geng_{W_{u'}} \nm1)$. 
\item Right side of \RH: Sum over $\ind(\beta^*)$.  \end{edesc} 

Consider a Nielsen class $\ni(G,\bfC,\bT)$ with $$\bT=(T_1, T_2), \text{  both } \text{  transitive faithful representations of }G.$$ Assume  $\ni(G,\bfC,T_i)$, $i=1,2$, are both Nielsen classes of genus 0 covers. Then, the genuses of the components in  \eqref{genx-ytoy} depend only on  $\ni(G,\bfC, \bT)$; not on branch cycles for the particular $(f,g)$. 
\end{cor}

Cor.~\ref{2-foldfiber} is a particular case of  Cor.~\ref{methodII}: The hypothesis implies that all ramification must appear in $W_2$, and the contribution to the sum where $s_{x''}=s_{x'}$ is automatically trivial. 

\begin{cor} \label{2-foldfiber} For, $f\in \bC(x)$, $T_f: G_f\to S_m$ is doubly transitive if and only if $\tilde \sC_{f,f}$ has the diagonal and one other component $W_2$ of degree $m\nm 1$ over $\prP^1_y$. If so, then the genus, $\geng_{W_2}$ of $W_2$, satisfies: 
$$ \begin{array}{rl}  &2(m\nm1\np \geng_{W_2} \nm1)\\ = &\sum_{\text{branch points z' of }f}\sum_{(x',x'', x'\ne x'') \mapsto z'} (s_{x'},s_{x''})\Bigl(\frac{[s_{x'},s_{x''}]}{s_{x'}}-1\Bigl).\end{array}$$ \end{cor}

\begin{rem} \label{congClassdep}
The last statement in Cor.~\ref{methodII} (dependency on only the Nielsen class) follows because the inclusion of particular $\beta^*\,$s for counting a particular component depends only on conjugacy classes  in $\bfC$. \end{rem} 

\begin{rem}[$T_f$ double transitive in Cor.~\ref{2-foldfiber}] Note:  $T_f$ is automatically primitive if it is doubly transitive. Examples of this include the case when $f\in \bC[x]$ (polynomial) is indecomposable, and neither a cyclic nor Chebychev polynomial \cite[Thm.~1]{Fr70}. The genus zero problem is easier but not trivial if we assume $f$ is doubly transitive. An easy characterization (as in the polynomial case, Rem.~\ref{subdegrees}) for double transitivity is unlikely. \end{rem} 

\begin{rem}[$H_r$ transitivity] \label{hrtrans} Elements of \eql{rampts2}{rampts2c} define the Nielsen class associated to the correspondence $W\leftrightarrow J$. If $r'$ is the number of branch points of the covers $\pr_W: W\to \prP^1_y$, then transitivity of $H_{r'}$ allows listing all elements in the Nielsen class of $\pr_W$, from identifying one such element from Cor.~\ref{prync} applied to the Nielsen class of $\tilde \sC_{\phi_X,\phi_Y}\to \prP^1_z$ as noted in \eqref{Schreier}.  

If, though, $H_{r'}$ is not transitive, this is a subtler problem that arises in each component type given by Prop.~\ref{compinc}. While that includes those for which $\pr_W$ has genus or 1 Galois closure, those Nielsen classes usually are recognized and known to have $H_{r'}$ transitivity as in \S\ref{orbzero}.   \end{rem} 

\section{Example genus component computations} \label{compgenuses}  Cors.~\ref{methodI} and \ref{methodII} have our genus computing results  for  fiber products of components of nonsingular covers.  This section shows  how {\sl one\/} Nielsen class can provide several example challenges  to extending Pakovich's Theorem.    

We complement  \cite[Ex.~5, p.~246]{Fr74} to produce examples of Nielsen classes with elements giving $\tilde \sC_{f^*,g^*}$ reducible, and then having genus 0 or 1 components. \S\ref{genus0prob} explains they are the first of a set of six examples in \eqref{genus0reduct}, all from groups with their core a projective linear group. 
 Expository sections \cite[\S1-4]{Fr12}, stemming from \cite{Fr73a}, document the literature.  
 
These examples arise from coalescing branch cycles from {\sl one\/} Nielsen class with $G=\PSL_2(\bZ/2)$ and explicit  classes $\bfC$ in $G$, Ex.~\ref{exgenDeg7}. 

 \subsection{Displaying examples of $g\in \sR_f$} \label{irreducibilityb}   Use the notation of Prop.~\ref{redRed} where  $(f^*,g^*)$ have the same Galois closures, or $(T_1,T_2)$ are faithful, transitive representations of the same group $G$, and consider the Nielsen class $\ni(G,\bfC,\bT)$. 
 
 In keeping with Pakovich's problem, \S\ref{noirred} 
considers   $\tilde \sC_{f^*,g^*}=\cup_{j=1}^u W_j$ with $u>1$ components (so the irreducibility hypothesis of Thm.~\ref{PakThm} does not hold), and we can control comparing the components of  $\tilde \sC_{f^*,g^*}$ and $\tilde \sC_{f^*,g^*\circ g_1}$.

Subsections follow the progression in \eqref{useNCs}. \S\ref{bcdeg7} displays  Nielsen classes $\ni(G,\bfC,\bT)$ whose representative $\tilde \sC_{f^*,g^*}$ are reducible, and $f^*$ has a totally ramified place (so including polynomial covers). Review Comment \eqref{quotreps} for why among the possibilities offered by Prop.~\ref{redRed} we have chosen as our running example the progenitor that shaped the genus 0 problem: Otherwise, we hardly know how to name groups and their permutation representations. 

\S\ref{7branchcyclesa} explains that both sets of Nielsen classes in \S\ref{bcdeg7} come by coalescing from one more general Nielsen class with $r=4$ branch points. \S\ref{deg7} computes the genus of the components.

\subsubsection{Example Nielsen classes} \label{bcdeg7}  
 Use the notation $({}_1f^*,{}_1g^*)$ and $({}_2f^*,{}_2g^*)$ for  respective pairs of degree 7 polynomials in these two Nielsen classes.  Each polynomial has three branch points (including $\infty$) which we can -- if we desire -- take to be  $0,1,\infty$ up to linear transformation (fixing $\infty$) on $\prP^1_z$. 

The  monodromy groups are  $\PSL_2(\bZ/2)=G_{{}_jf^*}=G_{{}_jg^*}$, $j=1,2$, corresponding to the inequivalent representations on points ($T_P$; for ${}_jf^*$) and on lines  ($T_L$; for ${}_jg^*$) in the dimension 2 projective plane $\prP^2(\bZ/2)$. 

Designate the permutation representations  by $T_{{}_jf^*}$ (resp.~$T_{{}_jg^*}$) on the  letters $\row x 7$ with the branch cycles given by  permutations ${}_j\sigma_i$ (resp.~using letters $\row y 7$ in permutations ${}_j\tau_i$). 

Example branch cycles in the Nielsen classes correspond to (right) subscripts 1 and 2 on the $\sigma\,$ and $\tau\,$s, as  in \cite[(2.42)]{Fr74}  and \cite[(2.41)]{Fr74}.  \begin{small}
\begin{equation} \label{bcfg} \begin{array}{rlrl} T_{{}_1f^*}: {}_1\sigma_1=&(x_1\,x_3)(x_4\,x_5) & {}_1\sigma_2=&(x_1\,x_4\,x_6\,x_7)(x_2\,x_3) \\ T_{{}_1g^*}: {}_1\tau_1=&(y_1\,y_2)(y_3\,y_5) &
{} _1\tau_2=& (y_1\,y_3\,y_6\,y_7)(y_4\,y_5) \\  T_{{}_2f^*}: {}_2\sigma_1=&(x_1\,x_2\,x_3)(x_4\,x_5\,x_7) & {}_2\sigma_2= &(x_1\,x_4)(x_6\,x_7) \\  T_{ {}_2g^*}: {}_2\tau_1=&(y_1\,y_2\,y_7)(y_3\,y_5\,y_6) & {}_2\tau_2= &(y_3\,y_7)(y_4\,y_5)\end{array} \end{equation} \end{small}

We indicate the cycle notation with integers when it is clear whether it is $x$ or $y$ involved. For example, we have normalized each representation up to one conjugation in $S_7$ so that the branch cycle at $\infty$ is always the 7-cycle $(1\,2\,3\,4\,5\,6\,7)^{-1}=\sigma_\infty=\tau_\infty$.  
\S\ref{7branchcyclesa} explains a potential confusion from doing this, despite its computational advantage. Just remember that the $\psigma$ entries act on different letters than the $\ptau$ entries. 

Now we name the Nielsen classes using the orders of their elements: $$\begin{array}{c} ({}_1\sigma_1,{}_1\sigma_2,{}_1\sigma_\infty;{}_1\tau_1,{}_2\tau_2,{}_1\tau_\infty)\in \ni(\PGL_2(\bZ/2),\bfC_{\infty\cdot2\cdot4},\bT) \\
 ({}_2\sigma_1,{}_2\sigma_2,{}_2\sigma_\infty;{}_2\tau_1,{}_2\tau_2,{}_2\tau_\infty)\in \ni(\PGL_2(\bZ/2),\bfC_{\infty\cdot2\cdot3},\bT) \end{array} $$

The involution conjugacy class below will be referenced as $\C_2$. The product, in order, of the branch cycles -- as always -- is 1, where the 3rd cycle is $(1\,2\,3\,4\,5\,6\,7)^{-1}$. 
 Permutations act on the right of the letters: The result of $\sigma$ on $i\in \{1,\dots,7\}$ is $(i)\sigma$, so $$T_{f^*}({}_j\sigma_1)T_{f^*}({}_j\sigma_2) = (1\,2\,\dots\,7),\  j=1,2, \text{ etc.}.$$ 

Given these choices, the only leeway is conjugation by the branch cycle at $\infty$. That amounts to a cycling, $(1\,2\,3\,4\,5\,6\,7)$ of $\{1,2,3,4,5,6,7\}$.  

\subsubsection{Essential differences between the $\psigma\,$s and $\ptau\,$s}   \label{7branchcyclesa} 

Consider a pair $(f^*,g^*)$ of degree 7 rational functions that give examples of simultaneous branch cycles for $\PSL_2(\bZ/2)$ in the respective representations. In the identification of $G_{f^*}$ and $G_{g^*}$ as subgroups of $S_7$, we have simplified in two ways: Dropping the $x\,$s and $y\,$s in the notation, and then identifying the 3rd terms of the $\psigma\,$s and $\ptau\,$s in \eqref{bcfg} with $(1\,2\,3\,4\,5\,6\,7)^{-1}$. 

The actual relation between the two Galois closures is preserved by continuing the embedding of $G_{f^*}\to S_7$ to $\GL_7(\bQ)$, and then conjugating by an incidence matrix $I_{f^*,g^*}$ (\cite[Proof of Thm.~1]{Fr73a} or \cite[\S4.2]{ Fr12}) that produces the representation $T_{g^*}$ in $\GL_7(\bQ)$. The latter, in this case, happens to conjugate $S_7$ into itself, though $I_{f^*,g^*}$ is not a permutation matrix. 

What is happening in these groups is that they have degree $n$ representations and $n$-cycles $\sigma_n$. The conjugacy class of the $n$-cycle, $\C_n$. 

\begin{defn}[multipliers] \label{multdef} The {\sl multipliers\/} attached to $\C_n$ are integers $u$, with $(u,n)=1$ for which $\sigma_n^u\in \C_n$. The quotient $(\bZ/n)^*/\sM$ with $\sM$ the group of multipliers of $\C_n$ is nontrivial if and only if there is more than one conjugacy class of $n$-cycles. \end{defn} Indeed, apply this to $\sigma_\infty$. The result is that the conjugacy class of $\sigma_\infty^{-1}$ is different from that of $\sigma_\infty$. That is,  $-1$ is a non-multiplier of the design, or that $\sigma_\infty$ and $\sigma_\infty^{-1}$ (albeit with the same cycle type) are not conjugate in the image permutation group of $T_f$, but $I_{f,g}$ does conjugate them.  

Similar statements apply to the other degrees $m$ in \eqref{genus0reduct}, though for those, the quotient of $(\bZ/n)^*$ by the multipliers is a larger group. For example, see Degree 13 in the comments following Thm.~\ref{polyPakThm}. 

The most general Nielsen class that contains degree 7 polynomial covers, $f$ -- we denote it $\ni(\PSL_2,\bfC_{ \infty\cdot 2^3}, \bT)$ (note the $\bT$ hasn't change) for which we get reducible $\tilde\sC_{f,g}$ with branch cycles of type $$((2)(2),(2)(2),(2)(2),(7)); $$ the integers here between ()'s are cycle lengths. Most general means the largest number of branch points (elements in $\bfC$). 

As above, there are two Nielsen classes differing between the two conjugacy classes of 7-cycles in this group. They come together in the tensor product of the representation on points and lines as in \S\ref{ncnontrans}.  The design mentioned above clarifies that the most possible fixed points for an element $M$ in this group are 3. The fixed points in that extreme case correspond to 3 points on the projective plane lying on a line $L$. \RH\ shows that to get such branch cycles for a genus 0 cover requires the three elements that aren't 7-cycles have the cycle type $(2)(2)$ (of index 2) above. 

Such an element $M\in \PSL_2(\bZ/2)$ is a transvection: having the form $$M: \bv \rightarrow \bv \np \phi(\bv)\bh $$ where $\phi$ is the linear functional with the points of $L$ in its kernel, and $\bh$ is also in $L$. In, however, Ex.~\ref{exgenDeg7} we find that for elements in these Nielsen classes the components of $\tilde \sC_{f,g}$ have genus exceeding 0 (actually 1). 

The cycle-types of ${}_1\psigma$ and ${}_2\psigma$ are different. So it is no surprise in \S\ref{deg7} their contributions generalizing Thm.~\ref{PakThm} have a significant difference, as  in Cor.~\ref{findeg7polya} and \ref{findeg7polyb}.  \S\ref{pakgoal} explains the following statements. 

Reordering the conjugacy classes in $S_7$ of the entries of these branch cycles is a minor change. Still, have we left out some significant 3-entry branch cycles that come from coalescing of the main Nielsen class consisting of 4-tuples? Cor.~\ref{satellitecor} shows we have not. \S\ref{sigH7} explains -- under the rubric that of all such coalescing giving Nielsen class {\sl satellites\/} come  one fixed Nielsen class -- the value of all these degree 7 examples cohering. 

\begin{rem}[Other degrees in \eqref{genus0reduct}] The case $n=7$ has several regularities that might be misleading. Example: In  Cor.~\ref{findeg7polyb} we take advantage that  three branch cycles for covers in the main Nielsen class are transvections. That isn't the case for $n=13$ and 15 \cite[\S3.4]{Fr05b}. \cite[\S4.3, Thm.~4.5]{Fr12} notes the missing difference set for $n=15$ in \cite[p.~134]{Fr73a}. It also shows precisely how Fried (in 1969) could write branch cycles for the Nielsen classes of the other examples. Thereby, modulo the conjecture,  verified later that we knew all the groups with such doubly transitive designs, displaying the precise list of degrees for Davenport pairs. 
\end{rem}

\subsubsection{Degree 7 example components: Cors.~\ref{methodI} and \ref{methodII}} \label{deg7} 
The distinction between Cor.~\ref{methodI} and Cor.~\ref{methodII}: The former computes the genus from an orbit of $x_i\otimes y_1$ under the group $G(\psigma\cdot\ptau)$ (as in \eqref{sigmadottau}).  The latter computes it from the orbit, under the smaller group $G(\psigma\cdot\ptau,y_1)$ stabilizing $y_1$.  

While the genus is of the same component (so, the same), the two methods differ in what they reveal.  Similarly, the two approaches play different roles in Prop.~\ref{compinc}. 

Lem.~\ref{invorbits} has two uses: It makes sense of how coalescing ties together many examples. Also, it shows  how conceptual ideas drove the simple group classification and the genus zero problem. Out labeling of orbits is consistent with Comments \eqref{proofuse}.
\begin{lem} \label{invorbits} 
Orbits of $G(\psigma\cdot\ptau,y_1)$ for  $({}_jf^*,{}_jg^*)$, for either $j=1,2$, are: 
\begin{equation} \label{staby}  J_{1}=\{x_1,x_2, x_4\}\otimes y_1\text{ and } J_{2}=\{x_7,x_3,x_5,x_6\}\otimes y_1.  \end{equation}  
Indeed, you can figure the orbits from the conjugates of the 7-cycle on an involution in the conjugacy class $\C_2$.
\end{lem}

\begin{proof} The group is generated by any 2 of the 3 branch cycles. So, we only need $\sigma_\infty$ and an element in $\C_2$. Use conjugation by $\sigma_\infty^t $ on $\tau_1$ for various $t\,$s that fix $y_1$, and check what orbits on the $x_i\,$s result. 
\begin{small}
\begin{equation} \label{conjbysiginfty} \sigma_1\mapsto (x_2\,x_4)(x_5\,x_6) (t=1); (x_3\,x_5)(x_6\,x_7) (t=2); (x_5\,x_7)(x_1\,x_2) (t=4).\end{equation} You can read the orbits $J_1$ and $J_2$ off the resulting three elements. \end{small} 
\end{proof}

In Lem.~\ref{invorbits}, $u=2$. Use $1\le u'\le u$ as the index for referring to a component. Denote the genus in each example computation  by ${}_j\text{\bf g}_{u'}$. Example:   ${}_2\text{\bf g}_{1}$ for the component associated with $J_1$ and polynomials $({}_2f,{}_2g)$. 
To use  Cor.~\ref{methodI}, compute the orbit of $x_1\otimes y_1$ under the action of $G(\psigma\cdot \ptau)$. 

Start with ${}_1\psigma=\psigma$ and ${}_1\ptau=\ptau$ in \eqref{bcfg}, dropping the pre-subscript ${}_1$ to simplify. Apply $(\sigma_1,\tau_1)$ (resp.~$(\sigma_2,\tau_2)$) to $x_1\otimes y_1$  (resp.~$x_3\otimes y_3$) to get 
\begin{equation} x_3\otimes y_2 \ (\text{resp. }x_2\otimes y_6). \end{equation} That alone gives us the expected orbit of length 21 by applying $(\sigma_\infty,\tau_\infty)$ to the subscripts, equivalencing them $\!\!\!\mod 7$: 
\begin{equation}  x_{1\np i}\otimes y_{1+i}, x_{3\np i} \otimes y_{2\np i}, x_{2\np i}\otimes y_{6\np i}, i=0,\dots,6. \end{equation} 

In abbreviated notation  the first two branch cycles give  
\begin{small} $$\begin{array}{rl} (\sigma_1,\tau_1): &\!\!\!\!((1,1)\,(3,2))((2,2)\,(2,1))((3,3)\,(1,5))((4,4)\,(5,4))((5,5)\,(4,3))\\
&\!\!\!\!\!\!\!\!\!((6,5)\,(6,3))((1,7)\,(3,7))((4,1)\,(5,2)) \\
(\sigma_2,\tau_2):&\!\!\!\! ((5,5)\,(5,4))((3,2)\,(2,2)) ((1,1)\,(4,3)\,(6,6)\,(7,7)) \\
&\!\!\!\!\!\!\!\!\!((7,6)\,(1,7)\,(4,1)\,(6,3))((1,5)\,(4,4)\,(6,5)\,(7,4))((2,6)\,(3,7)\,(2,1)\,(3,3)). \end{array}$$\end{small} 

Indices of $(\sigma_1,\tau_1), (\sigma_2,\tau_2),(\sigma_\infty,\tau_\infty)$ are the 3 numbers -- in order -- on the right  of \eqref{1O11}.  \RH\ applies to the degree 21 cover of $\prP^1_z$ to compute ${}_1\text{\bf g}_{1}$:  
\begin{equation} \label{1O11} 2(21 +{}_1\text{\bf g}_{1}-1)=8 + 14+ 18, \text{ or }{}_1\text{\bf g}_{1}=0 .\end{equation}

Now we use Method II, Cor.~\ref{methodII}. to compute ${}_1\text{\bf g}_{1}$ the genus of the degree 3 component, ${}_1W_1$ of the fiber product by applying \RH\ to $\pr_y: {}_1W_1\to \prP^1_y$. 

The algorithm to produce the branch cycles of the cover $({}_1\sigma_1,{}_1\tau_1)$ in Cor.~\ref{prync} immediately shows we have branch points (with nontrivial branch cycles) corresponding in ${}_1\tau_1$ with the length 1 cycles $y_4,y_6,y_7$. Powers of $\sigma_7=\sigma_\infty$ translate the subgroups and give us the respective branch cycles, as elements in $G(T_2,y_1)$ from \eql{rampts2}{rampts2b}. 

As in Lem.~\ref{invorbits}, use $\sigma_\infty$ to translate the subscripts to change $y_i\,$ to $y_1$. Then check which cycles end up with their support in $J_{1}$. Example: 
$$\begin{array}{rl} (x_4\otimes y_4\, x_5\otimes y_4)& \mapsto (x_1\otimes y_1\, x_2\otimes y_1) \text{ \rm while } \\ (x_1\otimes y_4\,x_3\otimes y_4)&\mapsto (x_5\otimes y_1\,x_7\otimes y_1).\end{array}$$ 
In \eqref{exonea}, cycles with superscript $^*$ have  support in $J_{1}$. The rest are in $J_{2}$. 

\begin{equation} \label{exonea}  \begin{array}{rl} (x_1\otimes y_4\,x_3\otimes y_4)^{\ }&(x_4\otimes y_4\, x_5\otimes y_4)^* \\
 (x_1\otimes y_6\,x_3\otimes y_6)^{\ }&(x_4\otimes y_6\, x_5\otimes y_6) \\
  (x_1\otimes y_7\,x_3\otimes y_7)^*&(x_4\otimes y_7\, x_5\otimes y_7). \end{array} \end{equation}

Now do the computation where the branch cycle is 
$({}_1\sigma_2,{}_1\tau_2)$. The 4-cycle in ${}_1\tau_2$ contributes nothing to ramification, but the fixed point ($y_2$)  and the 2-cycle do. For the 2-cycle we may use either $v=4$ or 5 in its support in \eql{rampts2}{rampts2b} after putting the 4-cycle in $\sigma_2$ to the power 2. 

We choose $v=4$. As above, list those contributions with the $^*$ superscript from translating the subscripts to see which  cycles end up in  $J_{1}$. 

\begin{equation} \label{diffend}  \begin{array}{rl}  (x_1\otimes y_2\,x_4\otimes y_2\,x_6\otimes y_2\,x_7\otimes y_2)&(x_2\otimes y_2\,x_3\otimes y_2)^*\\
(x_1\otimes y_4\,x_6\otimes y_4)&(x_4\otimes y_4\,x_7\otimes y_4)^*. 
\end{array} \end{equation}  

Both branch cycles over the 7-cycles are 7-cycles; they contribute nothing to ramification. Denote cycles with $^*\,$ superscripts by $\mu^*$.  Compute ${}_1\text{\bf g}_{1}$:  
$$2(3+{}_1\text{\bf g}_{{1}} -1)=\sum_{\mu^*} \ind(\mu^*)=4, \text{ or } {}_1\text{\bf g}_{{1}} =0.$$ 

Finally, use Method II for orbit $J_1$  (in \eqref{staby}) of $G(\psigma\cdot\ptau,y_1)$ for  $({}_2f,{}_2g)$.     
 From the branch cycles in \eqref{bcfg} the analogs of \eqref{exonea} and \eqref{diffend} are 
\begin{equation} \label{bcfg2} \begin{array}{rl} (x_1\otimes y_4\,x_2\otimes y_4\,x_3\otimes y_4)&(x_4\otimes y_4\,x_5\otimes y_4\,x_7\otimes y_4)^*\\ & \text{ and } \\
(x_1\otimes y_1\,x_4\otimes y_1)^*&(x_6\otimes y_1\, x_7\otimes y_1)\\
(x_1\otimes y_2\,x_4\otimes y_2)^{\ }&(x_6\otimes y_2\, x_7\otimes y_2) \\
(x_1\otimes y_6\,x_4\otimes y_6)^{\ }&(x_6\otimes y_6\, x_7\otimes y_6)^*.\end{array}\end{equation} 

We have already put the $^*\,$s on the translates in the right orbit when all the $y_i\,$s are set back to $y_1$.  As above we compute ${}_2\text{\bf g}_{1}$ by summing the indices of the $\mu^*\,$s (which is 4 again), to get ${}_2\text{\bf g}_{{1}}=0$. 

Starting with Ex.~\ref{extpakdeg7}, then going to  Rem.~\ref{comporbit} and Ex.~\ref{exgenDeg7}, we illustrate how Thm.~\ref{genPakThmGen} works. 

\begin{exmpl}[Pakovich Goal] \label{extpakdeg7} The two conjugacy class sets $\bfC_{\infty\cdot2\cdot4}$ and $\bfC_{\infty\cdot2\cdot3}$, and the labeling for each representative  ${}_j\bfC$, $j=1,2$ corresponding to the two permutation representations, appear in \S\ref{bcdeg7} (given by \eqref{bcfg}). The difference in the values of $j$ comes from the outer automorphism (\S\ref{7branchcyclesa}) of $\PSL_2(\bZ/2)$ that takes the letters of the representation $T_1$ on points of the projective space  to the letters of the representation on lines. Consider the pairs $(f^*,g^*)$ in either Nielsen class. Each $\tilde \sC_{f^*,g^*}$ has two components, $W_{u'}$,  labeled by the orbits $J_u'$, $u'=1,2$ in Lem.~\ref{invorbits}.  

Use the notation of Prop.~\ref{redRed} (as in Comments \eqref{proofuse}) to label the degrees of the natural projections of the $W_{u'}$, $u'=1,2$, to $\prP^1_y$: They are respectively $\ell_1=3$ and $\ell_2=4$ corresponding to the lengths of the orbits $J_1$ and $J_2$. 

In going from $\tilde \sC_{f^*,g^*}$ to $\tilde \sC_{f^*,g^*\circ g_1}$ in each of these cases where the components come out genus 0 or genus 1, Prop.~\ref{compinc} lists Nielsen class that might contain $g_1$ for which the genus of components on $\tilde \sC_{f^*,g^*\circ g_1}$ remains stable. Situations \eql{fail1}{fail1b} and \eql{fail1}{fail1c} force recognizing them through either our genus formula, or as belonging to a decomposition variant (Def.~\ref{decvar}). \end{exmpl}

\begin{rem} \label{comporbit} For the complementary orbit of $G(\psigma\cdot\ptau,y_1)$, $J_{2}$ in \eqref{staby}, use the analogous formula for the genus $\text{\bf g}$ in either case.  Then, the left side of \RH\ is $2(4 + \text{\bf g}-1)$ and the right side is the sum of the indices of the $\mu\,$s that {\sl don't\/} have a ${}^*$ superscript. 

In the resp.~cases this sum is 4+4 and 6: genuses ${}_1\text{\bf g}_2$ and  ${}_2\text{\bf g}_2$ of the complementary orbits are, resp.~1  and 0. That gives three  genus 0 components and one genus 1 component on reducible fiber products of degree 7 polynomial covers.   \end{rem}

\subsection{Context of \S\ref{deg7}} \label{contextdeg7}  \S\ref{pakgoal} starts from a Nielsen class $\ni(G,\bfC,\bT)$ whose representatives are reducible covers of $\prP^1_z$ of the form $\tilde \sC_{f^*,g^*}$ more general than the Nielsen classes  of \S\ref{irreducibilityb}. We refer to the latter as {\sl satellites\/} of it. The satellites come from {\sl coalescing}, using Cor.~\ref{compcoel} to assure that the genus of components does not increase when coalesced. We follow the path of \eqref{useNCs} to where we end at the two Nielsen classes of \S\ref{deg7}.

The topic of \S\ref{braid-coalesce}  is that the Nielsen classes form a natural space with their braid action. In each of our starting cases, with conjugacy classes $\bfC=\bfC_{\infty,2^3}$ and $\bfC_{2^6}$, but $G$ is still $\PSL_2(\bZ/2)$, there is just one component (of respectively four and six branch point covers), making them natural to find target Nielsen classes as satellites. 

Then, \S\ref{genus0prob} returns to the genus 0 problem to expand from our running example to situations that already have a place in the literature. The solution of the genus 0 problem and its extensions shows how much it was influenced by using aspects of Nielsen classes.  Further, those are references for finding permutation pairs that arise in applying Prop.~\ref{redRed}. 

\subsubsection{Coalescing and the Pakovich Goal}  \label{pakgoal}  Ex.~\ref{extpakdeg7} showing, by examples,  how  we would drop the irreducibility hypothesis of Thm.~\ref{PakThm}. 

Here is the source of 7-cycles  in $\PGL_3(\bF_2)$. Regard projective three space over $\bF_2$ as $\bF_{2^3}$, the degree 3 extension of   $\bF_2$. Then, the multiplicative group of nonzero elements $\bF_{2^3}^*$ is cyclic (Euler's Theorem). Take a generator, $\alpha$, of it, you get a 7-cycle.\footnote{Of course, this works in general in $\PGL_k(\bF_q)$ for $k\ge 3$.} 

Suppose $f^*$ (degree $m$) is indecomposable and the ranges of inequivalent polynomials pairs $(f^*,g^*)$ over $\bQ$ are the same for almost all primes $p$.  The first step in Davenport's problem showed that Prop.~\ref{redRed} applies. Further: 
\begin{prop} Then, for $\C_m$ an $m$-cycle conjugacy class, there is $t$ with $(t,m)=1$, for which $\C_m^t$ is not conjugate to $\C_m$. From the {\sl Branch Cycle Lemma\/}: $f^*$  and $g^*$ are conjugate over a proper extension of $\bQ$.\end{prop} 
That is in \cite{Fr73a}, with the story of what it led to, including the {\sl Genus Zero Problem\/}, in \cite[\S7.1]{Fr12}. 

We explain how   \S\ref{deg7}  examples arise from  degree 7 pairs $(f^*,g^*)$, each branched at four (not three) points. A connected (Hurwitz) space $\sH_{\infty\cdot 2^3}$ given by the Nielsen class $\ni_{\infty\cdot 2^3}\eqdef \ni(\PSL_2(\bF_2), \bfC_{\infty\cdot 2^3},\bT)$ parametrizes the pairs $(f^*,g^*)$ for which  $\tilde \sC_{f^*,g^*}$ has two connected components, and $(f^*,g^*)$ are represented by polynomials.   We call this Hurwitz space $\sH_7$ later.

We give two examples of these coalescings. Then, all examples naturally related to the \S\ref{deg7}  examples  of reducible $\tilde \sC_{f,g}$, with genus 0 components, come from coalescings in the  Nielsen class of  $(f,g)$. Cor.~\ref{findeg7polyb} finishes Rem.~\ref{7branchcyclesa}. Here are the coalescings. 

\begin{equation} \label{7branchcyclesb} \begin{array}{rl}\text{${}_1\psigma$-coalesce:}& ( (1\,3)(4\,5), (1\,6)(2\,3),(6\,4)(1\,7), \sigma_\infty)\\  \rightarrow& ((1\,3)(4\,5), (1\,4\,6\,7)(2\,3),\sigma_\infty) \\
\text{${}_2\psigma$-coalesce:}&((1\,3)(4\,7),(2\,3)(5\,7),(1\,4)(6\,7),\sigma_\infty)\\
\rightarrow& ((1\,2\,3)(4\,5\,7),(1\,4)(6\,7),\sigma_\infty). \end{array} \end{equation} 

The coalescing procedure: Multiply the 2nd and 3rd (resp.~1st and 2nd) entries in the line for ${}_1\psigma$-coalesce (resp.~${}_2\psigma$-coalesce). This would be a 1-step coalescing, but we can form many-step coalescings  as in Ex.~\ref{exgenDeg7} and Ex.~\ref{ratDeg7}.

\begin{exmp}[Before coalescing] \label{exgenDeg7}  Consider $\ni(G(\psigma\cdot\ptau), \bfC_{\infty\cdot 2^3},\bT)$, the Nielsen class  where $\bfC$ consists of three copies of the conjugacy class of $({}_1\sigma_1,{}_1\tau_1)$ in \eqref{bcfg} and one copy of a class of a 7-cycle. 
Before coalescing, the components of $\tilde C_{f,g}$ have genus 1  for the following reason. We can take branch cycles for these covers to be 
$$((\sigma_1',\tau_1'), (\sigma_2',\tau_2'), (\sigma_3',\tau_3'), (\sigma_\infty,\tau_\infty))$$ with each $(\sigma_i',\tau_i')$, $i=1,2,3$,  conjugate in $G(\psigma\cdot \ptau)$ to $(\sigma_1,\tau_1)$ in 
\eqref{bcfg}. So, for each of these indices in the representation computed, there will be 8. The analog computation for \eqref{1O11} of this genus $\geng_*$ gives 
$$  2(21 + \geng_{*}-1)=3\cdot 8+ 18, \text{ or }\text{\bf g}_{*}=1.$$ \end{exmp} 
 
\begin{exmp}[A bigger coalescing] \label{ratDeg7} Keeping $G=\PGL_2(\bZ/2)$, form 
Nielsen classes for rational functions $(f^*,g^*)$ by replacing  $(\sigma_\infty,\tau_\infty)$ by a repetition of three copies of the conjugacy class of $(\sigma_1,\tau_1)$. That is, $\bfC=\bfC_{2^6}$ consists of six repetitions of the involution class in $\PGL_2(\bZ/2)$. Here is the computation of the genus, $g_{**}$,  of the degree 3 component of $\tilde \sC_{f^*,g^*}$: 
$$  2(21 + \text{\bf g}_{**}-1)=6\cdot 8, \text{ or }\text{\bf g}_{**}=4.$$ \end{exmp}

\subsubsection{Coalescing and braids} \label {braid-coalesce} Here, we  finish the context of these degree 7 coalescings, describing their Nielsen classes as {\sl satellites\/} cohering to one space attached to a fixed Nielsen class (Thm.~\ref{satellitecor}). 

The {\sl braid action\/} on an $r$-tuple $\psigma=(\row \sigma r)$ satisfying the branch-cycle conditions \eqref{RET}, is generated by two elements: 
\begin{edesc} \label{braidaction} 
\item $q_1: \psigma \mapsto (\sigma_1\sigma_2\sigma_1^{-1},\sigma_1,\sigma_3,\dots,\sigma_r)$ the {\sl 1st\/} (coordinate) {\sl twist\/}, and
\item  $\sh: \psigma \mapsto  (\sigma_2,\sigma_3,\dots,\sigma_r, \sigma_1)$, the {\sl left shift\/}. \end{edesc} 

Conjugating $q_1$ by $\sh$, gives $q_2$: the twist moved to the right. Repeating gives $q_3,\dots, q_{r-1}$.  Denote the group generated by the braids by $H_r$ (more accurately described as the {\sl Hurwitz monodromy\/} quotient of the braid group). Here are uses of these braids as applied to a given Nielsen class $\ni$. Denote the absolute Galois group of the number field $L$ by $G_L$.  
\begin{edesc} \label{braiduse} \item \label{braidusea} Each braid preserves generation, product-one and the conjugacy class  conditions. So it preserves  branch cycles in $\ni$. 
\item \label{braiduseb} Given $\psigma\in \ni$, applying braids to $\psigma$  allows forming $\psigma'$ whose entries represent the elements of $\bfC$ in any desired order. 
\item \label{braidusec} There is a minimal cyclotomic field $L_{\ni}$ for which  $G_{L}$ maps all covers over $\bar L$ in $\ni$ into covers in $\ni$ if and only if $L_{\ni}\subset L$. 
\end{edesc} 

The {\sl branch cycle argument\/} (\cite[p.~62]{Fr77} or \cite[\S5.1.3]{Fr12}) gives \eql{braiduse}{braidusec}.  
Deeper points follow from this assumption  \cite[Thm.~5.1]{Fr77}: 
\begin{edesc} \label{htransHro} \item \label{htransHroa} For $\phi: X\to \prP^1$ in $\ni$, $H_r$ transitive on $\ni$. 
\item \label{htransHrob} \eql{htransHro}{htransHroa} is equivalent to there being one connected component of covers in $\ni$). \end{edesc} 

 \begin{edesc} \label{ctransHr} \item \label{ctransHra} Then, for any $\psigma\in \ni$ (Def.~\ref{defnc}), for some classical generators (\S\ref{usenc}), $\psigma$ is the branch cycle description for $\phi$. 
 \item \label{ctransHrb} For absolute Nielsen classes with $G(T_\phi,1)$ its own normalizer in $G$, the intersection of all definition fields of covers in $\ni$ is $L_{\ni}$.  
\end{edesc} 

Indeed, \eql{ctransHr}{ctransHra} is equivalent to \eqref{htransHro}, and even without it, $L_{\ni}$ is the right definition field for  the Hurwitz space. Here, though, there are three such spaces, $\sH_7, \sH_{13}$ and $\sH_{15}$ corresponding to the most interesting of the degrees that appear in \eqref{genus0reduct} and referred to in the examples, Cor. \ref{findeg7polya} and \ref{findeg7polyb}. Hypothesis  \eqref{htransHro} does hold for them. 

Connectedness of $\sH_7$ (as in \S\ref{pakgoal}) translates to transitive braid action for $r=4$ on all elements of $\ni(\GL_3(\bZ/2),\bfC_{\C_{\infty\cdot 2^3}},T_1)$ ($T_1$ just taking the presentation corresponding to one cover) as  in \cite[Prop.~4.1]{Fr05b} (Princ.~\ref{nielsenClass}). This Nielsen class is of 4-branch point covers of $\prP^1_z$. Lem.~\ref{H3action} effectively detects orbits of $H_3$ on Nielsen classes of 3-branch point covers. There is no such easy conclusion for $H_r$, $r\ge 4$ on Nielsen classes. 
\begin{lem} \label{H3action} The group $H_3=\lrang{q_1,q_2}$ acts on a Nielsen class $\ni(G,\bfC, T)$, with $\bfC$ consisting of 3 conjugacy classes,  as a quotient of the dihedral group $D_3=S_3$ of order 6, since $q_1^2$ and $q_2^2$ act trivially.  If all three classes in $\bfC$ are distinct modulo $N_{S_{\deg(T)}}(G)$ then all $H_3$  orbits have length 6.
  \end{lem}
  
\begin{proof} This follows fairly straightforwardly from \cite[\S2.4.1]{BaFr02}, which notes that for $\tau_1=q_1q_2$, $\tau_1^3=q_1^2=q_2^2$ from standard braid group relations. Now consider the action of $q_2^2$ on a Nielsen class element:  $$(g_1,g_2,g_3)q_2^2\mapsto  (g_1,\alpha(g_2\alpha^{-1}, \alpha(g_3\alpha^{-1}), \text{ with }\alpha=(g_2g_3)^{-1}.$$ 
From the product-one condition, $\alpha=g_1^{-1}$.  So, $q_2^2$ fixes the absolute class of $(g_1,g_2,g_3)$. Similarly for $q_1^2$. Therefore, in its action on Nielsen classes, $H_3$ is generated by two involutions, and therefore it is a dihedral group. It suffices to decide which dihedral group it factors through. 

Since $\tau_1$ has  order 3,  this is $D_3=S_3$, acting as permutations of the conjugacy classes. As the classes are distinct modulo $N_{S_{\deg(T)}}(G)$, the action on those will give all permutations since their orderings are also distinct. 
\end{proof} 

We conclude with comments on the space attached to elements in a Nielsen class. This  is done for absolute classes in detail in \cite[\S5]{Fr77}. It and \cite{BaFr02} emphasize that their properties come from the explicit {\sl Hurwitz monodromy\/} (braid) action on Nielsen classes. \eqref{braidaction} is a special case. Components, as in \eqref{Hrbasics}, correspond to orbits of $H_r$ on Nielsen classes. 

\cite[\S5]{Fr12} has exposition, using the {\sl branch cycle lemma\/} for finding the definition field of the whole Hurwitz space. \cite[\S6.4]{Fr12} uses deg 7 again (\cite[Thm.~6.7]{Fr12}). That includes displaying the relation between inner and absolute spaces (\eqref{innabsdiagram} does this for a given Nielsen class). \cite[\S4]{BaFr02} starts the more advanced topics on cusps, including effective computation of the genus of reduced Hurwitz space components for 4 branch point covers. Then,  each component is an upper half-plane quotient and $j$-line cover, though it is rarely a modular curve.

\subsubsection{\S\ref{deg7} and the genus 0 problem} \label{genus0prob}  This section expands on the available empirical data that can help transition from using the results of the Genus 0 problem to more general cases, say in testing computer programs (as suggested in \S\ref{touchPak}). For developing insight, we suggest testing the list from \eqref{genus0reduct} for what  Cors.~\ref{findeg7polya} and \ref{findeg7polyb} have done for degree 7. That would fulfill an indecomposable polynomial extension of Pakovich.  

The anomalous (decomposable) example $(f^*,g^*)=(T_4,-T_4)$ for which $\tilde \sC_{f,g}$ has two degree 2 components has been discovered numerous times (Rem.~\ref{evendihedral}). For it, though, \eql{Pak}{Paka} fails (the Galois closure has genus 0). Also, if we keep $f^*$ a polynomial, but allow $g^*$ to be any rational function, \S\ref{appSpaces} has the significant case where $\deg(f^*)=5, \deg(g^*)=10$.

Monodromy groups $(G,T)$ of indecomposable rational functions are called (primitive) {\sl genus 0\/} groups. Guranicks' version of the genus 0 problem (over $\bC$) formulates what pairs $(G,T)$ -- $\deg\ n$ permutation representation $T$ -- could be monodromy groups of $f: X\to\prP^1_z$ satisfying: 

\begin{equation}  \text{$X$ has genus 0 and $f$ is indecomposable ($T$ is primitive).}\end{equation} 
He extended this to genus 1 and any fixed genus $> 1$. Our qualitative results don't require these higher genus generalizations.  

The idea -- akin to the classification, but not restricted to simple groups -- divided possible $(G,T)$ into two sets by appeal to the classification of {\sl primitive\/} groups -- an offshoot primed by essentially indexing them using simple groups \cite{AOS85}.  Primitive genus 0 monodromy groups have two types: 

\begin{edesc} \label{genus0mon} 
\item \label{genus0monc} Genus 0 series: Infinite series of $(G,T)$ with $G$ having {\sl core\/} either $A_n$ or $=(\bZ/p)^a\xs \bZ/d$, $d\in\{1,2,3,4,6\}$  appearing as genus 0 monodromy groups \cite[p.~78]{Fr05b} or \cite[\S7.1]{Fr12}.   
\item \label{genus0mond} Genus 0 exceptional: excluding \eql{genus0mon}{genus0monc} there are only finitely many (primitive) genus zero $(G,T)\,$s.  
\end{edesc} 

Thm.~\ref{polyPakThm} is a {\sl polynomial\/} version of the main result, Thm.~\ref{genPakThmGen}, with the extra assumption $f$ is indecomposable. Further,  
\begin{edesc} \item they produce  polynomial pairs.   So $\bfC$ has a conjugacy class $$\C=\C_\infty  \text{ whose elements have order }\deg(T_1)=\deg(T_2)\text{ and;} $$
\item both $T_1$ and $T_2$ are primitive ($f^*$ and $g^*$ are indecomposable). \end{edesc}   
Comments below give more on the other cases that arose in Davenport's and Schinzel's problems. 

\begin{thm} \label{polyPakThm} Assume $f^*,g^*\circ g_1\in \bC[x]$, $f^*$ is indecomposable and $f^*$ and $g^*$ give inequivalent covers. Then,  
\eqref{genus0reduct} lists all degrees of $T_1$ for the Nielsen classes $\ni(G,\bfC,\bT)$ for which  $\tilde \sC_{f^*,g^*\circ g_1}$ is reducible. \begin{equation} \label{genus0reduct}  \text{$\deg(f^*)=\deg(g^*)$  is in $\{7, 11, 13, 15, 21,31\}$.} \end{equation} 

For each degree in \eqref{genus0reduct}, there is one main Nielsen class for which $\tilde \sC_{f^*,g^*}$ is reducible. That class is either unique or in degrees 7, 13, and 15; or it is one of a finite number of  Nielsen classes obtained from coalescing the main class (as with the degree 7 case of \S\ref{deg7}).  \end{thm} 

\begin{proof}[Comments] All degrees in \eqref{genus0reduct} can be analyzed with a rubric close to that of our running deg 7 examples. Each gives several examples of Nielsen classes, $\ni(G,\bfC,\bT)$, of pairs of rational functions $(f^*,g^*)$ all of whose representative fiber products $\tilde \sC_{f^*,g^*}$ are reducible. For some of those Nielsen classes, the representatives have components of genus 0 or 1. Further, the rubric and Thm.~\ref{genPakThmGen} apply to produce examples of generalizing Thm.~\ref{PakThm}.   
 
In each of these cases, $T_1$ and $T_2$ are distinct permutation reps, equivalent as group representations. We have mostly considered the case where $g^*$ has a totally ramified place, and composing with an element of $\PSL_2(\bC)$ applied to $\prP^1_z$ turns the cover into a polynomial. As however, with the degree 7 case, using coalescing and regarding these cases as satellites of more general cases gives rational function cases that I haven't analyzed completely.

 \cite[\S9]{Fr99} wraps up classifying the cases \eqref{genus0reduct}, describing the spaces of such pairs of polynomials. 
Like the degree 7 case, general covers in the degrees 13 and 15 Nielsen classes have four branch points. \cite[\S8.2]{Fr99} lists  elements of the degree 13 Nielsen class. Here $\{1,2,4,10\}$ is what $\{1,2,4\}$ was in degree 7 (as in \eqref{staby}): a difference set for the doubly-transitive design. Now apply Cor.~\ref{methodII} to find  the  $\tilde \sC_{f,g}$ components genuses, as  for degree 7. 

\cite[\S 8.3]{Fr99} finds the braid group generator  $q_1$ and $q_2$ actions. That produces the genus of the reduced Hurwitz space $\sH({1,2,4,10})$ of such covers. 
\cite[\S 8.4]{Fr99}  finds the defining field, $K_{13}$,  for $\sH({1,2,4,10}$ as moduli of such pairs. it is the fixed field of $M({1,2,4,10})$ in $\bQ(\zeta_{13})$: $ \bQ(\zeta_{13} + \zeta_{13}^3 + \zeta_{13}^9)$.  The Hurwitz space for these Davenport pairs is a rational variety \cite[\S 8.5]{Fr99}. 

\cite[Thm. 9.1]{Fr99}  bounds degrees of Davenport pairs, with detail for $n=31$, the highest degree, in \cite[\S9.2.2]{Fr99}. \cite[\S2.3.2 and \S3]{Fr05b} discuss difference sets, and describes Nielsen classes for the maximal families of Davenport pairs for $\deg(f)=7, 13, 15$. Here are two standout points.  

\begin{edesc} \label{conds} \item \label{condsa} Given a particular conjugacy class of an $m$-cycle (resp.~there are 2, 4, 2 for $m=7, 13, 15$ \cite[p.~61]{Fr05b}), then there is just one Nielsen class of covers with 4 branch points (counting $\infty$). 
\item \label{condsb} Each Nielsen class of polynomials having these degrees comes from coalescing from the Nielsen class in \eql{conds}{condsa}. 
%\item \label{condsc} 

\end{edesc}

\cite{Mu95} (or \cite[App.~C1]{Fr05}) describes all possible genus 0 monodromy groups of polynomials. The polynomial map case is substantial -- much given by the results of the Davenport Problem -- and sufficient for strong clues on the genus 0 problem. Still, the latter is the poster child for distinguishing between results on polynomials and rational functions. 
\end{proof}

\section{Dropping irreducibility hypothesis \eql{Pak}{Pakb}} \label{noirred}  
\S\ref{compgenuses} is a primer on Nielsen classes. \S\ref{overview-PartI} sets up considering the pairs $(f^*,g^*)$, (and their Nielsen classes) coming out of Prop.~\ref{redRed} with $\tilde \sC_{f^*,g^*}$ reducible.   Prop.~\ref{compinc} lists the categories into which the components of $\tilde \sC_{f^*,g^*\circ g_1}$ fall, and what $g_1\,$s to avoid to get our particular generalization of Thm.~\ref{PakThm}.. 

 \S\ref{ocharg0} uses our running degree 7 examples to show those ingredients are sufficient to produce Main Thm.~\ref{genPakThmGen}.  In \S\ref{noclassif} we referred to  Cor.~\ref{findeg7polya} and Cor.~\ref{satellitecor} as the {\sl Degree 7 Corollaries}. This is a model for  how well such an extension applies to Nielsen classes of rational functions $f$ of a given degree, containing representative pairs $(f^*,g^*)$ with $\tilde \sC_{f^*,g^*}$ reducible, and then those with genus 0 or 1 components. 

Finally, \S\ref{genPak} ties together the concepts and shows how to form the Nielsen classes of excluded $g_1\,$s for our main theorem.

\subsection{$\tilde\sC_{f,g}$ properties from $T_1,T_2$} \label{overview-PartI}    \S\ref{usingredRed} does some housekeeping on issues about computing with rational functions. Then, it gets algorithmic on using Prop.~\ref{redRed} and on listing components of representative $\tilde \sC_{f^*,g^*}$ in a Nielsen class. It also applies Pakovich's \ochar\ to some of our examples. \S\ref{setPak} precisely lists the technical obstructions, some akin to those Pakovich gave to the generalization. \S\ref{helpwithNielsen} gets algorithmic about setting up Nielsen classes.   

Here are some preliminary reminders. If $\tilde\sC_{f,g}$ is reducible, then so is $\tilde\sC_{f\circ f_2,g\circ g_2}$ for arbitrary (nonconstant) $f_2,g_2\in \bC(x)$. Prop.~\ref{redRed} points to where we would find the heart of  reducible $\tilde\sC_{f,g}$.  View Thm.~\ref{PakThm} as starting from a Nielsen class, $\ni(G,\bfC, T_1)$, of genus-zero covers. Then, branch cycles for $f:\prP^1_x\to \prP^1_z$ in the Nielsen class are in the classes $\bfC$ and $T_1=T_{f^*}$ as equivalent representations of $G$.

\subsubsection{Using Prop.~\ref{redRed}} \label{usingredRed}  We have already discussed aspects of $f^*$ and $g^*$ having a common left composite (\CLC). 

\begin{edesc} \label{exceptFP}  \item Ex.~\ref{diagcomp1} showed that gives reducibility of $\tilde \sC_{f^*,g^*}$. 
\item When $\tilde \sC_{f^*,g^*}$ has a genus 0 component, Lem.~\ref{twodecomp} showed, without amending the hypothesis, this violates the Thm.~\ref{PakThm} conclusion. \end{edesc} Now consider how to recognize \CLC; it can't be eliminated trivially. 

\begin{lem} \label{commoncomp}  For $f,g\in \bC[x]$, there is an effective check for whether a pair $(f,g)$ have a \CLC.  This is equivalent to $f(x) \nm g(y)$ has a variables separated factor of form $u(x)\nm v(y)$ with $f=f^*\circ u$ and $g(y)=f^*\circ v(y)$ ($\deg(f^*)> 1$). 

A Nielsen class equivalent without the polynomial assumption:  The representation $T_f\otimes T_g$ on the group $G_{f,g}$. has a component isomorphic to the pullback of $T_{f^*}$ from $G_{f^*}$. \end{lem} 

\begin{proof}[Comments] The case $(f,g)$ are polynomials is the main theorem of \cite{FrM69} with \cite[\S5]{FrM69} giving  counterexamples when either $f$ is a rational function or the field $K$ has characteristic dividing $\deg(f)$. These are also counter to the extension of $K$ to its algebraic closure, preserving the lattice of fields. Yet, for the general polynomial, maximal chains all have the same length and relative degrees (in possibly different order). 

The second paragraph, using Nielsen classes, is therefore stronger. 
\end{proof}

Use notation like that of \eqref{gendeg7rat} . Assume $\pmb \mu\in G^r\cap \bfC$  generates $G$, and its entries have product 1. Apply $T_1$ to $\pmb \tau$, giving $T_1(\pmb \tau)$ and, by \RET,  a representative $f^*$ of $\ni(G,\bfC, T_1)$. Another faithful transitive representation $T_2$ of $G$ produces $\ni(G,\bfC,T_2)$. Then, $T_2(\pmb \tau)$ produces a cover $X_{T_2(\pmb \tau)}\to \prP^1_z$. A rational function $g^*: \prP^1_y\to \prP^1_z$ represents this if covers in $\ni(G,\bfC,T_2)$ have genus zero.

Lem.~\ref{compcor}  uses the Galois correspondence. Each cover $\phi_W: W\to \prP^1_z$ through which $\hat X_{f^*}$ factors has the form $\hat X_{f^*}/G(T_{\phi_W},1)$.  Its conclusion lists the conjugacy classes of subgroups of $G$ defining components of $\tilde \sC_{f^*,g^*}$. 

\begin{lem} \label{compcor} A representation $T$ of $G$ is faithful if and only if the intersection of all conjugates of $G(T,1)$ in $G$ is trivial. If $G$ is a simple group, and $T$ is a nontrivial permutation representation of $G$, then $T$ is faithful. 

With $(f^*,g^*)$ as above, as in \eqref{proofuse} assume a component $W$ of  $\tilde\sC_{f^*,g^*}$ corresponds to an orbit, $I$, of $G(T_1,1)$ on $\{\row y n\}$.  Then $W\to \prP^1_z$ is equivalent (as a cover) to $\hat X_f/H\to \prP^1_z$ with $H=G(T_1,1)\cap G(T_2,j)$ for any $j\in I$. The genus, $\geng_W$ of $W$, is the genus of the covers in $\ni(G,\bfC,T_W)$.  \end{lem} 

\begin{proof} The first statement is well-known, a consequence of characterizing $h\in G$ fixing all cosets  $\{G(T,1)g\}_{g\in G} \leftrightarrow h\in \cap_{ g\in G} gG(T,1)g^{-1}$.

We show the second statement.  The Galois correspondence says $H$ corresponds to a cover $W^*=\hat X_f/H$ that  is an image from $\hat X_f=\hat X_g$. 

Since $H\le G(T_1,1)$ and $\le G(T_2,j)$, then $W^*\to \prP^1_z$ factors through $\hat X_f/G(T_1,1)$ (resp.~ $\hat X_g/G(T_2,j)$) a cover of $\prP^1_z$ equivalent to $\prP^1_x$ (resp.~$\prP^1_y$). From the universal property of fiber products, this gives a map from $W^*$ to $\tilde\sC_{f,g}$. The image, $W$,  of $W^*$ is a  component of $\tilde\sC_{f,g}$,  corresponding to a subgroup of $G(T_1,1)\cap G(T_2,j)$. So $W^*=W$. 

Now consider the converse: $W$ is a component of $\tilde \sC_{f^*,g^*}$. Prop.~\ref{redRed} says $W$ corresponds to an orbit, $I$, of $G(T_{f^*},1)$ in the representation $T_{g^*}$ (or an orbit of $G(T_{g^*},1)$ in the representation of $T_{f^*}$).  Suppose $j\in I$. Conclude, in the Galois correspondence, that $W\to \prP^1_z$  is equivalent to the cover that corresponds to the subgroup $G(T_{f^*},1)\cap G(T_{g^*},j)$ of $G(T_{f^*},1)$ leaving $j$ fixed.  

The genus comment is a restatement of a cover's genus being computed from the Nielsen class of the cover by \RH. This concludes the proof. \end{proof}

We use Ex.~\ref{appear7ochar} in Cor.~\ref{findeg7polya}. In each of the Nielsen classes pointed to by \eqref{genus0reduct}, $\ochar<0$. Yet, the condition $\ochar\ge 0$ arises in the complete description of Thm.~\ref{genPakThmGen}, even  in the degree 7 cases if we allow $g^*$ to be a rational, rather than polynomial, function.

\begin{exmpl}[Appearance of $\ochar\ge 0$ for degree 7] \label{appear7ochar} 
Suppose in any of the degrees in \eqref{genus0reduct}  there is a(t least one) Nielsen class $\ni(G_{f^*},\bfC,\bT)$ so that $\sC_{f^*,g^*}$ has a genus 0 component $W$.  In the notation of Lem.~\ref{twodecomp}, take $g=g^*\circ h_y$ with $h_y: \prP^1_w\to \prP^1_y$ representing the cover $W\to \prP^1_y$. 

Use the components of degrees 3 and 4 found in Lem.~\ref{invorbits}. We labeled the two Nielsen classes with a pre-subscript $j=1$ or 2. Components are correspondingly ${}_jW_{u'}$: $u'=1$ has  degree 3, and $u'=2$ has degree 4. 

Cor.~\ref{findeg7polya} finishes handling the projections onto $\prP^1_y$ of   the degree 3 (resp.~degree 4)  ${}_{u'}W_{1}$  (resp.~${}_{u'}W_2$) components of degree 7 polynomials. For example, with  $({}_2\sigma_1, {}_2\tau_1)$ and $({}_2\sigma_2,{} _2\tau_2)$ in \eqref{bcfg2}. $$\ochar=2+(1/3-1)-2(1/2)=1/3.$$ Branch cycles for ${}_2h_{1}: {}_2W_{1}=\prP^1_w\to \prP^1_y$ have a 3-cycle. Its Nielsen class is that of the degree 3 Chebychev polynomial  (\S\ref{ochargg0}). 

From Rem.~\ref{comporbit}, the degree 4 complementary component ${}_2W_{2}$ also has genus 0. Directly compute  
$\ochar=2+(1/3\nm1)+4(1/2\nm1)<0$. By inspection of the branch cycles of the degree 4 $${}_2h_{2}: {}_2W_{2}=\prP^1_w\to \prP^1_y$$ its monodromy is $S_4$; its branch cycles contain both a 2 and a 3-cycle.    \end{exmpl}

\begin{rem} \label{orbhist} Clearly, only if the Galois closure cover of $f$ has genus 0 or 1 is $\ochar_f$ nonnegative: condition \eql{Pak}{Paka} does not hold. \S\ref{orbzero} presents, differently than the short exposition of \cite[p.~2--3]{Pak18b},  on how that produces examples violating the conclusion of Thm.~\ref{PakThm}.  

\cite[p. 5-6]{Da} says orbifolds first appeared (as V-manifolds; see right below \eqref{orb}) in \cite{Sa56}. \cite[\S13.3]{Th76} changed the name V-manifold. \end{rem} 

\begin{rem}[Finding separated variable factors]  \label{findsepvar} Factoring 2-variable polynomials is easier than finding composition factors of a rational function. Then, checking if a two-variable polynomial has a separated variables factor, as in Lem.~\ref{commoncomp} is easier still. Therefore, checking for a given $g$, if $f$ and $g$ have a common composition factor, is not so hard. 

Handling genus 0 covers with imprimitive monody will require new ideas, going beyond the genus 0 problem (for primitive groups), but Prop.~\ref{compinc} sets us in an inductive direction to find new cases of it.  \end{rem}

\subsubsection{Pieces to extend Thm.~\ref{PakThm}}  \label{setPak}  We use the \eqref{proofuse} notation emphasizing the projections of components of $\tilde \sC_{f^*,g^*}$  on $\prP^1_y$.\footnote{Despite the symmetry in Prop.~\ref{redRed} between the $x$ and $y$ variables.} So, the components correspond to the orbits $\row J u$ of $G(T_2,y_1)$ on $\row x m$. The orbit $J_j\leftrightarrow W_j$ gives a cover, $\pr_{y,j}: W_j\to \prP^1_y$ of degree $\ell_j=|J_j|$. 

There are two Nielsen classes attached to the component $W_j$: That represented by $W_j\to \prP^1_z$ corresponding to the representation $T_H$ in Lem.~\ref{compcor} and that from the natural projection map from the fiber product, $\pr_{y,j}: W_j\to \prP^1_y$. For  $1\le j\le t$ for which $W_j$ has genus 0, denote $\pr_{y,j}$ by $h_j$, indicating representation by a rational function.\footnote{If you aren't being careful about definition fields; the rational function might not be over the definition field of $\pr_{y,j}$.} 

Prop.~\ref{compinc} uses Prop.~\ref{redRed} to preclude growth of inappropriate components  in going from $\tilde \sC_{f^*,g^*}$ to $\tilde \sC_{f^*,g^*\circ g_1}$.  The lines \eql{fail1}{fail1a} and \eql{fail1}{fail1b}  inductively consider components of $\tilde \sC_{f^*,g^*\circ g_1}$ that map to those of $\tilde \sC_{f^*,g^*}$. Then. \eql{fail1}{fail1c} considers additional  components based on  diagram \eqref{genus0diag}, invoking considerations of Ritt under the definition {\sl decomposition variant}. Use Prop.~\ref{redRed} in the form of Cor.~\ref{indredRed}. 

\begin{defn}[Case $f^*$ indecomposable]  \label{decvar} Call  $g^*\circ g_1=g^{\star}\circ g_1'$ a (nontrivial) decomposition variant ({\sl dec-var\/} in going) from $\tilde \sC_{f^*,g^*}$ to $\tilde \sC_{f^*,g^*\circ g_1}$ if $g^\star$ and  $g^*$ are inequivalent covers and \eqref{concredGal} holds by replacing  $(f^*,g^*)$ with $(f^*,g^\star)$. 
\end{defn} 
\noindent Rem.~\ref{decvar2} discusses generalizing Def.~\ref{decvar}  without assuming $f^*$ is indecomposable. 
 Reference to $g_1$ in Prop.~\ref{compinc} (and elsewhere) assumes $\deg(g_1)>1$. 

\begin{prop} \label{compinc}  The possible relations between components of $\tilde \sC_{f^*,g^*\circ g_1}$ and those of $\tilde \sC_{f^*,g^*}$ is given by \eqref{fail1}. 

\begin{edesc} \label{fail1} \item \label{fail1a} For $1\le j \le u$, components of  $\tilde \sC_{f^*,g^*\circ g_1}$  that map onto $W_j$ identify with components of $\tilde \sC_{\pr_{y,j},g_1}$.  
\item \label{fail1b} If for some $j$,  $\tilde \sC_{\pr_{y,j},g_1}$, $1\le j \le u$, has multiple components. Then, these give multiple  components of   
$\tilde \sC_{f^*,g^*\circ g_1}$ above $W_j$. 
\item \label{fail1c} The possibility for a component $W'$ of $\tilde \sC_{f^*,g^*\circ g_1}$ that does not lie above any $W_j$ is that it appears as  a dec-var from $\tilde \sC_{f^*,g^*}$ to $\tilde \sC_{f^*,g^*\circ g_1}$.
 \end{edesc} 
 If \eql{fail1}{fail1c} holds, then there is another decomposition of $g^*\circ g_1$ as $g^{\star}\circ g_1'$, and a correspondingly with $f^*=f^{\star}\circ f_1$,  for which some component on $\tilde \sC_{f^\star,g^\star}$ contains the component $W'$. When $f^*$ is indecomposable, then $f^{\star}=f^*$. 
\end{prop}

\begin{proof} As above, fix a component $W_j$ of $\tilde \sC_{f^*,g^*}$. We leave the $\tilde{}$ decoration off the fiber products so that we can describe and identify them before (projective) normalization. 

The collection of components, $W'_j$ of $\tilde \sC_{f^*, g^*\circ g_1}$ that map to $W_j$ identify with the normalization of 
\begin{equation} \label{fp1} W^\dagger_j\eqdef \{(x,y_2)\in \prP^1_x\times \prP^1_{y_2}\mid f^*(x)=g^*\circ g_1(y_2), (x,g_1(y_2))\in W_j\}.\end{equation} We show normalization of $W^\dagger_j$ identifies with normalization of 
\begin{equation} \label{fp2} \sC_{\pr_{y,j}, g_1}=\{(x,y; y_2)\in W_j^\dagger\times \prP^1_{y_2} \text{ with } g_1(y_2)=y\}. \end{equation}  To see that \eqref{fp1} and \eqref{fp2} define the same spaces, insert $(g_1(y_2),y_2)$ in place of the second coordinate of $(x,y_2)$ from \eqref{fp1}. 

That completes of the proof of \eql{fail1}{fail1a}, and immediately gives \eql{fail1}{fail1b}. 

From Prop.~\ref{redRed} if there is another component $W'$ not accounted,  \eql{redGal}{redGalb} fails. So, Prop.~\ref{redRed} implies there  must another component on $\tilde \sC_{f^*, g^*\circ g_1}$ accounted for by another decomposition of $g^*\circ g_1$ and $f^*$ as given in the last paragraph in the proposition.\end{proof} 

Now for all the component types that arose in Prop.~\ref{compinc}, we assume that none have Galois closure of genus 0 or 1 \eql{Pak}{Paka}. 

\begin{equation} \label{fail1aex} \text{The place of \eql{fail1}{fail1a} in Prop.~\ref{redRed}.}\end{equation} 

If $g_1=h_j$, $1\le j\le t$, then for    $g_2\in \bC(y)\setminus \bC$,   $\tilde \sC_{f^*,g^*\circ h_j^*\circ g_2}$ has a genus 0 component.  Use \eql{ratfunctdec}{ratfunctdecc} with $\tilde \sC_{\pr_{y,j},g_1\circ g_2}=\tilde \sC_{h_j,h_j\circ g_2}$. 

In general, multiple components  above $W_j$, implies that the pullback of of $G(T_1,x_1)$ to $G_{g^*\circ g_1}$ is not transitive on the solutions (in y) of $g_1(y)=y_j$.

\begin{equation} \label{fail1bex} \text{Possibilities related to \eql{fail1}{fail1b}.}\end{equation} 

Here are the $g_1\,$s to avoid to assure the genus of all components of $\tilde \sC_{\pr_{y,j},g_1}$ rises with the degree.  
\begin{edesc}  \label{W_jcomps} \item  \label{W_jcompsa}  For $W_j$ of genus $>1$, assure $\tilde \sC_{\pr_{y,j},g_1}$ no degree 1 component. 
\item \label{W_jcompsb}  For $W_j$ of genus $=1$, assure no component of $\tilde \sC_{\pr_{y,j},g_1}$ is unramified over $\tilde \sC_{\pr_{y,j},g_1}$. 
\item  \label{W_jcompsc} For $W_j$ of genus 0, $1\le j\le t$, assure  with $(f^*,g^*)\mapsto (h_j,g_1)$  in Prop.~\ref{redRed}, that  $T_{h_j}$ and $T_{g_1}$ are not entangled. \end{edesc} 

Lem.~\ref{avoidg1s} is a step toward using Nielsen classes  to avoid \eqref{W_jcomps} $g_1\,$s. 
\begin{lem} \label{avoidg1s} In \eql{W_jcomps}{W_jcompsa} avoid all the genus 0 quotients of $\pr_{y,j}$. 
In \eql{W_jcomps}{W_jcompsb} assure the branch points, $\bz$, of $\pr_{y,j}$  contain those of  $g_1$  and for $z_i\in \bz$, each disjoint cycle of  the branch cycle for $\pr_{y,j}$ at $z_i$ is a multiple of {\sl each\/} disjoint cycle of the branch cycle for $g_1$ at $z_i$. 

In \eql{W_jcomps}{W_jcompsc} assure  the stabilizer of a letter in the representation $T_{h_j}$ is transitive on the letters of the representation of $T_{g_1}$. As in \S\ref{noclassif}, this is akin to avoiding Davenport-entanglement. \end{lem}  

\begin{equation} \label{fail1cex} \text{Detecting the most difficult components: \eql{fail1}{fail1c}}\end{equation} 

The component denoted $W'$ is a {\sl new\/} genus zero component; not one on $\tilde \sC_{f^*,g^*}$, though it ends up on $\tilde \sC_{f^*, g^*\circ g_1}$. If $f^*$ is indecomposable, then it corresponds to a new representation of $G_{f^*}$. 

We may, though, luck out and find that $\tilde \sC_{f^*,g^\star}$ is irreducible -- there is no other representation that could be $T_{g^\star}$ for which \eqref{concredGal} holds --  and Pakovich's Thm.~applies to assure we need not worry about such a $g_1$. That is the kind of data we get from the genus 0 problem. 

For example, in our degree, 7 examples -- and in all the other examples that would come from the list of Thm.~\ref{polyPakThm} in a similar style --  we know that we do get several Nielsen classes of degree 7 pairs, $(f^*,g^*)$, whose representative $\tilde \sC_{f^*,g^*}$ are reducible, and even among those Nielsen classes many with genus 0 components. We also know there is no distinct 3rd representation to worry about in \eql{fail1}{fail1c}.

\begin{rem} \label{decvar2} To extend Def.~\ref{decvar} -- decomposition variant -- to the case $f^*$ is decomposable, suppose in considering $\tilde \sC_{f^*,g^*\circ g_1}$ that $g^\star\circ  g_1'=g^*\circ g_1$ and $f^\star\circ f_1'=f^*$ with the goal of an inductive argument on the number of decomposition factors of $f^*$. 
\end{rem}

\begin{rem} \label{nccomments} \cite{Pak22} raises iterated rational function applications. Some results try to grab random rational functions, but \lq\lq general\rq\rq\ rational functions are not where the action is.  

The Nielsen approach gives names to the functions you deal with. Then, you add to a particular Nielsen class further Nielsen class data for rational function pairs to describe the $g_1\,$s to avoid, corresponding to each of the comments above.  
\end{rem} 

\subsubsection{Handling pairs of reps.~$(T_1,T_2)$ of $G$} \label{helpwithNielsen} Below, fix either $f$ (or its Nielsen class).   
The genus 0 problem (\S\ref{classif}) limits $G$ and the permutation representations, $T_1$,  for genus 0 covers and their   Nielsen classes. 

Start from a particular Nielsen class, $\ni(G,\bfC, T_1)$  (say, for $f^*$), with $G_{f^*}=G$. Then, locate corresponding $T_2$ for which one must deal in generalizing Pakovich, without conditions \eqref{Pak}. This shows the description of the $g^*\,$s with $\tilde \sC_{f^*,g^*}$ having genus 0 (or 1) components has an entirely Nielsen class formulation starting from $f^*$.

\begin{edesc} \label{condPak} \item  \label{condPaka} Find $T_2$ for which $T_2$ is intransitive on $G(T_1,1)$ (as in \eql{redGal}{redGald}). 
\item   \label{condPakb} \eql{condPak}{condPaka} $\implies$  $g\in \ni(G,\bfC,T_2)$ fails \eql{Pak}{Pakb}: $\tilde \sC_{f,g}$ is reducible.
\item  \label{condPakc} Apply Cor.~\ref{methodII}  to identify the genuses of the components of $\tilde \sC_{f,g}$. 
\end{edesc}

\S\ref{irreducibilityb} examples give Nielsen classes that label  natural collections of pairs $(f^*,g^*)$ that produce the reducibility phenomena. The Nielsen class of $g^*$ is one of a finite number associated with the Nielsen class of $f^*$. If, however, you change the Nielsen class of $f^*$, then you start over again. 

{\sl Comments on\/}  \eql{condPak}{condPakb}: Branch cycles for $f^*$ consist of an $r$-tuple $$\text{$T_1(\pmb \tau)$ with } (\row {\pmb\tau} r)\in G^r \text{ satisfying \eqref{RET}}.$$  Automatically create branch cycles for $g^*$ as $T_2(\pmb \tau)$. 

The Prop.~\S\ref{Snex} example follows the steps in \eql{condPak}{condPaka} and \eql{condPak}{condPakb}.   Then, Cor.~\ref{Sngenus0} does step \eql{condPak}{condPakc}, finding  genus 0 components of $\tilde \sC_{f,g}$. 

We comment on what was essential about using the genus 0 condition and what was not. Then, we add additional comments on $T_1,T_2$ used in Thm.~\ref{genPakThmGen}. Start from Rem.~\ref{decchain}. 

\begin{edesc} \label{notgenus0} \item \label{notgenus0a} 
 Neither $f$ nor $g$  need be covers of genus 0 curves; components still correspond to orbits of $G(T_1,1)$ under $T_2$ in Prop.~\ref{redRed}.  
\item \label{notgenus0b} Orbits and component degrees apply even if $T_1=T_2$; one length 1 orbit   $\leftrightarrow$ a   $\tilde \sC_{f,g}$ component isomorphic to the diagonal. 
\item  \label{notgenus0c} It is convenient, but not necessary, in Cor.~\ref{methodII}  for the cover $f$  to have genus 0 to  find the genus of a component of $\tilde \sC_{f,g}$. 
\end{edesc} 

In \eql{notgenus0}{notgenus0c},   covers need only be nonsingular (irreducible) curves. As in \S\ref{irreducibilityb}, singular points of $\sC_{f,g}$ arise from coinciding images $z'$ of values of $x'$ (resp.~$y'$) ramified of order $e_{x'}$ (resp.~$e_{y'}$) over $z'$ with $(e_{x'}\nm1)(e_{y'}\nm1)>0$. 

Then, the $\gcd(e_{x'},e_{y'})$ (nonsingular) points  of $\tilde\sC_{f,g}$ in a neighborhood over of $z'$ are locally isomorphic to  $\tilde \sC_{x^{e_{x'}},y^{e_{y'}}}$ over $z'=0$. The notation for {\sl branch cycles\/} as in, say, Prop.~\ref{fiberram}, is still appropriate for a well-defined conjugacy class of elements in $G$ even if $f$ is not a genus 0 cover. 
 
 \begin{rem}[Comment on \eql{notgenus0}{notgenus0b}] \label{nonfaithful} No need to limit  $T_1$ and $T_2$ to faithful representations. For example, assume $T_1$ is imprimitive, corresponding to a system of imprimitivity $m_1$ (or $f=f_1\circ f_2$,  as in \S\ref{setPak}). Then, the representation ${}_{m_1}T=T_{f_1}$ extends to a non-faithful representation of $G_f$ by composing it with the natural cover $G_f\to {}_{m_1}G$.\footnote{For $\deg(f_i)> 1$, $i=1,2$, the  wreath product (Rem.~\ref{odds-ends}) says  $G_{f_1}$ is a proper quotient of $G_f$.}  A copy of the identity, from the diagonal on  $ \tilde \sC_{f,f}$, is in the kernel of $T_1\otimes T_1\to {}_{m_1}T\otimes {}_{m_1}T$.  
 \end{rem} 

\subsection{Treating $\ochar > 0$  when $f\in \bC[x]$ has degree 7} \label{ocharg0} This section uses our degree 7 example to show inputs to extend Thm.~\ref{PakThm}.  Ex.~\ref{appear7ochar}  applies the \ochar\ to the \S\ref{bcdeg7} examples of fiber products with genus 0 or 1 components. \S\ref{7branchcyclesb} (Cor.~\ref{findeg7polya}) finishes the analysis showing { ${}_2h_1$ is equivalent to a Chebychev polynomial, and thus its Galois closure has genus 0 \eqref{galClose}, while ${}_2h_2$ (degree 4) has group $S_4$. We see the value of using the $\ochar\ge 0$. Example: \eqref{chebychev} describes  those Chebychev covers with $$\ochar=2+(1/3-1) \np 2(1/2\nm1)> 0.$$ 

 \S\ref{handling1psigma} discusses the many other Nielsen classes with $G=\PSL_2(\bZ/2)$ and the components of their fiber products (Cor.~\ref{findeg7polyb}). They are all satellites of one Nielsen class.    These examples cover most of the territory by example. Both the author and Schinzel are/were number theorists. So, \S\ref{sigH7} includes a number theory discussion of these Hurwitz spaces.\footnote{There is a lesson here for those who think all reasonable moduli spaces have their representing objects given only by coordinates on their parameter spaces.}  

\subsubsection{Handling ${}_j\psigma$, $j=1,2$,  in \eqref{7branchcyclesb}}  \label{handling2psigma} 

Use this Nielsen class ($j=1,2$): 

\begin{equation} \label{deg3ex2-2nd} \begin{array}{c} \ni(\PGL_2(\bZ/2), \bfC_{2\cdot3\cdot7}, \bT\eqdef (T_1,T_2))\text{ of \eqref{bcfg} with representing}
\\ \text{branch cycles } (( {}_j\sigma_1, {}_j\tau_1),
({}_j\sigma_2, {}_j\tau_2), (( {}_j\sigma_1\cdot {}_j\sigma_2)^{\nm 1},({}_j\tau_1\cdot {}_j\tau_2)^{\nm1})). \end{array} \end{equation} A fiber product representative would be labeled $\tilde \sC_{{}_jf^*, {}_jg^*}$ corresponding to $T_{{}_jf^*}=T_1$ and $T_{{}_jg^*}=T_2$, both of degree 7. 

Prop.~\ref{compinc} classifies components on these fiber products. We figure what restrictions (Lem.~\ref{avoidg1s}) on $g_1 \in \bC(y_2)$ allow extending  Thm.~\ref{PakThm} to  $\tilde \sC_{{}_jf^*, {}_jg^*\circ g_1}$. 
\begin{edesc} \label{proccomps} \item \label{proccompsa}  List restrictions on $g_1$ satisfying Nielsen class conditions; and 
\item \label{proccompsb} excepting \eql{proccomps}{proccompsa}, show the two certain components on $\tilde \sC_{{}_jf^*, {}_jg^*\circ g_1}$  are the only ones and their genuses rise with $\deg(g_1)$. \end{edesc}

Use the notation of Rem.~\ref{appear7ochar}. For example, ${}_2W_1$ (resp.~${}_2W_2$) is the degree 3 (resp.~degree 4) component of $\tilde \sC_{{}_2f^*,{}_2g^*}$ for the second Nielsen class in \S\ref{bcdeg7}: denoted there $\bfC_{\infty\cdot2\cdot3}$; here $\bfC_{2\cdot3\cdot7}$, $\C_\infty=\C_7$.  

Also, for a fiber product component $W$,  we are calculating branch cycles for $\pr_y: W \to \prP^1_y$ as in Prop.~\ref{compinc}. The opening paragraph alludes to the representation given in Lem.~\ref{compcor}. 

\begin{cor} \label{findeg7polya} The cover ${}_2W_1\to \prP^1_z$ (resp.~${}_2W_2\to \prP^1_z$) has associated representation on index 21 (resp.~28) cosets of a 2-Sylow (resp.~$D_3$).  These are respectively defined by $G(T_1,x_j)\cap G(T_2,y_1)$ where $j\in J_1$ (resp.~$J_2$) the orbits of $G(T_1,y_1)$ on $\row x 7$, as in \eqref{staby}.  

Here are the results for the degree 4 $pr_y$ covers. 
\begin{edesc} \label{prycov2} \item \label{prycov2a} \!\!Degree 3 ${}_2\pr_y$: ${}_2W_1\to \prP^1_y$ has group $D_3$ and  genus 0 Galois closure. \item \label{prycov2b}  \!\!Degree 4  ${}_2\pr_y$: ${}_2W_2\to \prP^1_y$ has group  $S_4$  and  genus 3  Galois closure.
\end{edesc}
The Nielsen class of \eql{prycov2}{prycov2a} is of a degree 3 Chebychev polynomial. 

Here are the comparable results for the degree 3 $\pr_y$ covers.   
\begin{edesc} \label{prycov1} \item \label{prycov1a} \!\!Degree 3 ${}_1\pr_y$: ${}_1W_1\to \prP^1_y$ has group $D_3$ and  genus 1 Galois closure.
\item \label{prycov1b} \!\!Degree 4  ${}_1\pr_y$: ${}_1W_2\to \prP^1_y$ has group  $D_3$  and  genus 0  Galois closure.
\end{edesc} 

All, except the $S_4$ cover of \eql{prycov2}{prycov2b},  are in the \eql{Pak}{Paka} excluded cases. 
 
\end{cor} 
\begin{proof} 
A well-known computation gives the order of any group of the form $\GL_m(R)$ where $R$ is a finite ring. So, $$|\PSL_2(\bZ/2)|=(2^3\nm1)(2^3\nm 2)(2^3\nm 2^2).$$

Both $G(T_1,x_1)$ and $G(T_2,y_1)$ -- by their orders -- contain a 2-Sylow. By the Sylow Theorems, the 2-Sylows are conjugate. So,  for some $j$, $G(T_1,x_j)$ contains the same 2-Sylow as does $G(T_2,y_1)$. The two groups define different permutation representations, so cannot be equal: $G(T_1,x_j)\cap G(T_2,y_1)$ is exactly the 2-Sylow. The same argument for the 3-Sylows, says there is a $j'$ such that $G(T_1,x_{j'})\cap G(T_2,y_1)$ is exactly a $D_3$.  

The component ${}_2W_v$ defines a subgroup $H_v$, which we may assume is contained in  $G(T_2,y_1)$ and  also in one of the  conjugates of $G(T_1,x_1)$, stabilized by conjugating by $G(T_2,y_1)$. These containments correspond to the projection maps of ${}_2W_v$ on $\prP^1_x$ and $\prP^1_y$. The order of $H_v$ in each of the two cases in the paragraph above characterizes whether $H_v$ is a 2-Sylow or a $D_3$. That completes the first paragraph of the proposition. 

Working on the Nielsen class with conjugacy classes ${}_2\bfC=\bfC_{2\cdot3\cdot7}$, to find conjugacy classes of branch cycles for  $\pr_y: {}_2W_v\to \prP^1_y$,  take powers of conjugations of $({}_2 \sigma_i,{}_2 \tau_i)$, $i=1,2$, that are fixed on $y_1$. Then, see what they do to the orbit corresponding to $J_v$. We don't need the subscript ${}_7$ (or ${}_\infty$) because that only fixes $y_1$ if you take the trivial power. We have already done this just to find the orbits $J_1$ and $J_2$. 

The two separated expressions in \eqref{bcfg2} are the analogs of \eqref{exonea}. For ${}_2W_1$, we look to the disjoint cycles with ${}^*$ superscripts.  The branch cycle from one branch point $y'_1$ is a 3-cycle corresponding to $({}_2\sigma_1,{}_2\tau_1)$. There are  two  branch points, $y'_2,y'_3$; both have 2-cycle branch cycles corresponding to  $({}_2\sigma_2,{}_2\tau_2)$. 

For the component ${}_2W_2$, we get a 3-cycle corresponding to an orbit of $\sigma_7\cdot  {}_2 \sigma_1\cdot \sigma_7^{-1}$ for a branch point $y_1''$ and three branch points $y_2'',y_3'',y_4''$ with branch cycles (respectively) of type $(2), (2)(2),2$. So, this degree 4 cover has branch cycles with a 3-cycle and a 2-cycle, so it must have group $S_4$. Apply  \eqref{orb} to compute the Galois closure has genus $3 (>1)$.  

We will be briefer on the examples ${}_1W_1$ and $_1W_2$ of \eqref{prycov1}. For ${}_1W_1$, using \eqref{exonea}each of  the branch cycles $({}_1\sigma_k,\tau_k)$, $k=1,2$, gives two branch points on $\prP^1_y$, each having a 2-cycle as branch cycle.  The  Galois closure has genus 1. For ${}2W_1$, using \eqref{bcfg2}, the branch cycle for $({}_2\sigma_1,\tau_1)$ gives one branch point and corresponding branch cycle a 3-cycle. For $({}_2\sigma_2,\tau_2)$ we get two branch points and corresponding 2-cycles. Again, this is a Chebotarev case. 
\end{proof} 

\subsubsection{Satellite Nielsen classes related to $\PSL_2(\bZ/2)$} \label{handling1psigma} Cor.~\ref{findeg7polyb} shows  what we need to generalize Thm.~\ref{PakThm} when $f$ falls among a reasonably limited set of Nielsen classes. 
 Let $\ni(G,\bfC_b)$ and $\ni(G,\bfC_t)$ be two Nielsen classes with the same group $G$ ($b$ is for bottom, $t$ is for top). 

\begin{defn}[Nielsen satellites] \label{satellite} We say $\ni(G,\bfC_b)$  is a {\sl satellite\/} of $\ni(G,\bfC_t)$, if each $\psigma\in \ni(G,\bfC_b)$ is a coalescing from some  $\psigma^*\in \ni(G,\bfC_t)$. \end{defn}

Notation is as in \S\ref{7branchcyclesa}  where $T_1=T_f$ (resp.~$T_2=T_g$), etc.  
\begin{cor} \label{satellitecor} The branch cycles \eqref{bcfg} -- produced by the coalescings \eqref{7branchcyclesb} -- represent the only Nielsen classes, 
$$\ni_{2\cdot4\cdot7}\eqdef \ni(\PSL_2(\bZ/2),\bfC_{2\cdot4\cdot7},\bT)\text{  and }\ni_{3\cdot2\cdot7}\eqdef\ni(\PSL_2(\bZ/2),\bfC_{3\cdot2\cdot7},\bT)$$ 
with a pair of degree 7 representations $(T_1,T_2)$ and conjugacy classes $\bfC$ with  group $G$ containing $\C_7$ and  for $f\in \ni(G,\bfC, T_1)$ there is $g\in  \ni(G,\bfC, T_2)$ with $\tilde\sC_{f,g}$  reducible, having a genus 0 component.   

Further,  $\ni_{2\cdot4\cdot7}$ and $\ni_{3\cdot2\cdot7}$  each consist of precisely six elements and $$\ni(\PSL_3(\bZ/2),\bfC_{2^3\cdot 7}),\bT)\text{ has both as satellites.}$$  

So, all degree 7 Davenport pairs, $(f,g)$, correspond to points on a  compactification of the irreducible space $\sH_7$, with  points corresponding to the Nielsen class $\ni(\PSL_2(\bZ/2),\bfC_{2^3\cdot 7},\bT)$ off the boundary. 
\end{cor} 

\begin{proof} 
Expression \eqref{7branchcyclesb}  has the coalescings that give the two Nielsen classes above.  
On the other hand, from ${}_1\psigma$-coalesce, we get the result of coalescing 4 other Nielsen class elements by the following device. 

\begin{edesc} \label{taumu} \item $\tau=(1\,4\,6\,7)(2\,3)$ is the product of the two entries of $$\mu=( (1\,6)(2\,3),(6\,4)(1\,7)).$$ \item \label{taumub} Conjugate $\mu$ by $\tau$   to see that 4 elements in the Nielsen class (with $\sigma_\infty$ in the 4th position) give exactly the same coalescings. 
\item Achieve \eql{taumu}{taumub}  by the powers of the braid $\sh q_2^2\sh$. 
\end{edesc} 

The references above \eqref{bcfg} show that all possible branch cycles for Nielsen classes of degree 7,  containing $\sigma_7$, are contained in those listed in \eqref{taumu}, We now know these are coalescings  from the single braid orbit comprising $\ni(\PSL_2(\bZ/2),\bfC_{2^3\cdot 7})^{\abs_{1,2}}$. There are only two conjugacy classes of degree 7 elements in $\PGL_2(\bZ/2)$, completing all the Nielsen classes with Davenport pair representatives. Therefore the Nielsen classes we have listed are satellites of $\ni(\PSL_2(\bZ/2),\bfC_{2^3\cdot 7},\bT)$. 

See the beginning of \S\ref{contextdeg7}, for initial discussion of $\sH_7$, and then \S\ref{sigH7} for the compactification we refer to. Finally, we document that both $\ni_{2\cdot4\cdot7}$ and $\ni_{3\cdot2\cdot7}$  consist of precisely six elements. 

First, we have the natural map $\psi: H_3\to S_3$ given by the effect of  $q\in H_3$ on the order of the conjugacy classes. There are 3 distinct conjugacy classes in $\bfC$ modulo $N_{S_7}(G)$ (because the classes have distinct orders) and Lem.~\ref{H3action} implies $\psi$ is surjective. There are exactly 6 elements in each $H_3$ orbit in $\ni_{2\cdot4\cdot7}$ and $\ni_{3\cdot2\cdot7}$. 

Suppose each involution in a 3-tuple in one of our Nielsen classes, say $\ni_{2\cdot4\cdot7}$,  conjugates to any other by some power of $\sigma_\infty$. Assume further,  $$\bg=(g_1,g_2,g_3),\ \bg'=(g_1',g_2',g_3)\in \ni_{2\cdot4\cdot7}$$ with say $g_1$ and $g_1'$ involutions and $g_3=\sigma_\infty^{-1}$. Again from product one, if $\sigma_\infty^i$ conjugates $\bg$ to $(g_1',\sigma_\infty^ig_2\sigma_\infty^{-i},g_3)$, then $\sigma_\infty^ig_2\sigma_\infty^{-i}$ is $g_2'$. Conclude: $\bg$ and $\bg'$ are the same in  $\ni_{2\cdot4\cdot7}^\abs$. So there are six elements in this Nielsen class. Involutions, however, in this example (indeed all examples in the groups of  \eqref{genus0reduct}), correspond to transvections (\S\ref{7branchcyclesa}). So, they are conjugate. 

You can check this directly using the 7 Nielsen class representatives in $\ni(\GL_3(\bZ/2),\bfC_{2^3\cdot 7}, T_1)$ listed, say, in \cite[\S3.3]{Fr05b}. The same applies to $\ni(\GL_3(\bZ/2), \bfC_{2\cdot3\cdot7}, T_1)$. This easily completes the theorem.  \end{proof} 

\begin{rem}[Wreath products]  \label{odds-ends} We have typically labeled a Nielsen class with a group and a collection of conjugacy classes, and that can be done for $\ni_{2,\deg\nm3,m}$, with $(m,3)=1$. Indeed, for any composite cover $$f_1\circ f_2: X_2\to X_1 \to Z,$$ we may describe $G_{f_1\circ f_2}$ as a subgroup of the wreath product of the two covers, with their natural permutation reps. entwined \cite[\S2]{Fr70}. This is a fruitful way when, as in \cite{BiFr86}, $G$ is close to, even if not quite equal, to the full wreath product. 

Here it is not the full wreath product. The same $\bZ/2$ quotient of $D_m$ is mapped onto by $\GL_3(\bZ/2)$ in its map to $\prP^1_y$, for all $m$. We don't make use of it here, so give no details. More extreme even is  when $\tilde \sC_{f,g}$ has a genus 0 component, producing the cover $h_1:C\to \prP^1_y$ as in \S\ref{irreducibilityb}. Then, the Galois closure of $g\circ h_1$ is still just $G$. 

Also, we haven't done any detail on the case when $(m,2)=2$, as in Ex.~\ref{evendihedral}. The time for doing that would be when it plays a vital role in some example, as it doesn't work the same as the case $m$ is odd. 
\end{rem}

\subsubsection{Significance of $\sH_7$}  \label{sigH7}  The Hurwitz space $\sH_7$ is one case of the varieties constructed in \cite[\S4]{Fr77}. For almost all Nielsen classes in this paper, such as $\ni_{2^3\cdot7}$, each of those spaces, $\sH$, is a {\sl fine moduli space}, because the stabilizing group $G(T,1)$ self-normalizes in $G$ \cite[\S4, Prop.~3]{Fr77}. The remainder of this subsection explains more on how the construction of \cite[\S4]{Fr77} combined with \cite[Thm.~3.21]{Fr95b} shows the {\sl cohering\/} of those satellites. 

First: $\sH$ is an affine variety with a total space structure, $\sT$, over it: $\sT$ is a cover of $\sH\times \prP^1_z$ with this property. For $\bp\in \sH$, a  fiber of $\sT$ over $\bp\times \prP^1_z$ represents a  cover in the equivalence class of $\bp$. This was shown by producing complex analytic coordinates \cite[\S 4.B, pgs.~49--53]{Fr77}, and then applying a famous result of Grauert and Remmert \cite{GraR57}: An analytic (unramified) cover $W$ of a quasiprojective variety $V$ is quasi projective. So, you may complete $W$ to a projective variety $\bar W$ by normalization of $\bar\sH$ in the function field of $W$. 

\cite[Thm.~5.1]{Fr77}  gives, from the {\sl Branch Cycle Lemma\/}, the definition field of $\sH$ as a {\sl moduli space}. We can understand that nicely when fine moduli holds to be the well-defined minimal definition of $\sT$ with its structural maps. Particularly it says that the two connected families of Davenport polynomials with respective Nielsen classes $\ni(GL_3(\bZ/2),\bfC_{2^3\cdot7})^\abs$ and $\ni(GL_3(\bZ/2),\bfC_{2^3\cdot7'})^\abs$, are conjugate over $\bQ(\sqrt{-7})$. 

Continuing in generality, use the compactifications $\bar \sH$ and $\bar \sT$ as projective varieties with extending maps $\bar \sT\to \bar\sH\times \prP^1_z$, through normalization of (components $\bar\sH'$ of) $\bar\sH\times \prP^1_z$ in the function field of (components $\sT'$ of) $\sT$. We recover the families of covers attached to satellite Nielsen classes by inductively coalescing, using this {\sl normalization stratification\/} of $\bar \sH$: 
\begin{edesc} \label{compactify}  \item  \label{compactifya} Restrict $\bar \sT'$ over the boundary $\bar\sH'\times \prP^1_z \setminus \sH'\times \prP^1_z$. \item \label{compactifyb} Normalize (components of) that result and identify open unions of subsets of them as spaces of covers attached to Nielsen classes. 
\item  Continue inductively on the dimension from \eql{compactify}{compactifya} applied to the Nielsen classes in \eql{compactify}{compactifyb}. \end{edesc} 
  
\cite[proof of Thm.~3.21 and Lem.~3.22]{Fr95b} carried out these steps, under the names {\sl specialization sequences\/} and {\sl coalescing operators}. 

This means there is a path from a point on $\sH$ to any point on the space representing a satellite Nielsen class, a path that runs through a family of nonsingular covers. This is the case, say,  for a cover starting at any element in the Nielsen class $\ni(\PSL_2(\bZ/2),\bfC_{2^3\cdot7},\bT)$ to a cover in the Nielsen class  $\ni_{2\cdot4\cdot7}$, or the Nielsen class of the other satellite. A different,  Deligne-Mumford  style, compactification was constructed by  \cite{DeEm99} and \cite{We99}), motivated by the application \cite[Thm.~3.21]{Fr95b}. 

\subsection{The Main Theorem} \label{genPak}  \S\ref{touchPak}, with reminders of \cite{Pak22}, uses our approach to mitigate Pakovich's skepticism of an effective genus formula for components of $\tilde \sC_{f,g}$ when it is reducible.  

\S\ref{maintheorem}, based on \S\ref{setPak}, gives the Nielsen class formulation for avoiding $g_1\,$s that might not give a generalization of Thm.~\ref{PakThm}. The genus of $\tilde \sC_{f^*,g^*\circ g_1}$ rises with $\deg(g_1)$ for $(f^*,g^*)$ in a Nielsen class for which $\tilde \sC_{f^*,g^*}$ is reducible, and even has components of genus 0. 

\subsubsection{Results of \cite{Pak22}} \label{touchPak} 

Pakovich switches the reference to the degrees of $f$ and $g$: His $m$ is my $n$, et. cet. I keep mine, but I also use his (compatible to me) $k$ and $\ell$ for the respective bi-degrees of projection of a component, $W$, of $\tilde \sC_{f,g}$ onto $\prP^1_x$ and $\prP^1_y$.\footnote{The notation memorably extends to changing $\prP^1_x$, $\prP^1_y$, $\prP^1_z$ respectively to $X$, $Y$, $Z$.} 

\cite[p. 2]{Pak22} is skeptical about an effective genus formula for components. Since I have given such,  l put it this way. 
\begin{edesc} \label{effgenus} \item \label{effgenusa}  Using Nielsen classes, Prop.~\ref{redRed} gives a handle using the minimal left composition factor of $f$. 
\item \label{effgenusb}  As in \eql{effgenus}{effgenusa}, divide what happens with that minimal left composition factor according to properties listed for specific primitive groups by the solution of the genus 0 problem. 
\item \label{effgenusc} The data, for example, Nielsen classes and the representation results of the genus 0 problem, are programmable.  \footnote{With experience, one easily detects the likelihood of, say, Schur covering, Davenport pair, and Hilbert-Siegel situations with rational functions -- generalizing the known polynomial results -- directly from the extant results.}  
\end{edesc} 

By programmable in \eql{effgenus}{effgenusc}, we would include {\bf Mathematica\/}, {\bf Maple\/}, {\bf GAP}, {\bf Cayley}, \dots.  It would be part of applying programs to include the coelescings catching the genus 0 or 1 Nielsen classes. 

For example, start from $\ni(\PSL_2(\bZ/2), \bfC_{2^6},\bT)$, the most general degree 7 case in 
Ex.~\ref{ratDeg7},  with components of  $\tilde \sC_{f^*,g^*}$ having genus 4. A program would find the coalesced Nielsen classes of  Rem.~\ref{comporbit} for our running example, wherein each has  $k=3$ and $\ell=4$ and component genuses 0 or 1. 

If  $\tilde \sC_{f,f}$  has (excluding the diagonal) no component of genus 0 or 1, \cite[p. 3]{Pak22}  refers to it as {\sl tame}.\footnote{For example, if the natural representation for $f$ is double transitive, then Cor.~\ref{2-foldfiber} gives the computation of the genus of the non-diagonal component of $\tilde \sC_{f,f}$.} \cite{Pak18a}: Up to reduced equivalence (Def.~\ref{covequiv}) for a rational function $f$, $\hat f$ has genus 0 or 1 only for cyclic, Chebychev and the Cheybchev Galois closure functions, plus a finite number of other situations. \cite{Pak22} notes the tame $f\,$s have Galois closures of genus exceeding 1. 

We note Pakovich's other skeptical statements in Rem.~\ref{PakSkep}.  We should go beyond the unrepresentative (and rare) anomalies of \S\ref{orbzero} or those with extreme,  but often insignificant, coefficient patterns. The goal would be to progress in connecting the behavior of different algebraic functions to, say,  Galois/fundamental group theory and $\ell$-adic representations.  

\begin{thm} If $\tilde \sC_{f,g}$ is irreducible, of genus $>1$, then its genus  exceeds $\frac{n\nm 84 m \np 168}{84}$, unless the genus of $\hat f$ is $\le 1$ \cite[Thm.~2]{Pak18b}. \end{thm}

\cite[p. 4]{Pak22} generalizes the fiber product diagram to specialize to where he can drop the irreducibility hypothesis and replace $\tilde \sC_{f,g}$ by a component $W$. In his generalization, he introduces the concept of having a {\sl right composition factor}. In this case, the projections of $W$ to $\prP^1_x$ and $\prP^1_y$ would factor (nontrivially) through another cover $W'$. The 1st paragraph of Prop.~\ref{compimage}, though, excludes this. 

\cite[(10)]{Pak22} gives a formula using the bi-degree $(k,\ell)$ with $$k =\deg(\tilde \sC_{f^*,g^*}/\prP^1_x)>1, \ell=\deg(\tilde \sC_{f^*,g^*}/\prP^1_y)$$ involving the use of the higher fiber products of $f$ as in \S\ref{gcc} with the fiber product description of the Galois closure. Here, though we can revert to the framework of Prop.~\ref{redRed} and consider $\tilde \sC_{f^*,g^*\circ g_1}$ with $f^*$ and $g^*$ having the same Galois closures. It says, if  $f^*$ is tame, then  
\begin{equation} \label{paknum} \geng_W > 2 -m + n/m! \text{ unless }g^*\circ g_1 =f^*\circ f_1 \text{ (and $W$ as in Ex.~\ref{diagcomp1})}. \end{equation}  

The genus bound from \eqref{paknum} is much smaller than 0. It wouldn't point at the results of Rem.~\ref{comporbit} nor to the explicit  Ex.~\ref{extpakdeg7} for the Pakovich formulation in the 4 degree 7 cases where there is a genus 0 or 1 component.

%b He notes the more general Riemann-Hurwitz replacing \prP^1_z$ by Z, \prP^1_x by X and \prP^1_y by Y, g(\tilde \sC_{\phi_{X,Z},\phi_{Y\to Z}))\ge g(X) -1)(\deg(\tilde \sC_{\phi_{X,Z}/X)  +1. 

\begin{rem} \label{PakSkep} \cite[]{Pak22} says obtaining a full classification of components on $\tilde \sC_{f^*,g^*}$ of genus zero or one using the genus formula seems to be hardly possible. In addition, such an analysis results only in possible patterns of ramifications of $f^*$ and $g^*$. Pakovich calls this the {\sl Hurwitz problem\/}.  \cite{Pak09} provides Laurant polynomials, $u(1/z)+v(z)$ with $u,v\in \bC[x]$, examples.

This statement makes sense if your only input is from a combinatorial genus formula. The Nielsen class approach, though, suggests information on examples available for specific problems related to covers. There you would try to find what groups $G$ offer possible solutions, as with the Schur, Davenport, etc. problems.  A more precise result on appropriate  Nielsen classes  would ask when $\ni(G,\bfC, T)$ is non-empty. \end{rem} 

\begin{rem} \label{PakRev} \cite[p. 300]{Pak18b} states Ritt's second Theorem. Classifying curves $\tilde \sC_{f,g}$ of genus 0 when $f$ and $g$ are polynomials. 
He also includes the result with at most two points at $\infty$ in \cite{BT00}. 
Suppose $\{{}_ig_1\}_{i=0}^\infty$, rational functions, with $\deg({}_ig_1)\mapsto \infty$.  
He refers to an $f$ that has $\tilde \sC_{f,{}_ig_i}$ of genus 0 and {\sl irreducible\/} for all i, {\sl  a basis of curves of genus 0\/} \cite[p. 301]{Pak18b}. \cite[Thm.~1]{Pak18b} says this happens only if $\hat f$ has genus zero or one. \end{rem}
%Are those Bilu-Tichy curves coming up in my degree 7 examples?  

\subsubsection{The Main Theorem extending Thm.~\ref{PakThm}}\label{maintheorem} 
Prop.~\ref{compinc} gives a list of component types that can occur on a representative fiber product $\tilde \sC_{f^*,g^*}$ in a Nielsen class $\ni(G,\bfC,\bT)$ where these are reducible.  Following Pakovich, we have already excluded consideration of $f^*$ for which the Galois closure of $f^*$ has genus 0 or 1: condition \eql{Pak}{Paka}.\footnote{As we take considerable space, especially in \S\ref{C_2^4}, to consider this arising, there is reason to look further into this case, but not here.}  

\begin{defn} Property $P$ of fiber product representatives in $\ni(G,\bfC,\bT)$ is a Nielsen (class) invariant if it is constant on  the Nielsen class. By contrast, $P$ is a  {\sl braid invariant\/} if it is only constant on braid orbits. 
\end{defn} 

 All properties we consider are braid invariant, though not necessarily obviously so. 
Def.~\ref{niGCbT} has the action of $N_{S_m}(G,\bfC)\cap N_{S_n}(G,\bfC)$  used to define $\ni(G,\bfC,\bT)$ in \eqref{inntoabs} by quotienting out by this action. As we noted in Rem.~\ref{Hrbasics}, these outer automorphisms are not always braided. When they aren't, they must be added to get the correct equivalence between covers on those braid orbits. 
We know there is just one braid orbit in the cases we have used as examples.\footnote{We also know many Nielsen classes where there is more than one braid orbit. Example: The {\sl Lift Invariant\/} in the the Main Theorem of \cite{Fr10}.} 

To generalize Thm.~\ref{PakThm} we show the genus of representatives  $\tilde \sC_{f^*,g^*\circ g_1}$ goes up with the degree of $g_1$ with four possible exclusions. All exceptions are made using properties associated to fiber products denoted $\tilde \sC_{\pr_{y,j},g_1}$ in Prop.~\ref{compinc}. Here $\pr_{y,j}=\pr_{W_j}: W_j\to \prP^1_y$ with $W_j$ a component of the fiber product.  

There is a special situation, given by \eql{fail2}{fail2c} when $W_j$ has genus 1 requiring Def.~\ref{sigmadominatetau} and Lem.~\ref{genus1}. 
\begin{defn} \label{sigmadominatetau} Given $\sigma\in S_n$ and  $\tau\in S_{n'}$, we say $\sigma$ dominates $\tau$ if the length of each disjoint cycle of $\sigma$ is a multiple of the order of $\tau$. \end{defn}

We use the list of component types following the Prop.~ (and comments on them); \eqref{fail2} is a reminder of what to avoid.

For a cover $\phi_X: X\to \prP^1_y$ denote its branch points by $\bz_{\phi_X}$ and a branch cycle for $z'\in \bz_{\phi_X}$ by $g_{\phi_X,z'}$. 
\begin{edesc} \label{fail2} \item  \label{fail2a}  Condition \eql{Pak}{Paka}: Galois closure of $\pr_{y,j}$  has genus $\le 1$. 
\item  \label{fail2b}  Condition \eql{W_jcomps}{W_jcompsa}: $g_1$ giving any genus 0 quotients of  $\pr_{y,j}$. 
\item \label{fail2c}  Condition \eql{W_jcomps}{W_jcompsb}: with $\geng_{W_j}=1$, $g_1$ with $\bz_{g_1}\subset \bz_{\pr_{y,j}}$  and for $z'\in \bz_{g_1}$, $g_{\phi_X,z'}$ dominates $g_{g_1,z'}$. 
\item \label{fail2d}  Condition \eql{W_jcomps}{W_jcompsc}: when $\pr_{y,j}=h_j$, $g_1$ entangled with $h_j$, giving a new reducible $\tilde \sC_{h_j,g_1}$. \end{edesc} 

\begin{lem} \label{genus1} Suppose $W_j$ has genus 1, and there are $s$ branch cycles in $\bg_{g_1}$ with respective orders  $\row d s$.  Then, \eql{fail2}{fail2b} implies \begin{equation} \label{ordram} 2\deg(W_j)\ge \sum_{i=1}\deg(W_j)(\frac{d_i\nm1}{d_i}).\end{equation} In particular, this situation can only occur with $s\ge 4$ if $s=4$ and all $d_i\,$ are 2. In that case, $g_1$ is a cover with galois closure genus 1 (given in \S\ref{ochar2^2infty}). \end{lem} 

\begin{proof} The condition on the left of \eqref{ordram} is the sum of the indices of branch cycles for $\pr_{W_j}$ assuming $W_j$ has genus 1, and the condition on the right is the minimal value for the sum of the indices of branch cycles corresponding to $z'\in \bz_{g_1}$. The rest of the Lemma is evident from the characterization of \S\ref{ochar2^2infty} since the minimal value of the index is $\frac{n}2$, and that only when the orders of the branch cycles are 2. \end{proof} 

Prop.~\ref{ncinterp} gives the following general principle. 
\begin{prop} \label{ncinterp} We can express branch cycles for the Nielsen class of $\pr_{y,j}$ in terms of those for $\tilde \sC_{f^*,g^*}$. From this, in each case $v$ of \eqref{fail2} -- possibly excluding \eql{fail2}{fail2c} -- there is a (finite collection of) Nielsen class(s) $\ni_v$ containing those $g_1\,$s to avoid to assure the genus rises with $\deg(g_1)$. 

In case \eql{fail2}{fail2c}, using notation of Lem.~\ref{genus1}, either $s=4$ and the Galois closure of $g_1$ has genus 1, or $s=3$, and exceptional cases of \eql{fail2}{fail2c} are given by classical triangle groups; the Galois closure of $g_1$ has genus $\le 1$. 
\end{prop} 

\begin{proof}  The algorithm for finding the Nielsen class of $\pr_{y,j}$ from that of $\tilde \sC_{f^*,g^*}$ is given by Cor.~\ref{prync}, with the illustrating examples of Cor.~\ref{findeg7polya} the culmination of our degree 7 running example. For each item in \eqref{fail2}, we comment below on why it has a Nielsen class description stemming from that of $\ni(G,\bfC,\bT)$. Denote the Nielsen class of $\pr_{y,j}$ by $\ni_{\pr_{y,j}}$. 

{\sl Detecting condition \eql{Pak}{Paka}:} Cor.~\ref{findeg7polya} has given examples of using the \o-char\ to find the Galois closure genus and also examples of detecting the main cases of that condition happening from the list of \S\ref{orbzero}.  

{\sl Nielsen description of condition \eql{fail2}{fail2b}:} Given branch cycle representatives $\bg$ of a Nielsen class $\ni(G,\bfC,T_W)$ of irreducible covers $\phi_W: W\to \prP^1_y$, we consider how to find branch cycles for the genus 0 covers $\phi_{W'}: W'\to \prP^1_y$ through which it factors. 

The subgroup $G(T_{W},1)$ defines $W$, and each cover through which $\phi_W$ factors is defined as a quotient of $\hat \phi_W$ by a group $G(T_W,1)\le H \le G$ from the quotient $\hat \phi_W/H$. Compute the genus of this cover from its branch cycles, which are given by $\bg \mod H$: The action of each entry on the distinct cosets of $H$. Now apply this to the Nielsen class of $\pr_W$. 

{\sl Nielsen description of condition \eql{fail2}{fail2c}:} As in the case above, given $\bg$ an $s$-tuple of $S_n$ ($n=\deg(\phi_W)$), the Nielsen classes  to avoid contain elements $\bg'\in (S_{n'})^s$, with $n'$ arbitrary, subject to these conditions: 
\begin{edesc} \item $g_i$ dominates $g_i'$, $1\le i\le s$; and \item $\lrang{\bg'}$ is a transitive subgroup of $S_{n'}$, satisfies product-one  and \RH\ gives the genus of a cover with branch cycles $\bg'$ as 0. 
\end{edesc} 
Lem.~\ref{genus1} covers this situation by forcing $g_1$ to be in the excluded case of having its Galois closure (as a cover) of genus $\le 1$. 

{\sl Nielsen description of condition \eql{fail2}{fail2d}:}  We have already suggested this situation is the complicated one, but that is because it calls for knowing if there is another representation of $G$ that can be entangled with $T_1$. Finding out about such examples would go into the solution of the genus 0 problem beyond the list of Thm.\ref{polyPakThm} relating to the diagram \eqref{genus0diag}. 

Though this case is more complicated than the others, there still is a  Nielsen class description of the result. We need not exclude this situation except if it requires avoiding what we have already considered above. 
\end{proof}

\begin{thm} \label{genPakThmGen}  Prop.~\ref{ncinterp} has gone through the list of types of components of $\tilde \sC_{f^*,g^*}$, and noted a Nielsen class description -- based on that of $\ni(G,\bfC,\bT)$ in which $\tilde \sC_{f^*,g^*}$ falls. It has given a Nielsen class description of what to avoid with the covers given by $g_1$: 
\begin{edesc} \item Anyplace where the Galois closure of a rational function cover has genus 0 or 1; and 
\item \eql{fail2}{fail2b} where you must consider the genus 0 components of the starting Nielsen class represented by $\tilde \sC_{f^*,g^*}$. \end{edesc} 
It also gives leeway to avoid \eql{fail2}{fail2c}; that gets into an induction argument. In what is allowed, the genus increases with the degree of $g_1$. But see Rem.~\ref{ocharrising}. 
\end{thm} 

We conclude with final words on examples, beyond the many we have given based on $\ni(\PSL_2(\bZ/2),\bfC,\bT)$.  
\eqref{genus0mon} has divided the genus 0 primitive monodromy groups into two cases. In \eql{genus0mon}{genus0monc} are those related to small semidirect products and alternating groups. 

Now consider the primitive genus 0 monodromy related to alternating groups.  \S\ref{doubledegree}  uses them to produce  many examples of reducible fiber products with genus 0 components. That seems to illustrate alternating groups as the {\sl untamable\/} case of the (primitive) genus 0 problem. Yet, \S\ref{appSpaces} reminds us of a (serious diophantine motivated, really!)  alternating group series (before \cite{GSh07}) with a surprising conclusion. Compatible with the genus 0 result, from infinitely many possible examples, only one degree fulfilled the constraints on the problem and produced rich diophantine examples.  

Among the exceptional genus 0 groups of  \eql{genus0mon}{genus0mond}, there are only finitely many such $(G,T)$. Yet, the number is huge. A myriad of different authors contributed, applying in each case their expertise on pieces of the (simple group) classification through \cite{AOS85}. We doubt any tractable classification  could come from these for a full, explicit, Pakovich extension. That is, with the present state of the classification, that list might as well be infinite even without dropping  the primitive monodromy condition. 

\begin{rem} \label{ocharrising} With the Pakovich assumptions in Thm.~\ref{PakThm}, the genus of the irreducible $\tilde \sC_{f,g}$ rises with $\deg(g_1)$. He doesn't give the formulation using Prop.~\ref{redRed} that considers testing for reducibility. Nevertheless, the genus goes up, but we have been less careful about being explicit about its rise. \end{rem} 

\begin{appendix} \section{Comments on fiber products} \label{fpcomments} \S\ref{gcc} presents the Galois closure of a cover $f:X \to Z$ as defined by a fiber product. \cite{Fr10} and \cite{Fr12} both do Hurwitz space versions of this for conclusions about families of Galois closure covers. From the universality of fiber products this gives a tight connection between the Galois closure of a cover $f:\prP^1_x\to \prP^1_z$ and its Galois closure $\hat f$, especially considering the condition \eql{redGal}{redGala} in Prop.~\ref{redRed} when dealing with components of $\tilde \sC_{f^*,g^*}$ for $\hat f^*$ and $\hat g^*$ as equivalent covers. 

 \S\ref{orbzero} refers to the family of genus 0 or 1 covers of $\tilde \sC_{f,g_1}$ that arise from negating condition \eql{Pak}{Paka} as an {\sl $\ochar$-fan}. They {\sl fan\/} into a web of $g_1$ values for which $\tilde \sC_{f,g\circ g_1}$ remains irreducible, but still has bounded genus; unlike the other possible $g_1\,$s for which the genus rises with $\deg(g_1)$. 
 
\subsection{Presenting the Galois closure cover} \label{gcc} To form the Galois closure, $\hat f: \hat X_f \to Z$, of a cover $f:X \to Z$, start with the fiber product of $f$, $m=\deg(f)$ times. As usual, normalize, then remove the fat diagonal (loci where two coordinates are equal) to get $\hat X_m$. 

Finally,  take a connected component $\hat X(m)$, of $\hat X_m$, with its natural projection (restriction of) $\hat f: \hat X(m) \to Z$. That's the Galois closure: By restricting the natural action of $S_m$ to it, the stabilizer of the component has exactly the right order to be the Galois closure (by the fundamental theorem of Galois theory). This works over any field $F$ (of characteristic 0; even, with care about inseparability, in characteristic $p$), but here take $F=\bC$. 

Denote  the {\sl decomposition group\/} -- subgroup of $S_m$,  in its natural action on $\hat X_m$, whose elements map $\hat X(m)$ into itself -- by $G(\hat X(m)/Z)$. 

We can also construct $\hat X(k)$, based on the fiber product $k\le m$ times. Given $k\le m$, consider faithfulness condition \eqref{faithfulness} on $G(\hat X(m)/Z)$: 
\begin{equation} \label{faithfulness} \text{$F_K$: if  $\sigma\in G_f$ fixes $k$ integers, then it $\sigma=1$.} \end{equation} 
 
\begin{princ} \label{galprinc}  There is an isomorphism $\hat f^*: G(\hat X(m)/Z)\to G_f$, uniquely defined up to inner isomorphism of $G_f$. Given another connected component $\hat X(m)'$, then $\hat X(m)\to Z$ and $\hat X(m)'\to Z$ are equivalent as covers by an isomorphism that induces an isomorphism of $$G(\hat X(m)/Z)\to G(\hat X(m)'/Z),$$ unique up to inner isomorphism of $G_f$. 

Suppose $\hat f: \hat X(m)\to Z$ factors through a cover $\psi: V\to W$. Then, $\hat f$ factors through 
$\hat \psi:\hat V\to W$,  the Galois closure cover of $\psi$. 

Suppose $G_f$ satisfies $F_k$. Then, the Galois closure cover of $f$ is a connected component of $\hat X(k)$.  \end{princ} 

\begin{proof}  Since $|S_m|$ is the degree of $\hat X(m) \to Z$,  the degree of $\hat f$ is the same as the order of the decomposition group. The basic Galois principle shows this gives the Galois closure cover. The restriction of $S_m$ to any fiber $\hat X_{m,z}$ over $z\in Z$ is transitive. Therefore it is transitive on components of $\hat X(m)$, and some element $\sigma\in S_m$ takes $\hat X(m)$ to $\hat X(m)'$. Conclude the result by tracing the effect on the isomorphisms with $G_f$. 

Now consider the morphism $\psi$. The morphism $\hat X(m)\to W$, factoring through $\psi$ is a Galois cover factoring through $V$, while $\hat \psi$ is the minimal Galois cover of $W$ factoring through $V$ by the Galois correspondence. 

Now assume $G_f$ satisfies $F_k$. Then, the action of $G_f$ on $\hat X(k)$ is faithful. Take $Y$ to be an orbit of this action. Consider the projection $\hat X_m\to \hat X_k$ onto the first $k$ coordinates.  Then, the pullback of $Y$ is the image of some connected component $\hat X(m)$ of $\hat X_m$, and the action of $G(\hat X(m)/Z)$ commutes with the projection. This identifies $Y\to Z$ as the Galois closure of $f$. 
\end{proof}

The whole inner vs. absolute total space construction of \eqref{innabsdiagram} can be done in the same fiber-product style based on assuming the self-normalizing condition of Def.~\ref{selfnorm}. The diagram is part of the \cite[Main Thm]{FrV91}; the fiber-product construction is in \cite[\S 3.1.1]{BaFr02}. 
 
\subsection{When the orbifold characteristic is nonnegative} \label{orbzero}  We give Nielsen class formulations of two cases when the $\ochar$ is nonnegative, with each appearing in examples in this paper's body. 

\subsubsection{Examples with $\ochar=0$} \label{ochar0} 
Pakovich notes cases where \eql{Pak}{Paka} fails, go back to the late 1800s. True, but  the condition needs more discussion than \cite[p.~2]{Pak18b} gives.  We give a Nielsen class construction of families of towers of genus 0 components that fail \eql{Pak}{Paka}. Those, for a given $f$,  give irreducible $\tilde \sC_{f^*,g^*\circ g_1}$ of genus 0, with $\deg(g_1)$ arbitrarily large.

Take $E$ a copy of the complexes $\bC$. A 1-dimensional complex torus has the form $E/L$ with $L$  isomorphic to $\bZ^2$ viewed as a rank 2 module of translations by complex numbers on $E$. For any integers $n>1$ form $L\cdot \frac 1 n\eqdef L_n$. An isogeny (group homomorphism) between  complex torii has the form 

\begin{edesc}\item  $\beta: E/L \to E/L'$: $L\le L'\le L_n$ for some integer $n$.  
\item \label{odddihedral} For simplicity assume  $|L'/L|$ is odd. \end{edesc}

To find a genus 0 cover $f: \prP^1_x\to \prP^1_z$ whose Galois closure has genus 1, note that $\beta$ commutes with modding out by  $\{\pm 1\}$ generated by multiplication by -1 on both $E/L$ and $E/L' $. From {\sl Weierstrass normal form}:  $$E/L'\!\!\!\mod \{\pm 1\}=\prP^1_z\text{ and   }E/L \!\!\!\mod \{\pm 1\}=\prP^1_x.$$
The induced cover $f: \prP^1_x \to \prP^1_z$ has  
\begin{equation} \label{isocover} G_f\equiv  (L'/L)\xs \{\pm 1\}. \end{equation} 

As $L'/L$ is an abelian group of rank (minimal number of generators) 1 or 2. This produces two families of covers given by rational functions: corresponding  to $L'/L=\bZ/n$ and $L'/L=(\bZ/n)^2$,  rank 1 and rank 2 (degree $n$) cases. Other covers from this method are cofinal in these \cite[\S3]{Fr74}.

\begin{exmpl}[$L'/L=\bZ/n$] \label{Z/n} The Nielsen class is $\ni(D_n,\bfC_{2^4}, T)$: $D_n$ the dihedral group of order $2n$, $\C$, the class of involutions, with $\bfC_{2^4}$ having $\C$ repeated 4 times; $T: D_n\to S_n$ with the action on $\lrang{\sigma}$ cosets, $\sigma\in \C$. 

For $(n,n')=1$, fiber products of covers, with the same branch points, in $\ni(D_n,\bfC_{2^4}, T)$ and $\ni(D_{n'},\bfC_{2^4}, T)$ are in $\ni(D_{n\cdot n'},\bfC_{2^4}, T)$.  
These are cases where the Galois Closure has genus 1. Indeed, this follows immediately from \eqref{orb} and it having four branch cycles of order 2. So, the orbifold characteristic is 0. 

This applies to condition \eql{Pak}{Paka} in Thm.~\ref{PakThm}. The rational function covers $f$ there have branch cycles that are coalescings (as in \S\ref{contextdeg7}) in the Nielsen class $\ni(D_n,\bfC_{2^4})^\abs$, and their Galois closures (if they aren't already Galois) are components $\hat f_2: \hat \prP^1_{x,2}: $ of the $k=2$-fold fiber product  of the cover $f$ minus the fat diagonal.  \end{exmpl} 

Especially the most important part is this. Over a given algebraic point $j_0$ of $\prP^1_j$, what the action of the absolute Galois group $G_{\bQ(j_0)}$ does to the components of $\hat \prP^1_{x,2}$. Suppose we add the rank 2 case in a similar style to this (with $G_n=(\bZ/n)^2\xs \{\pm1\}$). Then, this characterizes the covers that arise in three seemingly  disparate diophantine problems: Serre's Open Image Theorem, the theory of complex multiplication and deciding precisely when such $f$ represent {\sl exceptional\/} covers: Covers for which over their definition field $K$, for infinitely many primes $\bp$ of the ring of integers, $\sO_K$,  provide one-one maps on $\sO_K/\bp \cup\{\infty\}$. \cite[\S6.1]{Fr05} puts that whole story together, referring back to the relevant points of \cite[\S2]{Fr78} and \cite{GMS03}. 

In particular, only a handful of covers in the whole collection of  Nielsen classes with $G=D_n$ have definition field $\bQ$. In contrast, the $\bQ$ covers, in the case $G=S_n$, are dense in the space of covers because they correspond to elliptic curves over $\bQ$ with the degree $n^2$ isogenies given by multiplication by $n$. As \cite[\S3]{Fr74} notes, this is an enhancement of Ritt's Second Theorem (which has no allusion to number theory) and the very motivation for the generalization of that to \cite{Fr73b}. 

\begin{exmp}[$\ochar$-fans for these Nielsen classes] \label{C_2^4}  A natural Nielsen class, $\ni(\bZ\xs \bZ/2,\bfC_{2^4})$ covers every Nielsen class in this subsection. It is one Nielsen class, but it is special because to cover all these examples, the group $G$ is infinite. We call this the ${}_1\bfC_{2^4}$ $\ochar$-fan. 

An even larger $\ochar$-fan $\ni((\bZ)^2\xs \bZ/2,\bfC_{2^4})$ of genus 0, degree $n^2$, covers  gives 
covers of all the degree $n^2$ isogenies of elliptic curves $E\to E$ from multiplication by $n$ followed by modding out on both sides by $\lrang{\pm1}$. Refer to this as the  ${}_2\bfC_{2^4}$ $\ochar$-fan. 
\end{exmp}

\subsubsection{Examples with $\ochar> 0$} \label{ochargg0} There is a natural polynomial case, by coalescing branch cycles just as was done according to \S\ref{contextdeg7} to arrive at the genus 0 fiber product components of \S\ref{deg7}. The corresponding Nielsen classes are $\ni(D_n,\bfC_{2^2\cdot n})^\abs$ with $\bfC_{2^2\cdot n}$ indicating two repetitions of the class of involutions together with the class of an $n$-cycle in $D_n$. Modulo absolute (but not inner) equivalence there is just one $n$ cycle class, because the outer automorphism group of $D_n$ is $$\{\smatrix a b 0 1\mid a \in (\bZ/n)^*, b\in \bZ/n\}.$$ 

These are cases where the Galois Closure has genus 0. This follows immediately from \eqref{orb} and two branch cycles of order 2, one of order $n$. So, the orbifold characteristic is $>0$. 
The family of these is easy to understand as they are just the Chebychev polynomials of degree $n$ modulo changes of variable (as noted in \cite{Fr70}). These are cases where the Galois closure of the cover has genus 0. 

\begin{exmp}[$\ochar$-fan $\ni(\bZ\xs \bZ/2,\bfC_{2^2\cdot\infty})$] \label{ochar2^2infty} Start with the $\ochar$-fan that is the analog of the fans in Ex.~\ref{C_2^4}, for the Nielsen classes in this subsection given by the projective limit of the Nielsen classes $\ni(\bZ\xs \bZ/2,\bfC_{2^2\cdot m})^\abs$. Here the representation is on cosets of a group $G(T,1)$ generated by an involution. For $m$ odd, the involutions form a unique conjugacy class. 

Covers in this Nielsen class form a connected family parametrized by $$\sH(D_m,\bfC_{2^2\cdot m})\eqdef (\prP^1)^3\setminus \Delta_3/S_2\times \{1\}$$ where: $\Delta_3$ is the fat diagonal;  and $S_2\times \{1\}$ equivalences $(y_1,y_2,y_3)$ to $(y_2,y_1,y_3)$.  The cover over the equivalence class of $(y_1,y_2,y_3)$ is in \eqref{chebychev}.  

Define the polynomial  $T_m$ by $$T_m\Bigl(\frac{t+1/t}2\Bigr) = \frac{t^m\np 1/t^m}2=z.$$ As a covering map it is in $\ni(D_m,\bfC_{2^2\cdot m})$, with finite branch points $\pm1$. For $(3,m)=1$, the (normalized) fiber product $\tilde \sC_{T_3,T_m}$, as a cover of $\prP^1_z$,  is $T_{3m}: \prP^1_w\to \prP^1_z$, with $w=\frac{t+1/t}2$. Take $L$ to be the linear fractional transformation on $\prP^1$.Then, the fiber over $\ell\times \prP^1_y$ of the map 
\begin{equation} \label{chebychev} L\times \prP^1_w \to L\times \prP^1_y \text{ by } \ell\times w \mapsto \ell \times l(T_m(\ell^{-1}(w))) \end{equation} has branch points $\ell(-1), \ell(+1)$;  respective images of  $\ell(-1), \ell(+1)$. 

This is, however, a case with a cover in $\ni(\bZ\xs \bZ/2,\bfC_{2^2\cdot m})^\abs$ having its Galois closure of genus 0: \begin{equation} \label{galClose} 2(2m\nm 1\np \textbf{g}_{m,\inn})= 2(m\nm1)+ 2(2m/2) \implies \textbf{g}_{m,\inn}=0.\end{equation}  \end{exmp} 
 
\begin{rem} \label{evendihedral} Take $n=4$, where \eqref{odddihedral} does not hold. Here $$G=D_4=(\bZ/4)\xs \{\pm 1\}=\bigl\{\smatrix a b 0 1\mid a\in \{\pm 1\}, b\in \bZ/4\bigr\},$$  has 2 classes of involutions $\C_1$ and $\C_2$, resp.~represented by $(a,b)=(-1,0)$ and $(-1,1)$. The Nielsen class analogous to that for odd $n$ is $\bfC_{1^22^2}$ indicating we take each involution class twice. Then, we get a 2-component fiber product $\tilde \sC_{f,g}$ by taking the respective pairs of permutation representations $T_f$ and $T_g$ as the resp.~representations on the cosets of $\lrang{(-1,0)}$ and $\lrang{(-1,1)}$.  

There are two anomalies. First, $\tilde\sC_{f,g}$ here fails {\sl both\/} of Pakovich's hypotheses \eqref{Pak}. The genus of the Galois closure of $f$ is 1, and we get a reducible fiber product. Since the involutions fix two integers in the representation, the 2-fold fiber product of $f$ will not give the Galois closure (as in Princ.~\ref{galprinc}). 

Consider the polynomial version by coalescing in this case, where the Nielsen class is $\ni(D_4,\bfC_{2^2\cdot 4})^\abs$, with the repetition of an involution class twice, and a 4-cycle. The orbifold characteristic is $2 \np 2(\frac1 2\nm 1)\np (\frac 1 4\nm1)> 0$. Normalize the polynomials here to have branch points $1, -1, \infty$, and this is the case of reducible degree 4, $T_4(x)+T_4(y)$ , or $g=-f$. 
\end{rem} 

\begin{exmp}[Ubiquitous $T_4$ example] Rem.~\ref{evendihedral} gives group data about this example. Despite the rareness of nontrivially reducible $f(x)-g(y)$ with $f,g\in \bC[x]$, this example pops up in {\sl many\/} papers: $f(x)=T_4(x)$, $g=-T_4(y)$. Affine pieces of the components of $\tilde \sC_{f,g}$ appear in \cite[p.~57]{Sc82}: 
$$\bigl(x^2+\sqrt{2}xy+y^2-2\bigr)\bigl(x^2-\sqrt{2}xy+y^2-2\bigr).$$ Each component is the Galois closure of $X_f$ with corresponding groups both $\{1\}$ and  $(x,y)\mapsto (-x,y)$ displays the two components are isomorphic. 
\end{exmp} 

\section{Expectations for $\tilde \sC_{f^*,g^*}$ genus 0 components} \label{spaces} We use the notation of Prop.~\ref{redRed}.  \S\ref{doubledegree} constructs $\infty$-ly many cases -- thanks to \RET\ -- where $\tilde \sC_{f^*,g^*}$ has $u=2$ components, with both $f^*$ and $g^*$ indecomposable, and having identical Galois closures. It also shows that one of those components has genus 0. Yet, \S\ref{appSpaces} gives  evidence that those are subtle exceptions to Genus zero problem (\S\ref{contextdeg7}) expectations.

\subsection{$\infty$-ly many  $\tilde \sC_{f^*,g^*}$ with $u=2$ components}  \label{doubledegree} For each $m \ge 4$, we produce $m$ and $(f^*,g^*)$, with $f^*\in \bC[x]$  of $\deg(f^*)=m$,  and (nontrivially) $\widetilde \sC_{f^*,g^*}$ has two components. Use the notation at the beginning of \S\ref{charred}. 
Here is a branch cycle description for $f^*: \prP^1_x\to \prP^1_z$, with $G_{f^*}=S_m$ and $T_{f^*}=T_1$ the standard representation of $S_m$: 
$$\psigma_f\eqdef(\sigma_1,\sigma_2,\sigma_3)=((1\,2) , (1\,3\,4\,\dots\, m), (1\,2\,\dots\,m)^{-1})\in  (S_m)^3). $$ 
Take   $T_{g^*}=T_2$  the degree $\frac{m(m -1)} 2 = n$ rep.~of $S_m$ on  unordered pairs $$\{\{i,j \} \mid  i\not= j, 1 \le  i,j \le m\}.$$

The cover for $g^*$ has the branch cycles $\ptau=(T_2(\sigma_1), T_2(\sigma_2), T_2(\sigma_3))$. A rational function represents the cover $g^*$  if   $\geng_*$ in the following equation is 0.  
\begin{equation} \label{genusg}  2\Bigl({\frac{m(m -1)} 2} + \geng_* -1\Bigr)=\sum_{i=1}^3 \ind(T_g(\sigma_i))=\sum_{i=1}^3 \ind(\tau_i).\end{equation}  

\begin{prop} \label{Snex} Computing indices in \eqref{genusg} shows $\geng_*=0$, producing $g\in \bC(y)$. Then, $\tilde \sC_{f^*,g^*}$ has exactly two components ($g^*\in \sR_f$, \eql{Pak}{Pakb} fails). For $m>4$, the orbifold characteristic is negative (\eql{Pak}{Paka} holds). 
\end{prop} 

\begin{proof} The action of $(1\,2)$ on the pairs $\{1,j\}$, $j\ne 1$ or 2, moves them to the pairs $\{2,j\}$. Conclude:  $\ind(\tau_1)= m -2$. Also, every $\tau_3$ orbit has  the form  
$$ \{\{j,k\np j\nm 1\} \mid 1\le j \le m\}, \ 2\le k\le m.$$ This will be an orbit of length $m$, unless for a given $k$, there are two distinct values of $j$ (say, $j' $ and $j''$) for which $$j'=j''\np k\nm1\text{ and }j''=j'\np k\nm 1, \text{ or } 2(k\nm1)\equiv 0 \!\!\!\mod m. $$ That is, $m$ is even and $k=\frac {m}{2}\np1$. Also, the orbit under translation of this set of two distinct elements is determined by the minimal absolute difference between the elements in the set. 

 Therefore, $\tau_3$ is a product of $m$-cycles, $\frac{m -1}2$ of them if $m$ is odd, but it has, besides $m$-cycles, one $\frac{m}{2}$-cycle if $m$ is even. Thus: $$\ind (\tau_3) =\begin{cases} \frac{(m-1)(m-1)}2&\text{ if $m$ is odd; and }\\ 
\frac{(m -1)(m -2)}2 + \frac{m}2 -1 = \frac{m(m -2)}2&\text{ if $m$ is even}.\end{cases}$$ Now we compute $\ind(\tau_2)$ based using  that $\geng_*\ge 0$.  

Case $m$ is even: From \RH, $$\sum_{i=1}^3 \ind(\tau_i)\ge 2\Bigl({\frac{m(m -1)} 2} -1\Bigr)= m(m\nm1)\nm 2. $$
Therefore, $\ind(\tau_2) \ge m(m\nm1) \nm 2\nm \frac{m(m\nm2)}2 \nm (m \nm2) = \frac{m(m\nm2)}2$.  Since $\tau_2$ has order $m-1$, the maximal value of $\ind(\tau_2)$ occurs if $\tau_2$ is a product only of $(m\nm 1)$-cycles, $\frac m 2$ of them. So, $\ind(\tau_2)$ is this maximal value and $\geng_*=0$. 
 
Case $m$ is odd: Same as in the Case $m$  is even, except that the maximal possible value of  $\tau_2$  if  besides $(m\nm 1)$-cycles,  $\tau_2$ has one $(\frac{m -1}2)$-cycle. 

Notice: $\tilde \sC_{f^*,g^*}$ has two components corresponding to the two orbits of $G_f(1)$ on the unordered pairs: One orbit on all pairs of the form $\{1,k\}$, $k\ne 1$; the other on the rest of the pairs. 

This concludes the proof except to show that for $m>4$, the orbifold characteristic is negative. Since it is given by $$2 \np (1/2-1)\np (1/m \nm 1)\np (1/(m\nm1) \nm 1),$$ a decreasing function of $m$. Check that it is -.05 for $m=5$. 
\end{proof} 

Cor.~\ref{Sngenus0} computes the genus of the two components discovered in the fiber products of Prop.~\ref{Snex}. 

\begin{cor} \label{Sngenus0} Let $m\ge 5$ in Prop.~\ref{Snex} be odd. For $m=5$, both components of $\tilde \sC_{f,g}$ have genus 0. 

For all $m\ge 5$, the component of degree $m\nm1$ over $\prP^1_z$ has genus 0. The genus of the other component grows quadratically with $m$. \end{cor} 

\begin{proof} We do the case where $m$ is odd in detail. 
Denote representatives of the Nielsen class pairs $(f^*,g^*)$ corresponding to $m$ by $(f^*_m,g^*_m)$. 
In applying Cor.~\ref{methodII},  switch $f_m^*$ and $g_m^*$; branch cycles for $f^*$ are simpler. 

That is, we are dividing the points corresponding to cycles in the  branch cycles of $\tilde \sC_{f^*_m,g^*_m}$ over $\prP^1_x$ between the two components. Start with $m=5$. The method is  completely analogous to that of \S\ref{deg7}. Since $\sigma_3$ is an $m$-cycle, the corresponding point contributes nothing to \RH\ in either component. As $\sigma_2$ fixes only one point -- corresponding to the integer 2 -- and the cycles of $\tau_2$ all have order dividing $m\nm1$, we can take  \begin{equation} \label{sigma2} \begin{array}{rl} &((2,\{2,1\})\,(2, \{2,3\})\,(2, \{2,4\})\,(2,\{2,5\}))^* \\ &((2,\{1,3\})\,(2,\{3,4\})\,(2,\{4,5\})\,(2,\{5,1\})) \\ &((2,\{3,5\})\,(2,\{1,4\})) \end{array} \end{equation} as a rep.~of the branch cycle conjugacy class in $G_{f^*}=G_{g^*}$ of $\tau_2$. The ${}^*$ superscript is used below. 

Similarly, the branch cycle classes for the points of $\prP^1_x$ corresponding to each of the three fixed points -- $w=3, 4, 5$ -- of $\sigma_1$ are represented by 
\begin{equation} \label{sigma1} \!\! ((w,\{1,3\})\,(w,\{2,3\}))((w,\{1,4\})\,(w,\{2,4\}))((w,\{1,5\})\,(w,\{2,5\})).\end{equation} 

We want conjugacy classes all in $G(T_1,1)$ stabilizing 1 in the representation $T_1$. Conveniently $\sigma_3$ is an $m$-cycle. As in \S\ref{deg7}, translate subscripts uniformly to get in $G(T_1,1)$. Translate by -1 in \eqref{sigma2}, and respectively by -2, +2, +1 in the expressions \eqref{sigma1} corresponding to $w=3,4,5$. Now we can drop the notation indicating 1 is fixed, and thereby recognize the two components correspond to the orbits of $G(T_1,1)$ on 
\begin{equation} \label{sigmasets} \begin{array}{rl} \bar S_1=&\{\{1,2\},\{1,3\},\{1,4\},\{1,5\}\} \\
\bar S_2=& \{\{2,3\},\{2,4\},\{2,5\},\{3,4\}, \{3,5\},\{4,5\}\} .
\end{array} \end{equation} We find one 4-cycle  (denoted by ${}^*$ in \eqref{sigma2})  and one 2-tuple   for each $w$ whose support is in $\bar S_1$  in, resp., \eqref{sigma2} and \eqref{sigma1}. Therefore the genus $\geng_{5,1}$ for the component corresponding to $\bar S_1$, is given by 
$$2(4\np \geng_{5,1} \nm1) = 3+3\cdot 1=6, \text{ or }  \geng_{5,1}=0.$$
Similarly, for these cycles remaining from \eqref{sigma2} and \eqref{sigma1}, the genus $\geng_{5,2}$ for the component corresponding to $\bar S_2$, is given by 
$$2(6\np \geng_{5,2} \nm1) = 3+1+3\cdot 2=10, \text{ or }  \geng_{5,2}=0.$$

Induct on odd $m$. Assume, for $m'=2(\ell\nm1)\np 1$, we know the correct result and notation. Then,  tack on $\{1,2\ell\},\{1,2\ell\np 1\}$ to  $\bar S_1$ in \eqref{sigmasets}, and $$\{w,k\}, k=2\ell, 2\ell\np1, w=2,\dots,2\ell\nm1, \text{ and }\{2\ell,2\ell\np1\} \text{ to $\bar S_2$}.$$ 
Then, the ${}^*$ term in \eqref{sigma2} becomes the $(m\nm1$)-cycle  by adding $(2,\{2, 2\ell\})$ and $(2,\{2, 2\ell\np 1\})$ to the end. To  \eqref{sigma1}  add $$((w,\{1, 2\ell\})\, (w,\{2, 2\ell\}))\text{ and } ((w,\{1, 2\ell\np 1\})\, (w,\{2, 2\ell\np 1\})) $$ for $w=3,\dots,2\ell\np1$ and for $w=2\ell,2\ell\np1$, add $$((w,\{1, k\})\, (w,\{2, k\})),\ k=3,\dots,2\ell\nm1.$$ 

Then, by writing $w=1\np k$, translate all terms with a given value of $w$ by $-k$, to conclude that only one 2-cycle -- corresponding to the given value of $w$ (again drop the $w$ slot) now occupied by 1 -- will end up with support in $\bar S_1$: $(\{1\nm k, u\nm k\}\,\{2\nm k, u\nm k\})$ where $u\nm k=1$. 

Now compute the genus of the degree $m\nm 1$ component as for  $m=5$: 
$$2((m\nm1)\np \geng_{m,1} \nm1) = m\nm2+(m\nm2)\cdot 1=2(m\nm2), \text{ or }  \geng_{m,1}=0.$$ Finally, do the same for the degree $\frac{(m\nm1)(m\nm2)} 2$ component by adding indices of all the cycles not covered by the degree $m\nm1$ component: 
$$\begin{array}{rl} &2(\frac{(m\nm1)(m\nm2)} 2 + \geng_{m,2} \nm 1) = m(m\nm3) \np  2\geng_{m,2} =  \\ &(m\nm 4)(m\nm 2) \np (\frac{m\nm1}2 \nm 1) \np (m\nm 3)(m \nm2).\end{array}$$
From the leading terms of the expression,  $\textbf{g}_{m,2}/m^2$ has limit $\frac 1 2$. 

That completes the case for odd $m$. For even $m$, the only serious adjustment comes in the case of arranging $\tau_2$, which now has $m/2$ disjoint cycles of length $m\nm1$. The cycles, however, that appear in the degree $m\nm1$ component are entirely analogous. Thus, the genus computation has the same terms on the right side of \RH as a function of $m$. 
\end{proof}  

\begin{rem}[Values for $(f^*,g^*)$ in Cor.~\ref[Sngenus0]] \label{Snvalues} In the case of Davenport pairs, the $(f^*,g^*)$ produced by them over a field $K$ have the same values $\mod$ any residue class field of the ring of integers of $K$. In the proof of Prop.~\ref{Snex}, we considered the index of $\tau_2=T_2(\mu)$ and found that for $m$ even (resp.~odd) in its action on the unordered distinct pairs $\{1\le \ne j\le m\}$ it consisted of distinct $m\nm1$-cycles (resp.~one $\frac{m\nm1}2$-cycle and the rest $m-1$-cycles). In either case it fixes no letter of the representation, though $T_1(\mu)=\sigma_2$ is an $m\nm1$-cycle and it fixes an integer.  The Chebotarev density theorem implies that a positive density of primes $f^*$ and $g^*$ will have different ranges.  
\end{rem} 

\begin{rem}[Summary of the Davenport case] \label{davcase} \cite[Lem.~2]{Fr73a}: Using that $T_1$ is a doubly transitive representation for $f^*$ an indecomposable polynomial, this says that $\tr(T_1(\sigma))>0$ if and only if $\tr(T_2(\sigma))>0$, then the representations are equivalent. By the Chebotarev theorem, the images' multiplicity is the same for both polynomials. The major work that went into the rest of Davenport's problem was in two parts:
\begin{edesc} \label{davres} \item \label{davresa}  The above implies these pairs are conjugate over $\bQ$, so Davenport's problem was affirmative over $\bQ$.
\item \label{davresb}  But over other numbers fields, there were only finitely many degrees giving Davenport pairs. For each of the possible degrees in Thm.~\ref{polyPakThm}, we gave the fields over which they occurred. 
\end{edesc} The explicitness of the  result \eql{davres}{davresb} and the way it used the classification motivated the full set of results in the genus 0 problem. \end{rem} 

\subsection{Series of genus 0 components of $\tilde \sC_{f,g}$} \label{appSpaces}  What is the context for the infinite series of examples  that appears in \S\ref{doubledegree} for which genus 0 components appear on $\tilde \sC_{f^*,.g^*}$?  This section gives an historical example that suggests that such series arising from alternating group-related series may be limited by a further constraint. 

\cite[\S7.1]{Fr12} goes through many of the applications that fostered classifying separated equations $f (x) - g(y) = 0$, with  $f,g\in \bQ[x] $ that have infinitely many quasi-integral points. As previously, a direct statement was to generalize what we refer to as Ritt's Theorem 2 as in \S\ref{thelit}. That  used Siegel's famous result on quasi-integral points on affine curves over a number field. The more general {\sl Hilbert-Siegel\/} problems were expanding on that. 1st Hilbert-Siegel Problem is Prop.~\ref{1stHS}. 

\begin{prop} \label{1stHS}  Assumptions:

\begin{edesc} \label{quasi-int} \item $f \in \bQ[x]$ is indecomposable. 
\item  {\sl Quasi-integral reducibility}: $\exists$ $a\in \bZ$ and $\infty$-ly many $y_0\in \bZ[1/a]$ $$\text{ for which $f(x)-y_0$ is reducible, but it has no zero in $\bQ$.}$$ \end{edesc} 

Conclusion:  All but  finitely many such $y_0$ are in the values of $g \in \bQ(x)$ where $\tilde \sC_{f,g}$ has $u\ge 2$ components. There are two cases.   
\begin{edesc} \label{appfac} \item {\sl Polynomial}: either $g \in \bQ[x]$;  or 
\item {\sl Double-degree}: with $\deg(f) = m$, $\deg(g) = 2m$  and a branch cycle $\sigma_\infty$ for $ g$ over $\infty$ has the shape $(m)(m)$.\end{edesc} \end{prop} 

\begin{cor} That no $g$ works in the polynomial case of \eqref{appfac} comes from the branch cycle  argument alluded to in  the title of \cite{Fr12}.\footnote{As used in the solution of Davenport's problem in \cite{Fr74}.  If we didn't restrict to $\bQ$, we would have to include the other cases of \eql{genus0reduct}{genus0reductb}.}  

Conclusions of the double-degree case of \eqref{appfac}: 
\begin{edesc}  \label{m=5HS} \item \label{m=5HSa} Only $m=5$ gives nontrivial examples. 
\item  \label{m=5HSb} There are several Nielsen classes, but -- as in the examples of  \S\ref{irreducibilityb} -- all are coalescings of the main Nielsen class: $G=S_5$, $r=4$,$\C_2=\C_3$  is the class of 2-cycles, $\C_1$ is the class of products of two disjoint 2-cycles and $\C_4$ that of 5-cycles. 
\item \cite[Thm.~1.2]{DeFr99} concludes: Using the natural moduli space, for any fractional ideal of $\bZ$, the set of such $f$ produces solutions dense in the points of the Hurwitz space.\end{edesc} \end{cor}

This example surprised the authors of \cite[Thm.~1.1 and 1.2]{DeFr99} for two reasons: Only one value of $m$ produced these. Yet,  that one did in great abundance. \cite[Exp.~6.3]{Fr99} has Guralnick's conjecture (\cite[\S7.1.4]{Fr12}; what monodromy groups can arise often and the precise version of \eqref{genus0mon} for what would be the exceptional genus 0 monodromy (over $\bC$). Now it is a theorem. In these lists, you see several related to $A_n$, including the permutation representation of the cover acting on distinct, unordered pairs of integers. That is the case above.

Thus, even Guralnick's strong formulation of the genus 0 problem is insufficient. \cite[\S3]{Fr86} has the   extra group theory for showing \eqref{quasi-int} holds for the groups in Guralnick's list for the Hilbert-Siegel problem when the double-degree condition of \eqref{appfac} holds.   

\cite{Mu96} has results for more general versions of Hilbert Siegel Problems as in \cite[\S4]{Fr86}. He still relied on \cite[\S2 and \S3]{Fr86}.

\end{appendix} 

\providecommand{\bysame}{\leavevmode\hbox to3em{\hrulefill}\thinspace}
\providecommand{\MR}{\relax\ifhmode\unskip\space\fi MR }
% \MRhref is called by the amsart/book/proc definition of \MR.
\providecommand{\MRhref}[2]{%
   \href{http://www.ams.org/mathscinet-getitem?mr=#1}{#2}
}
\providecommand{\href}[2]{#2}

%%%%%%%%%%%%%%%%%%% END: BIBLIOGRAPHY %%%%%%%%%%%%%%%%%%%%
%%%%%%

\end{document}